\documentclass[11pt]{article}
\usepackage[english]{babel}
\usepackage[T1]{fontenc}
\usepackage[utf8]{inputenc}
\usepackage{amsbsy}
\usepackage{amsmath}
\usepackage{amssymb}
\usepackage{amsfonts}
\usepackage{amsthm}
\usepackage{bm}
\usepackage{graphicx}
\usepackage[active]{srcltx}
\usepackage{upgreek}
\usepackage{xcolor}

\newcommand{\bd}{\bm{d}}

\newcommand{\C}{\mathbb{C}}

\newcommand{\N}{\mathbb{N}}

\newcommand{\R}{\mathbb{R}}

\newcommand{\Z}{\mathbb{Z}}

\newcommand{\boB}{\mathcal{B}}
\newcommand{\boC}{\mathcal{C}}

\newcommand{\boE}{\mathcal{E}}
\newcommand{\boF}{\mathcal{F}}

\newcommand{\boH}{\mathcal{H}}
\newcommand{\boI}{\mathcal{I}}
\newcommand{\boJ}{\mathcal{J}}

\newcommand{\boL}{\mathcal{L}}
\newcommand{\boM}{\mathcal{M}}
\newcommand{\boN}{\mathcal{N}}
\newcommand{\boO}{\mathcal{O}}

\newcommand{\boQ}{\mathcal{Q}}

\newcommand{\boU}{\mathcal{U}}

\newcommand{\boW}{\mathcal{W}}
\newcommand{\boX}{\mathcal{X}}
\newcommand{\boY}{\mathcal{Y}}
\newcommand{\boZ}{\mathcal{Z}}

\newcommand{\goE}{\mathfrak{E}}
\newcommand{\goH}{H}

\newcommand{\goL}{\mathfrak{L}}
\newcommand{\goN}{\mathfrak{N}}

\newtheorem{corollary}{Corollary}
\newtheorem{lemma}[corollary]{Lemma}
\newtheorem{proposition}[corollary]{Proposition}
\newtheorem{theorem}{Theorem}
\newtheorem{step}{Step}
\newtheorem*{theorem*}{Theorem}
\theoremstyle{definition}
\newtheorem*{merci}{Acknowledgments}

\newtheorem*{remarks}{Remarks}
\theoremstyle{remark}

\begin{document}

\title{Bifurcating solitonic vortices in a strip}
\author{
\renewcommand{\thefootnote}{\arabic{footnote}}
Amandine Aftalion\footnotemark[1],~Philippe Gravejat\footnotemark[2]~ and \'Etienne Sandier\footnotemark[3]}
\footnotetext[1]{Laboratoire de Math\'ematiques d’Orsay, CNRS-UMR 8628, Universit\'e Paris Saclay, 91405 Orsay, France. E-mail: {\tt amandine.aftalion@math.cnrs.fr}}
\footnotetext[2]{CY Cergy Paris Universit\'e, Laboratoire Analyse, G\'eom\'etrie, Mod\'elisation (UMR CNRS 8088), F-95302 Cergy-Pontoise, France. E-mail: {\tt philippe.gravejat@cyu.fr}}
\footnotetext[3]{LAMA, Universit\'e Paris Est Cr\'eteil, Universit\'e Gustave Eiffel, UPEM, CNRS, F-94010 Créteil,
France. E-mail: {\tt sandier@u-pec.fr}}
\maketitle

\begin{abstract}
The specific geometry of a strip provides connections between solitons and solitonic vortices, which are vortices with a solitonic behaviour in the infinite direction of the strip. We show that there exist stationary solutions to the Gross-Pitaevskii equation with $k$ vortices on a transverse line, which bifurcate from the soliton solution as the width of the strip is increased. After decomposing into Fourier series with respect to the transverse variable, the construction of these solitonic vortices is achieved by relying on a careful analysis of the linearized operator around the soliton solution: we apply a fixed point argument to solve the equation in the directions orthogonal to the kernel of the linearized operator, and then handle the direction corresponding to the kernel by an inverse function theorem.
\end{abstract}

\section{Introduction}

Our manuscript is devoted to the analysis of the Gross-Pitaevskii equation
$$
i \partial_t \Psi + \Delta \Psi + \Psi \big( 1 - |\Psi|^2 \big) = 0,
$$
in an infinite strip $\R \times (0, d)$ of width $d > 0$, with Neumann boundary conditions
\begin{equation}
\label{eq:cond-Neumann}
\partial_n \Psi = 0 \text{ on } \R \times \{ 0, d \}.
\end{equation}
We focus on solutions of the following stationary Gross-Pitaevskii equations, also called Ginzburg-Landau equation in the mathematical literature:
\begin{equation}
\label{eq:GL}
\tag{GP}
\Delta \Psi +\Psi \big( 1 - |\Psi|^2 \big) = 0.
\end{equation}
A specific solution is the one-variable black soliton $S_0$ which tends to $\pm 1$ as $x$ tends to $\pm\infty$. It is given by the explicit formula
$$
S_0(x) = \tanh \Big( \frac{x}{\sqrt{2}} \Big).
$$
It is known that for $d$ small, this soliton is the unique stationary solution~\cite{AftaSan1}. For $d$ large, it has been proved in~\cite{AftaSan1} that the ground state of the energy under the condition that there is a zero at $x=y=0$ is a solitonic vortex, that is a solution with a zero at the origin but which looks like the soliton in the infinite direction: a solitonic vortex does not have an algebraic decay but an exponential decay at infinity.

The general mathematical pattern of solutions as $d$ is increased is still not clear though it has been the focus of many experimental and numerical papers. Our aim is to characterize solutions as $d$ is increased. More precisely, we are going to show that, as the width $d$ of the strip is increased, there exist stationary solutions close to the black soliton, but with $k$ vortices on the line $x = 0$, $k$ depending on the width of the strip.

\subsection{Physical motivation}

Black solitons are observed in systems that combine dispersion with a defocusing or repulsive interaction. Solitonic structures arise in many physical systems such as surface water waves, nematic liquid crystals, mechanical lattices of coupled pendula, electrical transmission lines, nonlinear Kerr media and more recently in atomic Bose–Einstein condensates (BECs)~\cite{DauxPey0}. There are different ways to create solitonic states in BECs: either by dragging a laser beam through a BEC, or phase-imprinting, or matter-wave interference~\cite{Frantze1}. Recent experiments focused on exploring two-dimensional (2D) and three-dimensional (3D) solitary waves in detail. The intriguing resulting structures, including their breakup into vortical patterns in both bosonic and Fermi gases, is the subject of wide investigation and has a wide variety of potential applications ranging from atomic matter-wave interferometers to producing two-level qubit systems.

In this paper, we are interested in the instability of solitons in reduced dimensions following recent experimental and numerical studies. Quite a few experimental groups have recently attempted to study solitons by imposing a phase shift in an elongated condensate for respectively bosonic atoms (rubidium~\cite{BBDESSS1} and sodium~\cite{DSTPDLF1, LaDoSDF1}) and for fermionic atoms (lithium~\cite{BuFoKRW1, KJMGCYZ1, YSKCJBZ1}). One of the issue was to observe solitons and analyze their decay or instability. From the first picture, they thought they had observed solitons~\cite{LaDoSDF1, YSKCJBZ1}. Further investigations were needed to fully understand the phenomena: in the case of lithium, they realized it was not a soliton but thought it was a vortex ring~\cite{BuFoKRW1}, until~\cite{KJMGCYZ1} argued that in fact it was a single straight vortex called solitonic vortex. In sodium, it was also confirmed it was not a soliton but a solitonic vortex~\cite{DSTPDLF1}. Therefore, the issue to determine the existence and stability of solitons in a strip is a main one.

Numerical simulations to study the stability of solitons rely on the time-dependent Gross-Pitaevskii equation to obtain solitary wave solutions for infinitely elongated 2D or 3D traps. According to the width of the trap, or the strength of interactions, the soliton is either stable or exhibits bifurcation patterns. The full mechanism of decay of solitons at the onset of instability is still not completely clear. In a series of studies, Komineas et al.~\cite{Kominea1, KopaPap2, KopaPap3} analyze how the soliton destabilizes according to the strip width. For narrow traps, the soliton is stable. In some intermediate cases, the soliton initially deforms to become a pair of vortex-antivortex in 2D or a vortex ring in 3D. This structure eventually decays into a stable solitonic vortex. The numerics reveal that the vortex-antivortex pair or vortex ring is unstable, but it is sufficiently long lived to be observed both in the numerics and the experiments. As the transverse size of the trap is further increased, more pairs of vortex-antivortex solutions are exhibited.

In~\cite{BranRei1}, they analyze numerically the linear stability of solitons and find that the soliton instability is associated with the formation of one, two, and three vortices in the regimes where one, two, and three imaginary eigenvalues are present for the linearized operator, as the transverse size of the trap increases. It is this pattern of one, two, three, etc vortices that we want to investigate in this paper.

Our aim is to analyze the eigenvalues of the linearized operator to better understand the type of bifurcation to one or several vortices and prove the existence of such solutions close to the soliton according to the width of the strip.

Similar questions arise in the simulations for solitary waves, that is existence of solutions with several vortices close to the soliton, but moving at velocity $c$~\cite{Kominea1, KopaPap3}. The mathematical treatment of such issues seems to be more involved even though~\cite{deLgrSm2, deLGrSm1} have settled the framework for the construction of minimizing solitary waves.

From a mathematical point of view, the instability of solitons is a key question. Rousset and Tzvetkov~\cite{RousTzv3} have proved for instance the instability of solitons in the whole space but nevertheless nothing is known about the mode of destabilization: whether it turns into a single vortex, a pair of vortices in dimension two or a vortex ring in higher dimension. Here the specificity of the geometry of the strip and the existence of solitonic vortices leads to new and different mathematical issues.

\subsection{Main result}

Our aim is to solve the stationary Gross-Pitaevskii equation~\eqref{eq:GL} with Neumann boundary conditions~\eqref{eq:cond-Neumann} for solutions having the symmetry properties
\begin{equation}
\label{eq:sym}
\psi(- x, y) = - \overline{\psi(x, y)}.
\end{equation} This property corresponds to the physical technique leading to the production of solitons, namely phase imprinting: a phase shift of $\pi$ around the $x=0$ axis is imposed, leading to this property.

Our starting point (see e.g.~\cite{diMeGal1}) is to observe that the linearized operator of the one-variable stationary Gross-Pitaevskii equation around the soliton $S_0$ has a unique negative eigenvalue $- 1/2$, whose eigenspace is spanned by the function
\begin{equation}
\label{eq:chi-0}
\chi_0(x) = \frac{1}{\cosh \big( \frac{x}{\sqrt{2}} \big)}.
\end{equation}
Moreover, the kernel of this operator is spanned by the geometric invariances of the equation, namely translations and constant phase shifts.

If we consider this linearized operator in the two-dimensional setting, things are different. As we will show below, when
\begin{equation}
\label{def:d-k}
d = d_k := \sqrt{2} \pi k,
\end{equation}
the restriction of this operator to the Fourier sector of order $k$ in the transverse variable $y$ has a nontrivial kernel (excluding the invariances of the equation), which is spanned by the function
\begin{equation}
\label{def:chi-k}
\chi_k(x, y) = i \, \chi_0(x) \, \cos \Big( \frac{\pi k y}{d} \Big).
\end{equation}
We point out that the soliton $S_0$ is odd with respect to the variable $x$, whereas the eigenfunction $\chi_0$ is even, so that $S_0$ and $\chi_k$ have the symmetry in~\eqref{eq:sym}. For any $k \geq 1$, this suggests the possibility of a branch of stationary solutions bifurcating from $S_0$ as the parameter $d$ varies and hits the value $d_k$. This is what we are going to prove.

The Hamiltonian framework corresponding to~\eqref{eq:GL} is the Ginzburg-Landau energy
$$
E(\psi) := \frac{1}{2} \int_\R \int_{(0, d)} |\nabla \psi|^2 + \frac{1}{4} \int_\R \int_{(0, d)} \big( 1 - |\psi|^2 \big)^2,
$$
which we will always assume to be finite in sequel.
It is then natural to introduce the function space
$$
\boX = \Big\{ \psi : \R \times (0, d) \to \C \text{ s.t. } \psi \text{ satisfies~\eqref{eq:sym}}, \nabla \psi \in L^2(\R \times (0, d)) \text{ and } 1 - |\psi|^2 \in L^2(\R \times (0, d)) \Big\}.
$$
For $\sigma > 0$, we also let
$$
L_\sigma^\infty := \Big\{ \psi : \R \times (0, d) \to \C \text{ s.t. } (x, y) \mapsto \psi(x, y) \, e^{\sigma |x|} \in L^\infty(\R \times (0, d)) \Big\},
$$
and we denote by $W_\sigma^{2, \infty}$ the space of bounded functions (not necessarily decaying exponentially in the $x$ variable), whose derivatives and second derivatives belong to $L_\sigma^\infty$. With these definitions at hand, we can state our main result.

\begin{theorem}
\label{thm:main}
For $k \in \N^*$, let $d_k = \sqrt{2} \pi k$. Then, there exist positive numbers $\ell_k$ and $\sigma_k$, and a smooth branch $d \mapsto \Psi_{k, d}$ from $(d_k, d_k + \ell_k)$ to $\boX \cap W_{\sigma_k}^{2, \infty}$ such that the functions $\Psi_{k, d}$ are solutions to~\eqref{eq:GL}, with Neumann boundary conditions~\eqref{eq:cond-Neumann}.

For $|d - d_k| < \ell_k$, the functions $S_0$, $\Psi_{k, d}$ and $\overline\Psi_{k, d}$ are the only solutions in a neighbourhood of $S_0$ in $\boX \cap W_{\sigma_k}^{2,\infty}$. In particular, the black soliton $S_0$ is an isolated solution for $d_k - \ell_k < d < d_k$.

As $d \searrow d_k$, the following expansion holds in $W^{2, \infty}(\R \times (0, d))$:
$$
\Psi_{k, d}(x, y) = S_0(x) + i \, \Lambda \, \sqrt{\frac{d -d_k}{d_k}} \, \chi_0(x) \, \cos \Big( \frac{\pi k y}{d} \Big) + \boO \Big( d - d_k \Big),
$$
for a universal positive number $\Lambda$. In particular there exists a universal positive number $\boE$ such that the energy of $\Psi_{k, d}$ can be expanded as:
\begin{equation}
\label{eq:DL-E}
E(\Psi_{k, d}) = E(S_0) - \frac{(d - d_k)^2}{d_k} \, \boE + \boO \Big( (d - d_k)^\frac{5}{2} \Big).
\end{equation}
\end{theorem}

It turns out that the symmetries of the equation and the uniqueness statement in this theorem allow us to describe more precisely the dependence on $k$ and the vortex structure of $\Psi_{k, d}$.

For any function $\Psi : \R \times (0, d) \to \C$, we define the map $R \Psi$, which is the conjugate reflection with respect to the line $y=d/2$:
\begin{equation}
\label{def:R}
R \Psi(x, y) = \bar{\Psi}(x, d - y), \quad \forall (x, y) \in \R \times (0, d).
\end{equation}
If $\Psi$ is a solution to~\eqref{eq:GL}, with the Neumann boundary conditions in~\eqref{eq:cond-Neumann}, then so is the function $R \Psi$. If, moreover, we have $R \Psi = \Psi$, then we may extend $\Psi$ to a solution $\Upsilon$ of~\eqref{eq:GL} on $\R^2$ by setting, for any integer $j \in \Z$,
\begin{equation}
\label{def:tilde}
\Upsilon(x, y + j d) = (R^j \Psi)(x, y), \quad \forall (x, y) \in \R \times (0, d).
\end{equation}
Indeed, the symmetry condition given by the identity $R \Psi = \Psi$ together with the Neumann boundary conditions in~\eqref{eq:cond-Neumann} imply that the values of $\Upsilon$ and its derivatives match on the boundary of the strips $\R \times (j d, (j + 1) d)$.

We are then able to deduce from the above and Theorem~\ref{thm:main} the following

\begin{corollary}
\label{cor:main}
With the same notations as in Theorem~\ref{thm:main}, for $k = 1$, the function $\Psi_{1, d}$ has a single zero, which is a vortex of degree $- 1$ located at $(0, d/2)$. Moreover, $\Psi_{1, d}$ is symmetric in the sense that
$$
R \Psi_{1, d} = \Psi_{1, d},
$$
when $d \in (d_1, d_1 + \ell_1)$.

Similarly, for $d \in (d_k, d_k + \ell_k)$, where $d_k = k \, d_1$, the function $\Psi_{k, d}$ has exactly $k$ vortices, located on the line $x = 0$ at the ordinates
$$
y_j = \frac{d_k (2 j + 1)}{2 k},
$$
for $0 \leq j < k$, and with degree $(- 1)^{j + 1}$. Moreover, when $d \in (d_k, d_k + \min \{ k \ell_1, \ell_k \})$, the solution $\Psi_{k, d}$ is equal to $\Psi_{1, d/k}$ in $\R \times (0 ,d/k)$ and is the restriction of the function $\Upsilon_{1, d/k}$ to the strip $\R \times (0 ,d)$.
\end{corollary}

\begin{remarks}
\begin{enumerate}
\item The solutions $\Psi_{k, d}$ have constant limits in the $x$ direction, and their differences with their respective limits decay exponentially. This behaviour is specific to the strip geometry, since in $\R^2$, the difference between a travelling wave and its limit at infinity decays algebraically~\cite{Graveja5, Graveja6}. This is a consequence of the fact that the strip geometry is closer in behaviour to $\R$ rather than $\R^2$, as should be expected.

\item The bifurcation profile and the vortex structure of bifurcating branches are fully consistent with the numerical simulations and experiments described above, especially in~\cite{BranRei1,Kominea1}, where the soliton destabilizes and turns into a solution with $k$ vortices, the number $k$ depending on the width $d$. The critical width for $k$ vortices is indeed roughly proportional to $k$ in~\cite{BranRei1}.

\item A subtlety of our results is that, while $\Upsilon_{1, d/k}$ exists for any $d \in (d_k, d_k + k \ell_1)$, our fixed point argument does not yield the uniqueness of the branch in this interval. Thus we have not proved that $\ell_k = k \ell_1$, even though we strongly suspect this is true.

\item It is known that if $d$ is small enough, then $S_0$ is the unique non constant stationary solution, up to the invariances of the equation~\cite{AftaSan1}. We conjecture that this uniqueness holds up to $d_1$ among solutions with Neumann boundary conditions and symmetries. This is substantiated by the above bifurcation analysis, which shows a local uniqueness result.

\item The bifurcation analysis implies that the bifurcating branch is more stable, at least at the linear level, than $S_0$. The linearized operator around $S_0$ has $k + 1$ negative eigenvalues for $d$ slightly larger than $d_k$, whereas the stationary solutions $\Psi_{k, d}$ are expected to have only $k$ negative eigenvalues. Note also that the energy of $\Psi_{k, d}$ is strictly less that of the soliton $S_0$. In particular this suggests that the bifurcating branch at $d_1$, which has a single vortex, is the most stable one in a sense to specify.

\item Our analysis should hold if we replace the natural boundary conditions by a harmonic trap as in some physical experiments. This would change the critical values of the width. Indeed, the sine and cosine functions have to be replaced in this setting by Hermite functions, which changes the spectrum and critical widths.

\item Our analysis should be applicable in higher dimensions, where vortex rings are observed in spaces of the type $\R \times (0, d_1) \times \ldots \times (0, d_{N - 1})$ for $N \geq 3$, with similar possible changes in the critical values for the widths $d_1$, $\ldots$, $d_{N - 1}$.

\item A mathematical question is to investigate the evolution of the bifurcating branches when the width $d$ increases. When the domain is $\R^2$, no solutions of~\eqref{eq:GL} with $k > 1$ vortex of alternate degrees $\pm 1$ are known to exist. This does not contradict our result since, as $d\to +\infty$, even if one could prove that the branches of solutions persist, we expect the distance between vortices to tend to $+\infty$.

\item In contrast we expect the bifurcating branch with a single vortex to exist for any width $d > d_1$, and to converge to the vortex solution of degree $- 1$ for the two-dimensional stationary Gross-Pitaevskii equation (see~\cite{Mirones1} and the references therein).
\end{enumerate}
\end{remarks}

\section{Sketch of the proofs}
\label{sec:sketch}

In this section we describe the main elements in the proofs of Theorem~\ref{thm:main} and Corollary~\ref{cor:main}, and then complete these proofs.

\subsection{Construction of the solutions $\Psi_{k, d}$}

It is classical to solve boundary value problems in the strip $\R \times (0, d)$ with Neumann boundary conditions by relying on a decomposition into Fourier series of the possible solutions. There indeed exists a one-to-one correspondence between smooth functions $\psi : \R \times (0, d) \to \C$ satisfying the Neumann boundary conditions in~\eqref{eq:cond-Neumann}, and the functions, which are $2 d$-periodic and even in the $y$ variable. This correspondence is obtained by extending first the function $\psi$ to the strip $\Omega_d := \R \times (- d, d)$ by reflection with respect to the $x$ axis, and then to $\R^2$ by $2 d$-periodicity with respect to the $y$ variable.

In the sequel, we take advantage of this correspondence by working with functions now defined in $\Omega_d$. More precisely, we consider the natural energy space in this context
$$
X(\Omega_d) := \big\{ \psi \in H_\text{loc}^1(\Omega_d, \C) \text{ s.t. } \nabla \psi \in L^2(\Omega_d) \text{ and } 1 - |\psi|^2 \in L^2(\Omega_d) \big\},
$$
and we look for solutions to~\eqref{eq:GL} (in $\Omega_d$) in the closed subset with the relevant symmetries, that is
\begin{equation}
\label{def:X-sym}
\boX(\Omega_d) = \big\{ \psi \in X(\Omega_d) \text{ s.t. } \psi(x, - y) = \psi(x, y) = - \overline{\psi(- x, y)} \text{ for any } (x, y) \in \Omega_d \big\}.
\end{equation}
Here as in the sequel, we use the calligraphic notation $\boF(\Omega_d)$ in order to denote the subset of a function space $F(\Omega_d)$ formed by the functions satisfying the symmetries in~\eqref{def:X-sym}.

One advantage to work in this setting lies in the possibility to decompose a function $\psi \in \boX(\Omega_d)$ in a Fourier series with respect to the $y$-variable. More precisely, we set
\begin{equation}
\label{eq:psik}
\psi_0(x) := \frac{1}{d} \int_0^d \psi(x, y) \, dy, \hbox{ and }
\psi_k(x) := \frac{2}{d} \int_0^d \psi(x, y) \, \cos \Big( \frac{\pi k y }{d} \Big)\, dy,
\end{equation}
so that the function $\psi$ can be written as
$$
\psi(x, y) = \sum_{k = 0}^{+ \infty} \psi_k(x) \, \cos \Big( \frac{\pi k y }{d} \Big),
$$
for any $(x, y) \in \Omega_d$.

In the sequel, we often use the decomposition $\psi = \psi_0 + w$ of a function $\psi \in \boX(\Omega_d)$. By the Poincar\'e-Wirtinger inequality, the function $w$ in this further decomposition belongs to the Sobolev space $H^1(\Omega_d)$, while the function $\psi_0$ is in the one-variable energy space
$$
X(\R) := \big\{ \psi \in H_\text{loc}^1(\R, \C) \text{ s.t. } \psi' \in L^2(\R) \text{ and } 1 - |\psi|^2 \in L^2(\R) \big\},
$$
or more precisely in its subset $\boX(\R)$ corresponding to functions with odd real part and even imaginary part. Note here again that, given a function space $F(\R)$, we always use in the sequel the calligraphic notation $\boF(\R)$ to denote the subset of functions with odd real part and even imaginary part.

When a function $\psi = \psi_0 + w$ is a solution to~\eqref{eq:GL}, the functions $\psi_0$ and $w$ solve the system
\begin{equation}
\label{eq:sys-0-w}
\begin{cases}
- \psi_0'' - \psi_0 \big( 1 - |\psi_0|^2 \big) = - f_0, \\
- \Delta w - w \big( 1 - |\psi_0|^2 \big) + 2 \, \langle \psi_0, w \rangle_\C \, \psi_0 = - 2 \, \langle \psi_0, w \rangle_\C \, w - |w|^2 (\psi_0 + w) + f_0,
\end{cases}
\end{equation}
where we have set
\begin{equation}
\label{def:f-0}
f_0(x) := \frac{1}{d} \int_0^d \Big( 2 \, \langle \psi_0(x), w(x, y) \rangle_\C \, w(x, y) + |w(x, y)|^2 (\psi_0(x) + w(x, y)) \Big) \, dy,
\end{equation}
for any $x \in \R$. A natural strategy to construct solutions is therefore to invert first the left-hand side of the two equations in~\eqref{eq:sys-0-w}, at least in the neighbourhood of the soliton $S_0$, and then to handle the nonlinear terms in the right-hand side by a fixed point argument.

This strategy is complicated on the one hand by the fact that the left-hand side of the first equation in~\eqref{eq:sys-0-w} is nonlinear. We will by-pass this difficulty by implementing a variational argument in order to construct a solution $\psi_0$ for any suitable right-hand side $- f_0$, and then rely on coercivity estimates in order to establish the uniqueness of this solution.

On the other hand, it turns out that the linear operator
\begin{equation}
\label{def:boT}
T(w) = - \Delta w - w (1 - |\psi_0|^2) + 2 \langle \psi_0, w \rangle_\C \, \psi_0,
\end{equation}
in the left-hand side of the second equation in~\eqref{eq:sys-0-w} is not always invertible. We rely on this property to construct our bifurcating branches of solutions $\Psi_{k, d}$. More precisely, we first observe that since the function $\psi_0$ is independent of the variable $y$, we can decompose the operator $T$ in the previous Fourier sectors as
\begin{equation}
\label{eq:dec-boT}
T(w)(x, y) = \sum_{k = 1}^{+ \infty} T_k(\psi_k)(x) \, \cos \Big( \frac{\pi k y }{d} \Big),
\end{equation}
where
$$
T_k(\psi_k) := - \psi_k'' + \frac{\pi^2 k^2}{d^2} \psi_k - \psi_k (1 - |\psi_0|^2) + 2 \langle \psi_0, \psi_k \rangle_\C \, \psi_0,
$$
and put the focus on the one-variable operators $T_k$ in this identity.

In order to analyze $T_k$, we rely on our previous assumption that the function $\psi_0$ is a small perturbation of the black soliton $S_0$. Hence the main properties of $T_k$ can be inferred from the ones of the operator
$$
L_0(\psi) := - \psi'' - \psi (1 - S_0^2) + 2 \langle S_0, \psi \rangle_\C \, S_0 = L_0^+(u) + i L_0^-(v),
$$
where we have set $\psi := u + i v$, as well as
$$
L_0^+(u) := - u'' - u \big( 1 - 3 S_0^2 \big), \quad \text{ and } \quad L_0^-(v) := - v'' - v \big( 1 - S_0^2 \big).
$$
In the previous decomposition, the Sturm-Liouville operators $L_0^+$ and $L_0^-$ are self-adjoint on $L^2(\R)$, with domain $H^2(\R)$. As a consequence of the Weyl criterion, their essential spectrum is equal to $[2, + \infty)$, respectively $[0, + \infty)$. Moreover, the function $S_0'$ is in the kernel of the operator $L_0^+$. Since this function does not vanish, it follows from classical Sturm-Liouville theory that $0$ is the lowest eigenvalue of $L_0^+$ and that the kernel of this operator is spanned by the function $S_0'$. Given any integer $k \geq 1$, the operator $L_0^+ + \pi^2 k^2/d^2$ is therefore positive definite, and as a consequence, invertible from $H^2(\R)$ to $L^2(\R)$.

Similarly the function $S_0$ belongs to the kernel of the operator $L_0^-$. Since this function has a unique zero, this operator has a unique negative eigenvalue $- \lambda_0$. A direct computation shows that $\lambda_0 = 1/2$ and that the function $\chi_0$ in~\eqref{eq:chi-0} is a corresponding eigenfunction of the operator $L_0^-$. As a consequence, the invertibility properties of the operator $L_0^- + \pi^2 k^2/d^2$ depend on the precise values of the integer $k \geq 1$ and on the width $d$. When $d \neq d_k$, the operator $L_0^- + \pi^2 k^2/d^2$ is invertible from $H^2(\R)$ to $L^2(\R)$. On the contrary, when $d = d_k$, the kernel of this operator is spanned by the function $\chi_0$. This self-adjoint operator is invertible only when restricted to the orthogonal of its kernel.

This analysis is the starting point for developing a bifurcation argument. Indeed, when the function $\psi_0$ is close enough to the black soliton $S_0$, the operators $T_k$ have invertibility properties similar to the ones of the operators $L_0^- + \pi^2 k^2/d^2$, so that a bifurcation onset is expected for $d = d_k$.

Now that our strategy to obtain bifurcating branches is clarified, we enter in more details dealing first with the invertibility of the first equation in~\eqref{eq:sys-0-w}. Solving this equation is complicated by the fact that the functional framework corresponding to the set $X(\R)$ is quite involved. Since we eventually look for perturbations of the black soliton $S_0$, we can benefit from the analysis in~\cite{GravSme1} to prove the orbital and asymptotic stabilities of $S_0$ for the one-dimensional time-dependent Gross-Pitaevskii equation. In this direction, we introduce the weighted Sobolev norm
\begin{equation}
\label{def:H0}
 \big\| \psi \big\|_{H(\R)}^2 := \int_\R \Big( |\psi'|^2 + (1 - S_0^2) \, |\psi|^2 \Big),
\end{equation}
and we endow the set $X(\R)$ with the complete metric structure associated to the distance $d$ given by
\begin{equation}
\label{def:d0}
d \big( \psi_2, \psi_1 \big)^2 := \big\| \psi_2 - \psi_1 \big\|_{H(\R)}^2 + \big\| 1 - |\psi_2|^2 - (1 - |\psi_1|^2) \big\|_{L^2(\R)}^2.
\end{equation}
In~\cite[Proposition 1]{GravSme1}, this functional setting was used to establish a coercivity estimate for the one-dimensional Ginzburg-Landau energy $\goE$ in the neighbourhood of the black soliton $S_0$. In the sequel, we rely on a similar coercivity estimate for defining properly the local minimization problem from which we solve the first equation in~\eqref{eq:sys-0-w}. In order to state precisely this alternative coercivity estimate, we introduce the subset
$$
\boX_0(\R) := \big\{ \psi \in \boX(\R) \text{ s.t. } \psi(0) = 0 \big\},$$
We are then able to show

\begin{lemma}
\label{lem:coer-S0}
There exists a number $\Lambda_0 > 0$ such that, given a function $\psi \in \boX_0(\R)$, we have
\begin{equation}
\label{eq:coer-S0}
\goE(\psi) - \goE(S_0) \geq \Lambda_0 \Big( \| \varepsilon \|_{H(\R)}^2 + \| \eta \|_{L^2(\R)}^2 \Big) - \frac{1}{\Lambda_0} \big\| \varepsilon \big\|_{H(\R)}^3,
\end{equation}
where $\varepsilon := \psi - S_0$ and $\eta := 1 - |\psi|^2 - (1 - S_0^2) = - 2 \langle S_0, \varepsilon \rangle_\C - |\varepsilon|^2$.
\end{lemma}

With Lemma~\ref{lem:coer-S0} in hand, we can describe our strategy to solve the first equation in~\eqref{eq:sys-0-w}. This strategy is first variational. Fix a function $f_0 \in L^1(\R, \C) \cap L^2(\R, \C)$, and consider the functional
$$
I_{f_0}(\psi) = \goE(\psi) + \int_\R \langle f_0, \psi \rangle_\C.
$$
Since
\begin{equation}
\label{eq:estim_f_0}
\big| \langle f_0, \psi \rangle_\C \big| \leq |f_0| + \big| 1 - |\psi|^2 \big| \, |f_0|,
\end{equation}
the functional $I_{f_0}$ is well-defined in $\boX_0(\R)$. Given a number $\alpha > 0$, we can introduce the subsets
$$
\boU(\alpha) := \big\{ \psi \in \boX_0(\R) \text{ s.t. } d(\psi, S_0) < \alpha \big\},
$$
and consider the local minimization problems
\begin{equation}
\label{def:I-alpha}
\boI_{f_0}(\alpha) := \inf_{\psi \in \boU(\alpha)} \, I_{f_0}(\psi).
\end{equation}

\begin{lemma}
\label{lem:sol-psi0}
There exist two numbers $\alpha_0 := \Lambda_0^2/2$, and $\beta_0 > 0$, depending only on $\Lambda_0$, such that, for $0 \leq \beta \leq \beta_0$ and any continuous function $f_0 \in \boL^1(\R, \C) \cap \boL^2(\R, \C)$ satisfying
\begin{equation}
\label{eq:cond-f0}
\| f_0 \|_{L^1(\R)} + \| f_0 \|_{L^2(\R)} \leq \beta,
\end{equation}
the minimization problem $\boI_{f_0}(\alpha_0)$ given by~\eqref{def:I-alpha} has a minimizer $\psi_0 \in \boU(\alpha_0)$. This minimizer is of class $\boC^2$ in $\R$, and it satisfies
\begin{equation}
\label{eq:psi-0}
- \psi_0'' - \psi_0 (1 - |\psi_0|^2) + f_0 = 0.
\end{equation}
Moreover, there exists a positive number $C_0$, depending only on $\Lambda_0$, such that
\begin{equation}
\label{eq:dist-S0}
d \big( \psi_0, S_0 \big)^2 \leq C_0 \beta.
\end{equation}
\end{lemma}

In view of Lemma~\ref{lem:sol-psi0}, the function $N_0 := 1 - |\psi_0|^2$ corresponding to the minimizer $\psi_0$ is also of class $\boC^2$ on $\R$, and we can check that it decays exponentially at $\pm \infty$ when the function $f_0$ also decays exponentially. In order to prove this claim, we first introduce the weighted Lebesgue spaces
$$
L_\sigma^p(\R) := \Big\{ \psi \in L_\text{loc}^p(\R, \C) \text{ s.t. } \big\| \psi \big\|_{L_\sigma^p(\R)} := \big\| \psi \, e^{\sigma |\cdot|} \big\|_{L^p(\R)} < + \infty \Big\},
$$
for any numbers $\sigma \in \R$ and $1 \leq p \leq \infty$. Note that these function spaces reduce to the classical Lebesgue spaces for $\sigma = 0$.

We next derive from~\eqref{eq:psi-0} the equation
\begin{equation}
\label{eq:eta-0-1}
- N_0'' + 2 N_0 = 2 \big( |\psi_0'|^2 + N_0^2 + \langle f_0, \psi_0 \rangle_\C \big).
\end{equation}
By taking the complex scalar product of~\eqref{eq:psi-0} with the derivative $\psi_0'$, we also obtain
$$
\Big( |\psi_0'|^2 - \frac{N_0^2}{2} \Big)' = 2 \langle f_0, \psi_0' \rangle_\C.
$$
Recall at this stage that any function in the energy space $X(\R)$ is bounded (see e.g.~\cite{Gerard2}). As a consequence, the function $N_0$ belongs to the Sobolev space $H^1(\R)$, and by~\eqref{eq:psi-0} again, so does the derivative $\psi_0'$. In particular, both functions tend to $0$ at $\pm \infty$, and we obtain the equation
\begin{equation}
\label{eq:deriv-eta}
|\psi_0'|^2 = \frac{N_0^2}{2} + 2 g_0,
\end{equation}
where we have set
\begin{equation}
\label{def:g-0}
g_0(x) = \int_{- \infty}^x \langle f_0(t), \psi_0'(t) \rangle_\C \, dt = - \int_x^{+ \infty} \langle f_0(t), \psi_0'(t) \rangle_\C \, dt.
\end{equation}
Going back to~\eqref{eq:eta-0-1}, we are led to the differential equation
\begin{equation}
\label{eq:eta-0}
- N_0'' + 2 N_0 - 3 N_0^2 = 2 \big( \langle f_0, \psi_0 \rangle_\C + 2 g_0 \big).
\end{equation}
Relying on this further equation, we can show the following estimate for
\begin{equation}
\label{eq:epseta}
\varepsilon_0 = \psi_0 - S_0, \quad \text{ and } \quad \eta_0 = \big( 1 - |\psi_0|^2 \big) - \big( 1 - |S_0|^2 \big) = |S_0|^2 - |\psi_0|^2.
\end{equation}

\begin{lemma}
\label{lem:exp-dec-eta-0}
Let $0 < \sigma < \sqrt{2}$. There exists a number $0 < \beta_\sigma \leq \beta_0$ such that, given a function $f_0 \in \boL_\sigma^\infty(\R)$ satisfying the condition in~\eqref{eq:cond-f0} for a number $0 \leq \beta \leq \beta_\sigma$, and the corresponding minimizer $\psi_0$ of the problem $\boI_{f_0}(\alpha_0)$, the functions $N_0$ and $\psi_0'$ belong to the weighted Lebesgue space $L_\sigma^\infty(\R)$. Moreover, there exists a number $C_\sigma > 0$, depending only on $\sigma$ and $\Lambda_0$, such that the functions $\eta_0$ and $\varepsilon_0$ defined above satisfy
\begin{equation}
\label{eq:exp-dec-0}
\big\| \eta_0 \big\|_{L_\sigma^\infty(\R)}^2 +\big\| \varepsilon_0 \big\|_{L^\infty(\R)}^2 + \big\| \varepsilon_0' \big\|_{L_\sigma^\infty(\R)}^2 \leq C_\sigma \Big( \beta + \big\| f_0 \big\|_{L_\sigma^\infty(\R)}^2 \Big) \Big( 1 + \beta + \big\| f_0 \big\|_{L_\sigma^\infty(\R)}^2 \Big).
\end{equation}
Note that $\varepsilon_0$ is bounded in $L^\infty(\R)$, {\em not} in $L^\infty_\sigma(\R)$.
\end{lemma}

A difficulty with the variational construction in Lemma~\ref{lem:sol-psi0} lies in the property that the minimizer $\psi_0$ is not necessarily unique. Since this function is in $\boX_0(\R)$, we know that $\psi_0(0) = 0$, and also that $\text{Im}(\psi_0')(0) = 0$, but this is not enough to apply the Cauchy-Lipschitz theorem. In order to by-pass this difficulty, we rely on the coercivity estimate in Lemma~\ref{lem:coer-S0} to establish some Lipschitz dependence of the function $\psi_0$ on the function $f_0$. In the sequel, this property will also be useful to solve the complete system in~\eqref{eq:sys-0-w} by a fixed point argument.

In order to show this Lipschitz dependence, we first recall that the minimizer $\psi_0$ is a perturbation of the black soliton $S_0$ due to~\eqref{eq:dist-S0}. Define $\varepsilon_0$ and $\eta_0$ by~\eqref{eq:epseta}. Then we check that
\begin{equation}
\label{eq:eps-0}
- \varepsilon_0'' - (1 - S_0^2) \varepsilon_0 - (S_0 + \varepsilon_0) \eta_0 = - f_0.
\end{equation}
Similarly, we derive from~\eqref{eq:eta-0} that
\begin{equation}
\label{eq:N-0}
- \eta_0'' + 2 \eta_0 - 6 (1 - S_0^2) \eta_0 - 3 \eta_0^2 = 2 \langle f_0, \psi_0 \rangle_\C + 2 g_0.
\end{equation}
We then consider another function $f_1 \in \boL_\sigma^\infty(\R)$ satisfying the condition in~\eqref{eq:cond-f0} for the same number $0 \leq \beta \leq \beta_\sigma$, and we denote by $\psi_1 \in \boU(\alpha_0)$ a minimizer of the minimization problem $\boI_{f_1}(\alpha_0)$ given by Lemma~\ref{lem:sol-psi0}. Introducing the function $g_1$ given by~\eqref{def:g-0}, and setting as above $\varepsilon_1 := \psi_1 - S_0$ and $\eta_1 := N_1 - (1 - S_0^2)$, we first control the differences $\varepsilon_1 - \varepsilon_0$ and $\eta_1 - \eta_0$ in $H(\R)$, respectively in $L^2(\R)$.

\begin{lemma}
\label{lem:control-V-N}
Let $0 < \sigma < \sqrt{2}$. There exists $\gamma_\sigma$, $A_\sigma>0$ such that for any $\gamma < \gamma_\sigma$, if $f_0$, $f_1\in \boL_\sigma^\infty(\R)$ are such that
\begin{equation}
\label{eq:cond-exp-f0}
\big\| f_0 \big\|_{L_\sigma^\infty(\R)} + \big\| f_1 \big\|_{L_\sigma^\infty(\R)} \leq \gamma,
\end{equation}
then the differences $\psi_1 - \psi_0$ and $\eta_1 - \eta_0$ satisfy the estimate
\begin{equation}
\label{eq:V-N}
\| \psi_1 - \psi_0 \|_{H(\R)} + \| \eta_1 - \eta_0 \|_{L^2(\R)} \leq A_\sigma \| f_1 - f_0 \|_{L_\sigma^\infty(\R)}.
\end{equation}
\end{lemma}

In particular, under the assumptions of Lemma~\ref{lem:control-V-N}, there exists a unique minimizer $\psi_0\in \boU(\alpha_0)$ for the minimization problem $\boI_{f_0}(\alpha_0)$. Indeed, given any other possible minimizer $\psi_1$, the inequality~\eqref{eq:V-N} holds for $f_1 = f_0$, so that $\psi_1 = \psi_0$.

We next upgrade the previous control on the differences $\varepsilon_1' - \varepsilon_0'$ and $\eta_1 - \eta_0$ from $L^2(\R)$ to $L_\sigma^\infty(\R)$. More precisely, we establish the following Lipschitz dependence of the functions $\psi_0$ and $N_0$ on the function $f_0$.

\begin{lemma}
\label{lem:control-exp-V-N}
Let $0 < \sigma < \sqrt{2}$. Assume that two continuous functions $f_0$ and $f_1$ in $\boL_\sigma^\infty(\R)$ satisfy the condition in~\eqref{eq:cond-exp-f0} for a number $\gamma \leq \gamma_\sigma$. There exists a positive number $B_\sigma$, depending only on $\sigma$, such that the differences $\psi_1 - \psi_0$ and $\eta_1 - \eta_0$ satisfy the estimate
$$
\big\| \psi_1' - \psi_0' \big\|_{L_\sigma^\infty(\R)} + \big\| \psi_1 - \psi_0 \big\|_{L^\infty(\R)} + \big\| \eta_1 - \eta_0\big\|_{L_\sigma^\infty(\R)} \leq B_\sigma \big\| f_1 - f_0 \big\|_{L_\sigma^\infty(\R)}.
$$
\end{lemma}

This concludes our analysis of the first equation in~\eqref{eq:sys-0-w}. We now turn to the second equation. In view of Lemma~\ref{lem:sol-psi0}, we assume that the function $\psi_0$ is fixed in $\boX_0(\R)$, and that it belongs to the open set $\boU(\mu)$ for some number $\mu > 0$.

Our goal is now to invert the operator $T$ in~\eqref{def:boT}. Decomposing it as in~\eqref{eq:dec-boT}, we put the focus on the one-variable operators $T_k$. We first claim that we can restrict their analysis to the space $H^1(\R)$. Invoking the Poincar\'e-Wirtinger inequality (again with respect to $y$)
$$
\big\| w \big\|_{L^2(\Omega_d)} \leq \frac{d}{\pi} \big\| \partial_y w \big\|_{L^2(\Omega_d)},
$$
we indeed check that the function
$$
w(x,y) = \sum_{k=1}^{+\infty} \psi_k(x)\cos \Big( \frac{\pi k y} d \Big)
$$
does not only belong to $\dot{H}^1(\Omega_d)$, but actually to $H^1(\Omega_d)$, so that the functions $\psi_k$ are in $H^1(\R)$.

We next recall that, when the function $\psi_0$ is close to the black soliton $S_0$, the operator $T_k$ appears as a perturbation of the operator $L_0 + \pi^2 k^2/d^2$, which we have previously analyzed in this sketch of the proof. Summarizing this analysis, we can claim that the operator $L_0 + \pi^2 k^2/d^2$ is invertible, when $d \neq d_k$, whereas it is not for $d = d_k$. In this latter case, it is however invertible when restricted to the intersection of $H^2(\R, \C)$ with the orthogonal space
$$
H_0 := \big\{ \psi \in L^2(\R, \C) \text{ s.t. } \langle \psi, i \chi_0 \rangle_{L^2(\R)} = 0 \big\}.
$$
In the two cases, we can also check that the norm of its inverse, as an operator from $L^2(\R)$ to $L^2(\R)$, is bounded uniformly with respect to $k \geq 1$.

We claim that the operators $T_k$ have similar invertibility properties when the function $\psi_0$ is close enough to the black soliton $S_0$. By using the Parseval identity, this provides a precise description of the invertibility properties of the operator $T$.

In order to state them precisely, we now fix an integer $k \geq 1$ and introduce the subspace
$$
H_k := \Big\{ \psi \in L^2(\Omega_d, \C) \text{ s.t. } \psi_0 = 0 \text{ and } \psi_k \in H_0 \Big\},
$$
where $\psi_0$ and $\psi_k$ are given by~\eqref{eq:psik}, as well as the corresponding orthogonal projection $\pi_k$ given by
\begin{equation}
\label{def:pi-k}
\forall \psi \in L^2(\Omega_d, \C), \, \pi_k(\psi) := \psi - \frac{\langle \psi, \chi_k \rangle_{L^2(\Omega_d)}}{\| \chi_k \|_{L^2(\Omega_d)}^2} \chi_k,
\end{equation}
where $\chi_k$ is the function in~\eqref{def:chi-k}.

We also introduce the weighted Lebesgue spaces
$$
L_\sigma^p(\Omega_d) := \Big\{ \psi \in L_\text{loc}^p(\Omega_d, \C) \text{ s.t. } \big\| \psi \big\|_{L_\sigma^p(\Omega_d)} := \big\| \psi(x, y) \, e^{\sigma |x|} \big\|_{L^p(\Omega_d)} < + \infty \Big\},
$$
for any numbers $\sigma \in \R$ and $1 \leq p \leq \infty$.

When $d$ and $\psi_0$ are close enough to $d_k$, respectively $S_0$, we obtain the following invertibility properties for the operator $T$.

\begin{lemma}
\label{lem:invert-T-k}
Let $k \geq 1$. There exist three numbers $\delta_0 > 0$, $\mu_0 > 0$ and $\kappa_0 > 0$ such that, for any width $d_k - \delta_0 < d < d_k + \delta_0$ and any function $\psi_0 \in \boU(\mu_0)$ such that
\begin{equation}
\label{eq:unif-eta-eps-0}
\big\| \eta_0 \big\|_{L^\infty(\R)} +\big\| \varepsilon_0 \big\|_{L^\infty(\R)} < \mu_0,
\end{equation}
where $\varepsilon_0$ and $\eta_0$ are defined by~\eqref{eq:epseta}, the following holds:

\noindent $(i)$ The operator $T$ given by~\eqref{def:boT} is invertible from $\boH^2(\Omega_d, \C) \cap H_k$ to $\boL^2(\Omega_d, \C) \cap H_k$ : Given any function $g \in \boL^2(\Omega_d, \C),$ there exists a unique function $w \in \boH^2(\Omega_d, \C) \cap H_k$ such that
\begin{equation}
\label{eq:proj-inv-L1}
\pi_k \big( T(w) - g \big) = 0.
\end{equation}
Moreover this function satisfies
\begin{equation}
\label{eq:proj-inv-L2}
\big\| w \big\|_{H^2(\Omega_d)} \leq \kappa_0 \big\| g \big\|_{L^2(\Omega_d)}.
\end{equation}

\noindent $(ii)$ There exist $\sigma_0$, $\delta_1$, $\nu_0 > 0$ such that, if $0<\tau<\sigma<\sigma_0$ and $d_k - \delta_1 < d < d_k + \delta_1$ then, assuming that
\begin{equation}
\label{eq:cond-eta-eps-0}
\big\| \eta_0 \big\|_{L_\sigma^\infty(\R)} + \big\| \varepsilon_0 \big\|_{L^\infty(\R)} + \big\| \varepsilon_0' \big\|_{L_\sigma^\infty(\R)} < \nu_0,
\end{equation}
the condition in~\eqref{eq:unif-eta-eps-0} is satisfied as well, and for any $g$ in $L_\sigma^\infty(\Omega_d, \C)$, the function $w$ given by~\eqref{eq:proj-inv-L1} satisfies the following estimate
\begin{equation}
\label{eq:exp-bound-invert-T}
\big\| w \big\|_{L_\tau^\infty(\Omega_d)} \leq \kappa_\tau \big\| g \big\|_{L_\sigma^\infty(\Omega_d)},
\end{equation}
where $\kappa_\tau \geq 0$ depends on $k$, $\delta_1$, $\nu_0$, $\sigma_0$, $\sigma$ and $\tau$.
\end{lemma}

In the sequel, we invoke this invertibility property in order to write the second equation in~\eqref{eq:sys-0-w} as a fixed point equation, and then solve it by a fixed point argument. The fact that this argument must be combined to a similar one for solving the first equation in~\eqref{eq:sys-0-w}, in view of the Lipschitz control in Lemma~\ref{lem:control-exp-V-N}, is why it is natural to perform this second fixed point argument in the spaces $L_\sigma^\infty(\Omega_d)$.

With Lemma~\ref{lem:invert-T-k} at hand, we are in a position to settle the fixed point argument which we apply in order to solve the system in~\eqref{eq:sys-0-w}. We fix a number $k \geq 1$ and a width $d_k - \delta_1 < d < d_k + \delta_1$, where $\delta_1$ is the positive number given by Lemma~\ref{lem:invert-T-k}. For $0 < \sigma \leq \sigma_0/2$, we also introduce the set
$$
\boY_\sigma^\infty(\R) := \Big\{ \psi_0 \in \boX_0(\R) \text{ s.t. } \psi_0' \in L_\sigma^\infty(\R) \Big\},
$$
Note that this set is not a vector space, but it is a subset of the vector space
$$
W_{0, \sigma}^{1, \infty}(\R) := \Big\{ \psi_0 \in \boC^0(\R, \C) \text{ s.t. } \psi_0(0) = 0 \text{ and } \psi_0' \in L_\sigma^\infty(\R) \Big\},
$$
which we can endow with the norm
$$
\big\| \psi_0 \big\|_{W_{0, \sigma}^{1, \infty}(\R)} := \big\| \psi' \big\|_{L_\sigma^\infty(\R)}.
$$
In particular, the set $\boY_\sigma^\infty(\R)$ is naturally endowed with the metric structure given by the distance corresponding to this norm.

Given a function $\psi_0 \in \boY_\sigma^\infty(\R)$, and a map $w \in \boL_\sigma^\infty(\Omega_d)$, we next denote by $f_0(\psi_0, w)$ the function given by~\eqref{def:f-0}, and we also define the nonlinearity
\begin{equation}
\label{def:g}
g(\psi_0, w) := - 2 \, \langle \psi_0, w \rangle_\C \, w - |w|^2 (\psi_0 + w) + f_0(\psi_0, w).
\end{equation}
We claim that these functions are in $\boL_{2 \sigma}^\infty(\R)$, respectively in $\boL_{2 \sigma}^\infty(\Omega_d)$, and have Lipschitz continuous dependence on the functions $\psi_0$ and $w$. More precisely, we show

\begin{lemma}
\label{lem:non-linearity}
Let $0 \leq \sigma \leq \sigma_0/2$. The maps $f_0$ and $g$ are well-defined from $\boY_\sigma^\infty(\R) \times \boL_\sigma^\infty(\Omega_d)$ to $\boL_{2 \sigma}^\infty(\R)$, respectively to $\boL_{2 \sigma}^\infty(\Omega_d)$. Moreover there exists a universal constant $K \geq 0$ such that
\begin{equation}
\label{eq:lips-f-0-g}
\begin{split}
\Big\| f_0 \big( \tilde{\psi}_0, \tilde{w} \big) - f_0 \big( \psi_0, w \big) \Big\|_{L_{2 \sigma}^\infty(\R)} + & \Big\| g \big( \tilde{\psi}_0, \tilde{w} \big) - g \big( \psi_0, w \big) \Big\|_{L_{2 \sigma}^\infty(\Omega_d)} \\
\leq K \Big( \big\| \psi_0 - \tilde{\psi}_0 \big\|_{L^\infty(\R)} \, \big\| w \big\|_{L_\sigma^\infty(\Omega_d)}^2 + & \big\| \tilde{w} - w \big\|_{L_\sigma^\infty(\Omega_d)} \, \big( \big\| \tilde{w} \big\|_{L_\sigma^\infty(\Omega_d)} + \big\| w \big\|_{L_\sigma^\infty(\Omega_d)} \big) \times \\
& \times \big( \big\| \tilde{\psi}_0 \big\|_{L^\infty(\R)} + \big\| \tilde{w} \big\|_{L_\sigma^\infty(\Omega_d)} + \big\| w \big\|_{L_\sigma^\infty(\Omega_d)}\big) \Big),
\end{split}
\end{equation}
for any pairs $(\tilde{\psi}_0, \tilde{w})$ and $(\psi_0, w)$ in $\boY_\sigma^\infty(\R) \times \boL_\sigma^\infty(\Omega_d)$.
\end{lemma}

It follows from~\eqref{eq:lips-f-0-g} that, if we restrict ourselves to pairs $(\tilde{\psi}_0, \tilde{w})$ and $(\psi_0, w)$, which are close enough to the pair $(S_0, 0)$, the Lipschitz constant of $f_0$ and $g$ are small. This last observation allows to implement a fixed point argument, which we now detail.

We first introduce a small parameter $\rho > 0$ and the closed ball $\overline{B}((S_0, 0), \rho)$ with center $(S_0, 0)$ and radius $\rho$ in $\boY_\sigma^\infty(\R) \times (\boL_\sigma^\infty(\Omega_d) \cap H_k)$. Given a pair $(\psi_0, w)$ in this ball, and a real number $|\lambda| \leq \rho $, we deduce from the definition of the function $f_0(\psi_0, w + \lambda \chi_k)$ and its Lipschitz continuity in~\eqref{eq:lips-f-0-g} that this function satisfies the conditions in~\eqref{eq:cond-f0} (with $\beta = \beta_{2 \sigma}$) and in~\eqref{eq:cond-exp-f0} (with $\gamma = \gamma_{2 \sigma}$) when $\rho$ is chosen small enough. As a consequence of Lemma~\ref{lem:exp-dec-eta-0}, there exists a unique minimizer $\psi_0^\lambda \in \boU(\alpha_0)$ of the minimization problem $\boI_{f_0(\psi_0, w + \lambda \chi_k)}(\alpha_0)$. Moreover, the corresponding functions $\eta_0^\lambda$ and $(\varepsilon_0^\lambda)'$ are in $L_\sigma^\infty(\R)$ and satisfy the estimate in~\eqref{eq:exp-dec-0}. Decreasing if necessary the value of the number $\rho$, we can assume that these functions satisfy the condition in~\eqref{eq:cond-eta-eps-0} with $\nu \leq \nu_0$, and then that the function $\psi_0^\lambda$ is in $\boU(\mu_0)$.

Invoking Lemma~\ref{lem:invert-T-k}, we next consider the function $w^\lambda \in \boL_\sigma^\infty(\R)$ corresponding by~\eqref{eq:proj-inv-L1} to the function $g(\psi_0, w + \lambda \chi_k)$, and then the map $\Xi_k$ given by
$$
\Xi_k \big( \psi_0, w, \lambda \big) = \big( \psi_0^\lambda, w^\lambda - \lambda \chi_k \big).
$$
Going back to Lemmas~\ref{lem:control-exp-V-N} and~\ref{lem:invert-T-k}, we claim that for $\rho$ small enough and $|\lambda| < \rho$, the map $\Xi_k(\cdot, \cdot, \lambda)$ is a contraction on the closed ball $\overline{B}((S_0, 0), \rho)$. This property is enough to show the following existence result in which we also use the vector space
$$
W_{0, \sigma}^{2, \infty}(\R) := \Big\{ \psi_0 \in W_{0, \sigma}^{1, \infty}(\R) \text{ s.t. } \psi_0'' \in L_\sigma^\infty(\R) \Big\},
$$
which we naturally endow with the norm
$$
\big\| \psi_0 \big\|_{W_{0, \sigma}^{2, \infty}(\R)} := \big\| \psi_0 \big\|_{W_{0, \sigma}^{1, \infty}(\R)} + \big\| \psi_0'' \big\|_{L_\sigma^\infty(\R)},
$$
as well as the vector space
$$
W_{0, \sigma}^{2, \infty}(\Omega_d) := \Big\{ \psi = \psi_0 + w \in L^2(\Omega_d, \C) \text{ s.t. } \psi_0 = 0, \, \nabla w \in L_\sigma^\infty(\Omega_d) \text{ and } D^2 w \in L_\sigma^\infty(\Omega_d) \Big\},
$$
similarly endowed with the norm
$$
\big\| \psi \big\|_{W_\sigma^{2, \infty}(\Omega_d)} = \big\| \nabla w \big\|_{L_\sigma^\infty(\Omega_d)} + \big\| D^2 w \big\|_{L_\sigma^\infty(\Omega_d)}.
$$

\begin{proposition}
\label{prop:exist-proj-sol}
For $k \geq 1$, let $0 \leq \sigma \leq \sigma_0/2$ and $d_k - \delta_1 < d < d_k + \delta_1$, where $\sigma_0$ and $\delta_1$ are defined in Lemma~\ref{lem:invert-T-k}. There exists a number $\rho_1 > 0$ such that, given any real number $|\lambda| < \rho_1$, there exist maps $(\Psi_0^\lambda, W^\lambda) \in \boY_\sigma^\infty(\R) \times (\boL_\sigma^\infty(\Omega_d)\cap H_k)$ such that $\Psi_0^\lambda$ and $W^\lambda$ are smooth on $\Omega_d$ and solve the equations
\begin{equation}
\label{eq:proj-syst-0}
- \big( \Psi_0^\lambda \big)'' - \Psi_0^\lambda \big( 1 - |\Psi_0^\lambda|^2 \big) = - f_0 \big( \Psi_0^\lambda, W^\lambda + \lambda \chi_k \big),
\end{equation}
and
\begin{equation}
\label{eq:proj-syst-1}
\begin{split}
\pi_k \bigg( - \Delta \big( W^\lambda + \lambda \chi_k \big) - \big( W^\lambda+ \lambda \chi_k \big) \big( 1 - |\Psi_0^\lambda|^2 \big) & \\ + 2 \, \langle \Psi_0^\lambda, W^\lambda + & \lambda \chi_k \rangle_\C \, \Psi_0^\lambda - g\big( \Psi_0^\lambda, W^\lambda + \lambda \chi_k \big) \bigg) = 0.
\end{split}
\end{equation}
Moreover, these maps are the unique solutions of the previous equations in $\boY_\sigma^\infty(\R) \times (\boL_\sigma^\infty(\Omega_d)\cap H_k)$ such that
$$
\big\| \Psi_0^\lambda - S_0 \big\|_{W_{0, \sigma}^{1, \infty}(\R)} < \rho_1, \quad \text{ and } \quad \big\| W^\lambda \big\|_{L_\sigma^\infty(\Omega_d)} < \rho_1,
$$
and the map $\lambda \to (\Psi_0^\lambda, W^\lambda)$ is smooth from $(- \rho_1, \rho_1)$ to $W_{0, \sigma}^{2, \infty}(\R) \times W_{0, \sigma}^{2, \infty}(\Omega_d)$, while the map $\lambda \mapsto 1 - |\Psi_0^\lambda|^2$ is smooth from $(- \rho_1, \rho_1)$ to $L^2(\R)$.
\end{proposition}

Equation~\eqref{eq:proj-syst-1} may be written
$$
\pi_k \Big( T \big( W^\lambda + \lambda \chi_k \big) - g \big( \Psi_0^\lambda, W^\lambda + \lambda \chi_k \big) \Big) = 0,
$$
where $\pi_k$ is defined in~\eqref{def:pi-k}, hence $W^\lambda+\lambda\chi_k$ satisfies the second equation in~\eqref{eq:sys-0-w} if and only if $\langle T(W^\lambda+\lambda\chi_k),\chi_k\rangle_{L^2} = 0$. Therefore, the function $\Psi^\lambda = \Psi_0^\lambda + W^\lambda + \lambda \chi_k$ is a solution to~\eqref{eq:GL} whenever the quantity
\begin{equation}
\label{eq:Jdl}
\begin{split}
J(d, \lambda) := \big\langle - \Delta \big( W^\lambda + \lambda \chi_k \big) & - \big( W^\lambda + \lambda \chi_k \big) \big( 1 - |\Psi_0^\lambda|^2 \big) \\
& + 2 \, \langle \Psi_0^\lambda, W^\lambda + \lambda \chi_k \rangle_\C \, \Psi_0^\lambda - g\big( \Psi_0^\lambda, W^\lambda + \lambda \chi_k \big), \chi_k \big\rangle_{L^2(\Omega_d)},
\end{split}
\end{equation}
vanishes. For $d_k - \delta_1 < d < d_k + \delta_1$ and $\lambda = 0$, the function $\Psi_0$ is, by uniqueness of the fixed point, equal to the black soliton $S_0$, so that
$$
J(d, 0) = 0.
$$
Invoking this property, it is natural to implement a perturbative argument in order to describe locally the set $\boZ$ of widths $d$ and of parameters $\lambda$ for which the function $J$ vanishes. Note that this is the reason why we make the dependence on $d$ of the function $J$ explicit. In practice, applying such a perturbative argument requires to establish first some smoothness for the function $J$. This can be done by invoking some fixed-point theorem with parameters. In this direction, we show

\begin{lemma}
\label{lem:diff-I}
Let $k \geq 1$ and $0 < \sigma \leq \sigma_0/2$ be fixed as in Proposition~\ref{prop:exist-proj-sol}. There exist two numbers $0 < \delta_2 \leq \delta_1$ and $0 < \rho_2 \leq \rho_1$ such that the function $J$ defined by~\eqref{eq:Jdl} is smooth on $(d_k - \delta_2, d_k + \delta_2) \times (- \rho_2, \rho_2)$, and odd in its second variable $\lambda$. Moreover, the following holds:
\begin{equation}
\label{eq:der-equal-0}
J(d_k, 0) = \partial_d J(d_k, 0) = \partial_\lambda J(d_k, 0) = \partial_{d, d} J(d_k, 0) = \partial_{\lambda, \lambda} J(d_k, 0) = 0,
\end{equation}
and
\begin{equation}
\label{eq:der-neq-0}
\partial_{d, \lambda} J(d_k, 0) = - 2 \sqrt{2}, \quad \text{ and } \quad \partial_{\lambda, \lambda, \lambda} J(d_k, 0) := \omega \, d_k,
\end{equation}
where $\omega > 0$ is a universal constant.
\end{lemma}

Invoking Lemma~\ref{lem:diff-I}, we are in position to apply the Morse lemma in order to characterize the vanishing set $\boZ$ of the function $J$ in the neighbourhood of the point $(d_k, 0)$. In view of~\eqref{eq:der-equal-0} and~\eqref{eq:der-neq-0}, this set is locally diffeomorphic to two secant lines. This property eventually provides the existence of two smooth branches of solutions to~\eqref{eq:GL}, the first one corresponding to the black solitons. The final point in the proof of Theorem~\ref{thm:main} is to establish that the second one provides truly two-dimensional solutions on strips with widths close to the critical width $d_k$.

\subsection{End of the proof of Theorem~\ref{thm:main}}

Our starting point is the fact that the black soliton $S_0$ is a solution to~\eqref{eq:GL} in any strip $\Omega_d$. As a result, the pair $(S_0, 0)$ is a solution of equations~\eqref{eq:proj-syst-0} and~\eqref{eq:proj-syst-1}. By the uniqueness of this solution in Proposition~\ref{prop:exist-proj-sol}, we infer that $\Psi_0^0 = S_0$ and $W^0 = 0$ for $d_k - \delta_1 < d < d_k + \delta_1$. This guarantees that
$$
J(d, 0) = 0,
$$
for any $d_k - \delta_1 < d < d_k + \delta_1$. We can therefore write the function $J$ as
\begin{equation}
\label{def:boJ}
J(d, \lambda) = \lambda \, \boJ(d, \lambda) := \lambda \int_0^1 \partial_\lambda J(d, \lambda t) \, dt,
\end{equation}
where, invoking Lemma~\ref{lem:diff-I}, the function $\boJ$ is smooth on $(d_k - \delta_1, d_k + \delta_1) \times (- \rho_2, \rho_2)$. Using Lemma~\ref{lem:diff-I} again, we compute
$$
\boJ(d_k, 0) = \partial_\lambda J(d_k, 0) = 0, \quad \text{ and } \quad \partial_d \boJ(d_k, 0) = \partial_{d, \lambda} J(d_k, 0) = - 2 \sqrt{2} < 0,
$$
Applying the implicit function theorem, we can find two numbers $0 < \delta_2 < \delta_1$ and $0 < \rho_3 < \rho_2$, and a smooth function $\bd : (- \rho_3, \rho_3) \to (d_k - \delta_2, d_k + \delta_2)$ such that, for any $d_k - \delta_2 < d < d_k + \delta_2$ and any $- \rho_3 < \lambda < \rho_3$,
\begin{equation}
\label{eq:equiv-d-lambda}
\boJ(d, \lambda) = 0 \Longleftrightarrow d = \bd(\lambda).
\end{equation}
In view of~\eqref{def:boJ}, this equivalence means that the intersection of the vanishing set $\boZ$ of the function $(d,\lambda)\to J(d,\lambda)$ with the subset $(d_k - \delta_2, d_k + \delta_2) \times (- \rho_3, \rho_3)$ is the union of the smooth curves $\lambda = 0$ and $d = \bd(\lambda)$, which intersect at the point $(d_k, 0)$.

Concerning the function $\bd$, we first recall that the function $J$ is odd with respect to the variable $\lambda$. Hence it follows from~\eqref{def:boJ} and~\eqref{eq:equiv-d-lambda} that $\bd$ is an even function. In particular, its derivative $\bd'(0)$ is equal to $0$. Observing that
$$
\partial_{\lambda, \lambda} \boJ(d, \lambda) = \int_0^1 \partial_{\lambda, \lambda, \lambda} J(d, \lambda t) \, t^2 \, dt,
$$
we deduce from applying the chain rule to the identity $\boJ(\bd(\lambda), \lambda) = 0$ that
\begin{equation}
\label{eq:d-der-twice}
\bd''(0) = \frac{\omega d_k}{6 \sqrt{2}} > 0,
\end{equation}
where $\omega$ is a universal constant, which comes from~\eqref{eq:der-neq-0}. As a consequence, we can reduce the value of the number $\rho_3$ (if necessary) so that
$$
\bd'(\lambda) > 0,
$$
for any $0 < \lambda < \rho_3$. The function $\bd$ is then smoothly invertible from $(0, \lambda_3)$ onto its image $(d_k, d_k + \ell_k)$, and its inverse $\bm{\lambda} := \bd^{- 1}$ satisfies
$$
J \big( d, \bm{\lambda}(d) \big) = 0,
$$
for any $d_k < d < d_k + \ell_k$. Note here that this inverse is actually continuous on $[d_k, d_k + \ell_k)$. Note also in view of~\eqref{eq:d-der-twice}, that $\bd(\lambda) - d_k \sim \omega d_k \, \lambda^2/(12 \sqrt{2})$ as $\lambda$ goes to $0$, so that
\begin{equation}
\label{eq:equiv-lambda}
\bm{\lambda}(d) \sim \sqrt{\frac{12 \sqrt{2} (d_k - d)}{\omega d_k}},
\end{equation}
as $d \to d_k$.

It then follows from Proposition~\ref{prop:exist-proj-sol} that the functions
\begin{equation}
\label{def:Psi-d}
\Psi_{k, d} := \Psi^{\bm{\lambda}(d)} := \Psi_0^{\bm{\lambda}(d)} + W^{\bm{\lambda}(d)} + \bm{\lambda}(d) \, \chi_k,
\end{equation}
are solutions to~\eqref{eq:GL} on $\Omega_d$. Concerning the smoothness of the map $d \mapsto \Psi_{k, d}$, we have to deal with the property that they are not defined on the same strips. We settle this difficulty by applying the change of variables $z = d_k y/d$, and introducing the rescaled functions
$$
\tilde{\Psi}_{k, d}(x, z) := \Psi_{k, d} \Big( x, \frac{d \, z}{d_k} \Big),
$$
which are all defined in the strip $\Omega_{d_k}$. Letting similarly $\tilde{\Psi}_0^d(x) := \Psi_0^{\bm{\lambda}(d)}(x)$ and $\tilde{w}^d(x, z) := W^{\bm{\lambda}(d)}(x, d z/d_k) + \bm{\lambda}(d) \, \chi_k(x, d z/d_k)$, we derive from Proposition~\ref{prop:exist-proj-sol} and the fact that the function $\bm{\lambda}$ is smooth that the map $d \mapsto (\tilde{\Psi}_0^d, \tilde{w}^d)$ is smooth from $(d_k, d_k + \ell_k)$ with values into $(\boY_\sigma^\infty(\R) \cap \boW_{0, \sigma}^{2, \infty}(\R)) \times \boW_{0, \sigma}^{2, \infty}(\Omega_{d_k})$ for some number $0 < \sigma < \sigma_0/2$. Since the function $\bm{\lambda}$ remains continuous when $d \to d_k$, the smoothness of this map on $(d_k, d_k + \ell_k)$ extends to a continuity property on $[d_k, d_k + \ell_k)$.

We now conclude the proof of Theorem~\ref{thm:main} by observing that the restrictions of the functions $\Psi_{k, d}$ to the half-strip $\R \times (0, d)$ remain solutions to~\eqref{eq:GL}, but with Neumann boundary conditions. Note also that these restrictions are in $\boX \cap W_\sigma^{2, \infty}$, with a smooth dependence on $d$ in view of the previous analysis of the map $d \mapsto (\tilde{\Psi}_0^d, \tilde{w}^d)$. The fact that $S_0$, $\Psi_{k, d}$ and $\overline{\Psi_{k, d}}$ are the only solutions of the equation in a neighbourhood of $S_0$ in $\boX \cap W_\sigma^{2, \infty}$, follows from the previous description of the vanishing set $\boZ$, which is diffeomorphic to two secant lines. The solutions $\overline{\Psi_{k, d}}$ correspond to the case of negative values for the number $\lambda$.

Note finally that we can derive from~\eqref{ddlambdapsi} and~\eqref{ddlambdaW} that the functions $\Psi_0^\lambda$ and $W^\lambda$ in Proposition~\ref{prop:exist-proj-sol} satisfy, as $d \to d_k$ and $\lambda \to 0$,
\begin{equation}
\label{eq:DL-Psi0-W}
\Psi_0^\lambda = S_0 + \boO \big( |d - d_k|^2 + |\lambda|^2 \big), \quad \text{ and } \quad W^\lambda = \boO \big( |d - d_k|^2 + |\lambda|^2 \big),
\end{equation}
these convergences holding in $W_{0, \sigma}^{2, \infty}(\R)$, respectively $W_{0, \sigma}^{2, \infty}(\Omega_d)$. In view of~\eqref{eq:equiv-lambda} and~\eqref{def:Psi-d}, this can be rephrased as the fact that
\begin{equation}
\label{eq:exp-Psi-d}
\Psi_{k, d}(x, y) = S_0(x) + i \sqrt{\frac{12 \sqrt{2} (d - d_k)}{\omega d_k}} \, \chi_0(x) \, \cos \Big( \frac{\pi k y}{d} \Big) + \boO \big( |d - d_k| \big),
\end{equation}
as $d \to d_k$, these convergences holding in particular in $W^{2, \infty}(\R \times (0,d))$. This is exactly the asymptotic description of the solutions $\Psi_{k, d}$ in Theorem~\ref{thm:main}, with $\Lambda := \sqrt{12 \sqrt{2}/\omega}$. Finally we use these asymptotics in order to expand the Ginzburg-Landau energy $E(\Psi_{k, d})$ as in~\eqref{eq:DL-E}. We refer to Subsection~\ref{sub:DL-E} below for the detailed computations. This concludes the proof of Theorem~\ref{thm:main}. \qed

\subsection{Proof of Corollary~\ref{thm:main}}

The first step in the proof is to describe the possible vortices of the solution $\Psi_{k, d}$ in the regime in which $d$ is close to $d_k$. This description is based on the expansion in~\eqref{eq:exp-Psi-d}. Recall that these asymptotics hold in $W^{2, \infty}(\Omega_d)$, hence are uniform in $\Omega_d$, and the corresponding expansions for the derivatives with respect to $x$ and $y$ also hold uniformly in $\Omega_d$.

Note also that the function $\chi_0$ takes positive values, so that the first two terms in the right-hand side of~\eqref{eq:exp-Psi-d} vanish if and only if $x = 0$ and $y = (2 j + 1) d/{2 k}$, with $- k \leq j \leq k - 1$. In the limit $d \to d_k$, by uniform convergence, the function $\Psi_{k, d}$ cannot vanish except in the neighbourhood of the corresponding $2 k$ points $(0, (2 j + 1) d/{2 k})$. Moreover, the Jacobian matrix of the function $\Psi_{k, d}$ uniformly satisfies the asymptotics
\begin{equation}
\label{eq:Jacob-Psi}
D \Psi_{k, d}(x, y) = \begin{pmatrix} S_0'(x) & 0 \\ \Lambda \sqrt{\frac{d - d_k}{d_k}} \chi_0'(x) \cos \Big( \frac{\pi k y}{d} \Big) & - \frac{k \pi \Lambda}{d} \sqrt{\frac{ d - d_k}{d_k}} \chi_0(x) \sin \Big( \frac{\pi k y}{d} \Big) \end{pmatrix} + \boO \big( |d - d_k| \big).
\end{equation}
As a consequence, it is invertible in the neighbourhoods of these $2 k$ points. Invoking the inverse function theorem is then sufficient to guarantee that the function $\Psi_{k, d}$ has exactly $2 k$ zeroes for $d$ close to $d_k$, which are located close to the $2 k$ points $(0, (2 j + 1) d/{2 k})$, $- k \leq j \leq k - 1$. Note here that due to the symmetry properties of the function $\Psi_{k, d}$ the horizontal component of these zeroes is exactly $0$. Note also that the fact that these zeroes have alternate degrees $\pm 1$ follows from the uniform asymptotic expansion in~\eqref{eq:exp-Psi-d} and~\eqref{eq:Jacob-Psi}.

Restricting our attention to the restriction of the function $\Psi_{1, d}$ to the strip $\R \times (0, d)$, we first claim that the ordinate of its unique zero is exactly $d/2$. Consider indeed the function $R \Psi_{1, d}$ given by~\eqref{def:R}. This function is by construction a solution to~\eqref{eq:GL} with Neumann boundary conditions. Moreover it remains in the neighbourhood of $S_0$ in $\boX \cap W_\sigma^{2, \infty}$, in which the unique solutions are $S_0$, $\Psi_{1, d}$ and $\overline{\Psi}_{1, d}$. Since $R \Psi_{1, d}$ has a unique zero in the strip $\R \times (0, d)$, this function is either equal to $\Psi_{1, d}$, or to $\overline{\Psi}_{1, d}$. By uniqueness, the zero of $R \Psi_{1, d}$ is moreover equal to the one of $\Psi_{1, d}$ and $\overline{\Psi}_{1, d}$. However, the ordinate of the zero of $R \Psi_{1, d}$ must be equal to $d - y_0$, if $y_0$ is the ordinate of the zero of $\Psi_{1, d}$. Hence we have $d - y_0 = y_0$, so that
$$
y_0 = \frac{d}{2}.
$$

In view of~\eqref{eq:exp-Psi-d}, we next check that the zero of the function $\Psi_{1, d}$ has degree $- 1$. Concerning the degree of the zero of the function $R \Psi_{1, d}$, it is by construction equal to the one of the function $\Psi_{1, d}$, whereas the zero of $\overline{\Psi}_{1, d}$ has opposite degree. As a conclusion, the function $R \Psi_{1, d}$ is necessarily equal to the function $\Psi_{1, d}$.

Using this property, we can define the function $\Upsilon_{1, d/k}$ for $k \geq 2$ and $d_k < d < d_k + k \ell_1$ according to~\eqref{def:tilde}, and check that it remains a solution to~\eqref{eq:GL} with Neumann boundary conditions. When $d_k < d < d_k + \ell_k$, this function is in the neighbourhood of $S_0$ in $\boX \cap W_\sigma^{2, \infty}$ in which the unique solutions are $S_0$, $\Psi_{k, d}$ and $\overline{\Psi}_{k, d}$. Considering as before the location and the degree of the zeros of $\Upsilon_{1, d/k}$, we check that this function is equal to the function $\Psi_{k, d}$. As a consequence, the zeroes of this function are located at the points $(0, d_k (2 j + 1)/(2 k))$ for $0 \leq j < k$ and their degrees are equal to $(- 1)^{j + 1}$. This completes the proof of Corollary~\ref{cor:main}. \qed

\subsection{Outline of the paper}

In the next sections, we provide the detail of the proofs of the various lemmas and propositions stated in the previous sketch of the proof of Theorem~\ref{thm:main}. Section~\ref{sec:psi-0} is devoted to the resolution of the first equation in~\eqref{eq:sys-0-w}, and more precisely, to the proof of Lemmas~\ref{lem:coer-S0},~\ref{lem:sol-psi0},~\ref{lem:exp-dec-eta-0},~\ref{lem:control-V-N} and~\ref{lem:control-exp-V-N}. Section~\ref{sec:psi-neq-0} deals with the invertibility properties of the operator $T$ given by Lemma~\ref{lem:invert-T-k}. In Section~\ref{sec:fixed-point} are gathered the proofs of Lemma~\ref{lem:non-linearity} and Proposition~\ref{prop:exist-proj-sol} concerning the main fixed point argument. Finally, Section~\ref{sec:diff-I} provides the detail of the proof of Lemma~\ref{lem:diff-I} regarding the differentiability properties of the function $J$, and of the computation of the expansion of the energy $E(\Psi_{k, d})$ as in~\eqref{eq:DL-E}.

\numberwithin{equation}{section}
\section{Analysis in the zero Fourier sector}
\label{sec:psi-0}

In this first section, we collect the proofs of the results related to the zero Fourier sector, that is dealing with the function $\psi_0$.

\subsection{Proof of Lemma~\ref{lem:coer-S0}}
\label{sub:proof-coer}

The proof of the coercivity estimate in~\eqref{eq:coer-S0} is reminiscent to the proof of~\cite[Proposition 1]{GravSme1}. Consider a function $\psi \in \boX_0(\R)$, and set $\varepsilon = \psi - S_0$ and $\eta = 1 - |\psi|^2 - (1 - S_0^2) = - 2 \langle S_0, \varepsilon \rangle_\C - |\varepsilon|^2$ as in the statement of Lemma~\ref{lem:coer-S0}. Since $S_0$ is a critical point of the energy $\goE$, we can expand the energy $\goE(\psi)$ as
\begin{equation}
\label{eq:decomp-E}
\goE(\psi) = \goE(S_0) + Q_0(\varepsilon_1) + Q_0(\varepsilon_2) + \frac{1}{4} \int_\R \eta^2,
\end{equation}
where the quadratic form $Q_0$ is given by
$$
Q_0(f) = \frac{1}{2} \int_\R \Big( (f')^2 - (1 - S_0^2) f^2 \Big).
$$
When $\psi$ belongs to the set $\boX_0(\R)$, the real part $\varepsilon_1$ and the imaginary part $\varepsilon_2$ of the function $\varepsilon$ are odd, respectively even, and both of them are in the space
$$
\goH_0(\R) := \Big\{ f \in \boC^0(\R, \R) \text{ s.t. } f(0) = 0 \text{ and } \| f \|_{\goH(\R)} < + \infty \Big\},
$$
where $\| f \|_{\goH(\R)}$ is defined in~\eqref{def:H0}.
By the Sobolev embedding theorem, this vector space is a Hilbert space for its natural norm $\| \, \|_{\goH(\R)}$, and the quadratic form $Q_0$ is well-defined and continuous on it. In particular, we can define a self-adjoint operator $\boQ_0$ on $\goH_0(\R)$ such that
$$
Q_0(f) = \langle \boQ_0(f), f \rangle_{\goH(\R)},
$$
for any function $f \in \goH_0(\R)$. Arguing as in the proof of~\cite[Proposition 1]{GravSme1}, we can check that this operator can be written as $\boQ_0 = I/2 - K_0$, where $K_0$ is the self-adjoint non-negative compact operator defined by
$$
\langle K_0(f), g \rangle_{\goH(\R)} = \int_\R (1 - S_0^2) f g,
$$
for any functions $(f, g) \in \goH_0(\R)^2$. As a consequence, we can apply the spectral theorem in order to find a non-decreasing sequence of eigenvalues $\mu_n$ of the operator $\boQ_0$, with $\mu_n \to 1/2$, and a corresponding Hilbert basis $(e_n)_{n \geq 0}$ of $\goH_0(\R)$ such that
\begin{equation}
\label{eq:diago}
\boQ_0(e_n) = \mu_n e_n,
\end{equation}
for any $n \geq 0$.

We next claim that the operator $\boQ_0$ is non-negative so that
$$
\mu_n \geq 0,
$$
for any $n \geq 0$. Consider indeed a function $f \in \goH_0(\R)$, which is of class $\boC^1$ on $\R$. Since $f(0) = 0$, we can continuously extend the function $g := S_0' f/S_0$ to the whole line $\R$ by setting $g(0) = f'(0)$. Since $f$ is in $\goH_0(\R)$, the function $g$ belongs to $L^2(\R)$, and we are allowed to derive from~\eqref{eq:GL} that
$$
\int_\R \Big( f' - \frac{S_0'}{S_0} f \Big)^2 = \int_\R \Big( (f')^2 - 2 \frac{S_0'}{S_0} f f' + \frac{(S_0')^2}{S_0^2} f^2 \Big) = \int_\R \Big( (f')^2 + \frac{S_0''}{S_0} f^2 \Big) = Q_0(f),
$$
by integrating by parts. Hence the quadratic form $Q_0$ is non-negative on the subspace $\goH_0(\R) \cap \boC^1(\R)$. The non-negativity of the operator $\boQ_0$ on $\goH_0(\R)$ then follows from the property that $\goH_0(\R) \cap \boC^1(\R)$ is a dense subspace of $\goH_0(\R)$.

We now claim that the kernel of the operator $\boQ_0$ is spanned by the function $S_0$. Indeed, when a function $f$ belongs to this kernel, it solves the second order differential equation
$$
- f'' - (1 - S_0^2) f = 0,
$$
and it moreover satisfies the initial condition $f(0) = 0$. We observe that $S_0$ is a special solution of this problem. Therefore, it follows from the Cauchy-Lipschitz theorem that the kernel of the operator $\boQ_0$ is spanned by this function.

As a consequence of the two previous claims, we obtain that $\mu_0 = 0$ and $\mu_1 > 0$. Going back to the orthogonal decomposition in~\eqref{eq:diago}, we conclude that
\begin{equation}
\label{eq:coer-Q0}
Q_0(f) \geq \mu_1 \Big\| f - \frac{\langle f, S_0 \rangle_{\goH(\R)}}{\| S_0 \|_{\goH(\R)}^2} S_0 \Big\|_{\goH(\R)}^2,
\end{equation}
for any function $f \in \goH_0(\R)$.

We are now in a position to estimate the various quantities in the decomposition~\eqref{eq:decomp-E}. Since the function $\varepsilon_2$ is even, we first compute
$$
\langle \varepsilon_2, S_0 \rangle_{\goH(\R)} = 2 \int_\R \varepsilon_2 S_0 (1 - S_0^2) = 0.
$$
Hence we deduce from~\eqref{eq:coer-Q0} that
$$
Q_0(\varepsilon_2) \geq \mu_1 \| \varepsilon_2 \|_{\goH(\R)}^2.
$$
In particular, it follows from~\eqref{eq:decomp-E} and the non-negativity of the operator $\boQ_0$ that
\begin{equation}
\label{eq:coer1}
\goE(\psi) - \goE(S_0) \geq \mu_1 \| \varepsilon_2 \|_{\goH(\R)}^2 + \frac{1}{4} \int_\R \eta^2.
\end{equation}

Concerning the real part $\varepsilon_1$, we argue as in the proof of~\cite[Proposition 1]{GravSme1}. We first compute
$$
\frac{1}{4} \int_\R \eta^2 \geq \frac{1}{4} \int_\R (1 - S_0^2) \eta^2 = \int_\R (1 - S_0^2) S_0^2 \varepsilon_1^2 + \int_\R \Big( (1 - S_0^2) S_0 \varepsilon_1 |\varepsilon|^2 + \frac{1}{4} (1 - S_0^2) |\varepsilon|^4 \Big).
$$
Using the identity $S_0' = (1 - S_0^2)/\sqrt{2}$, and invoking the Sobolev embedding theorem for the function $(1 - S_0^2)^{1/2} \varepsilon$, we obtain
$$
\int_\R (1 - S_0^2) S_0 \varepsilon_1 |\varepsilon|^2 = \frac{1}{\sqrt{2}} \int_\R (1 - S_0^2) \Big( \varepsilon_1' \big( 3 \varepsilon_1^2 + \varepsilon_2^2 \big) + 2 \varepsilon_1 \varepsilon_2 \varepsilon_2' \Big) \leq C \big\| \varepsilon \big\|_{\goH(\R)}^3,
$$
for some universal positive number $C$. Hence we have
$$
\frac{1}{4} \int_\R \eta^2 \geq \int_\R (1 - S_0)^2 S_0^2 \varepsilon_1^2 - C \big\| \varepsilon \big\|_{\goH(\R)}^3.
$$
We next deduce from the Cauchy-Schwarz inequality that
$$
\langle \varepsilon_1, S_0 \rangle_{\goH(\R)}^2 \leq 8 \sqrt{2} \int_\R (1 - S_0^2) S_0^2 \varepsilon_1^2.
$$
Since $\| S_0 \|_{\goH(\R)}^2 = 4 \sqrt{2}/3$, we obtain
$$
\frac{1}{4} \int_\R \eta^2 \geq \frac{\langle \varepsilon_1, S_0 \rangle_{\goH(\R)}^2}{6 \| S_0 \|_{\goH(\R)}^2} - C \big\| \varepsilon \big\|_{\goH(\R)}^3.
$$
Combining the non-negativity of the operator $\boQ_0$ with~\eqref{eq:coer-Q0}, we conclude from~\eqref{eq:decomp-E} that
$$
\goE(\psi) - \goE(S_0) \geq \mu_1 \Big\| \varepsilon_1 - \frac{\langle \varepsilon_1, S_0 \rangle_{\goH(\R)}}{\| S_0 \|_{\goH(\R)}^2} S_0 \Big\|_{\goH(\R)}^2 + \frac{\langle \varepsilon_1, S_0 \rangle_{\goH(\R)}^2}{6 \big\| S_0 \big\|_{\goH(\R)}^2} - C \big\| \varepsilon \big\|_{\goH(\R)}^3,
$$
so that
$$
\goE(\psi) - \goE(S_0) \geq \mu \big\| \varepsilon_1 \big\|_{\goH(\R)}^2 - C \big\| \varepsilon \big\|_{\goH(\R)}^3,
$$
for $\mu = \min \{ \mu_1, 1/6 \}$. Going back to~\eqref{eq:coer1}, this gives
$$
\goE(\psi) - \goE(S_0) \geq \frac{\mu}{2} \big\| \varepsilon_1 \big\|_{\goH(\R)}^2 + \frac{\mu_1}{2} \| \varepsilon_2 \|_{\goH(\R)}^2 + \frac{1}{8} \int_\R \eta^2 - \frac{C}{2} \big\| \varepsilon \big\|_{\goH(\R)}^3,
$$
and the coercivity estimate in~\eqref{eq:coer-S0} follows for $\Lambda_0 = \min \{ \mu/2, \mu_1/2, 1/8, 2/C \}$. This completes the proof of Lemma~\ref{lem:coer-S0}. \qed

\subsection{Proof of Lemma~\ref{lem:sol-psi0}}
\label{sub:constr-psi-0}

Consider two positive numbers $\alpha$ and $\beta$ to be fixed later, and assume that the function $f_0 \in \boL^1(\R, \C) \cap \boL^2(\R, \C)$ satisfies the condition in~\eqref{eq:cond-f0}. Recall that the functional $I_{f_0}$ is well-defined on the open subset $\boU(\alpha)$, so that the minimization problem is also well-defined. From~\eqref{eq:cond-f0} and the fact that $\|S_0\|_{L^\infty(\R)} =1$ we deduce the following upper-bound
\begin{equation}
\label{eq:up-boI}
\boI_{f_0}(\alpha) \leq I_{f_0}(S_0) \leq \goE(S_0) + \big\| f_0 \big\|_{L^1(\R)} \leq \goE(S_0) + \beta.
\end{equation}
When $\psi$ is in $\boU(\alpha)$, we derive from~\eqref{eq:estim_f_0} that
\begin{align*}
\bigg| \int_\R \langle f_0, \psi \rangle_\C \bigg| & \leq \big\| f_0 \big\|_{L^1(\R)} + \big\| f_0 \big\|_{L^2(\R)} \Big( \big\| 1- S_0^2 \big\|_{L^2(\R)} + \big\| 1 - |\psi|^2 - (1- S_0^2) \big\|_{L^2(\R)} \Big) \\
& \leq \beta \Big( 1 + \big\| 1 - |\psi|^2 - (1- S_0^2) \big\|_{L^2(\R)} \Big) ,
\end{align*}
so that, by the definition of the distance $d$ in~\eqref{def:d0}, the coercivity estimate in~\eqref{eq:coer-S0}, and the upper bound in~\eqref{eq:up-boI},
\begin{equation}
\label{eq:cuvette}
\begin{split}
I_{f_0}(\psi) & \geq \goE(S_0) + \Lambda_0 \, d(\psi, S_0)^2 - \frac{1}{\Lambda_0} \, d(\psi, S_0)^3 - \beta \Big( 1 + d(\psi, S_0) \Big) \\
& \geq I_{f_0}(S_0) + d(\psi, S_0)^2 \Big( \Lambda_0 - \frac{d(\psi, S_0)}{\Lambda_0} \Big) - \beta \Big( 2 + d(\psi, S_0) \Big).
\end{split}
\end{equation}
Let $\alpha_0 :=\Lambda_0^2/2$. Then
$$
\Lambda_0 - \frac{\alpha_0}{\Lambda_0} \geq \frac{\Lambda_0}{2},
$$
and, therefore, choosing
$$
\beta_0 = {\alpha_0}^2\frac{\Lambda_0}{3(2+\alpha_0)},
$$
we deduce that for any $\beta < \beta_0$, if~\eqref{eq:cond-f0} is satisfied then
\begin{equation}
\label{eq:minint}
d(\psi,S_0) = \alpha_0\quad\Longrightarrow\quad I_{f_0}(\psi)>\boI_{f_0}(\alpha).
\end{equation}

We are now in position to solve the minimization problem $\boI_{f_0}(\alpha_0)$. We consider a minimizing sequence $(\xi_n)_{n \geq 0}$. Since $\xi_n\in \boU(\alpha)$ we have $d \big( \xi_n, S_0 \big) \leq \alpha_0$ for any $n \geq 0$. As a consequence, we can find three functions $(\psi_1, \psi_2, \psi_3) \in L^2(\R, \C)^3$ such that, up to a subsequence,
$$
\xi_n' \rightharpoonup \psi_1, \quad (1 - S_0^2)^\frac{1}{2} \, \xi_n \rightharpoonup \psi_2 \quad \text{ and } \quad 1 - |\xi_n|^2 \rightharpoonup \psi_3 \quad \text{ in } L^2(\R, \C),
$$
as $n \to \infty$. Invoking the Rellich-Kondrachov theorem, we can also exhibit a function $\psi_0 \in \boC^0(\R, \C)$ such that
\begin{equation}
\label{eq:loc-unif-conv}
\xi_n \to \psi_0 \quad \text{ in } \boC_\text{loc}^0(\R, \C).
\end{equation}
In view of the previous convergences, the function $\psi_0$ is actually in $X(\R)$, with $\psi_1 = \psi_0'$, $\psi_2 = (1 - S_0^2)^{1/2} \, \psi_0$, and $\psi_3 = 1 - |\psi_0|^2$. Moreover, we also have
$$
d \big( \psi_0, S_0 \big) \leq \liminf_{n \to \infty} d \big( \xi_n, S_0 \big) \leq \alpha_0.
$$

Similarly, we know that
\begin{equation}
\label{eq:conv-weak-E}
\goE \big( \psi_0 \big) \leq \liminf_{n \to \infty} \goE \big( \xi_n \big),
\end{equation}
and we additionally claim that
\begin{equation}
\label{eq:conv-f-0}
\int_\R \langle f_0, \xi_n \rangle_\C \to \int_\R \langle f_0, \psi_0 \rangle_\C.
\end{equation}
For a fixed positive number $R$, we indeed check as for~\eqref{eq:estim_f_0} that
$$
\bigg| \int_\R \langle f_0, \xi_n - \psi_0 \rangle_\C \bigg| \leq \int_{|x| \leq R} \big| f_0 \big| \, \big| \xi_n - \psi_0 \big| + \int_{|x| \geq R} \big| f_0 \big| \Big( 2 + \big| 1 - |\xi_n|^2 \big| + \big| 1 - |\psi_0|^2 \big| \Big),
$$
so that
$$
\bigg| \int_\R \langle f_0, \xi_n - \psi_0 \rangle_\C \bigg| \leq \big\| f_0 \big\|_{L^1(\R)} \, \big\| \xi_n - \psi_0 \big\|_{L^\infty([- R, R])} + 2 \big\| f_0 \big\|_{L^1([- R, R]^c)} + 2 \alpha_0 \big\| f_0 \big\|_{L^2([- R, R]^c)}.
$$
Due to the local uniform convergence in~\eqref{eq:loc-unif-conv} and to the fact that $f_0$ is in $\boL^1(\R, \C) \cap \boL^2(\R, \C)$, this is sufficient to conclude that~\eqref{eq:conv-f-0} does hold, and to derive from~\eqref{eq:conv-weak-E} that
$$
I_{f_0}(\psi_0) \leq \liminf_{n \to \infty} I_{f_0}(\xi_n) = \boI_{f_0}(\alpha_0).
$$
Hence the function $\psi_0$ is a minimizer of the problem $\boI_{f_0}(\alpha_0)$. From~\eqref{eq:minint}, we then deduce that in fact $d(\psi_0, S_0) < \alpha_0$, so that $\psi_0 \in \boU(\alpha_0)$.

As a consequence of the fact that $\boU(\alpha_0)$ is open, this minimizer is a critical point of the functional $I_{f_0}$, so that it is a weak solution to~\eqref{eq:psi-0}. Since the function $f_0$ is assumed to be continuous, and since $\psi_0$ is also continuous, we conclude from a standard bootstrap argument that the function $\psi_0$ is actually of class $\boC^2$ on $\R$. In particular, it is a classical solution to~\eqref{eq:psi-0}.

To establish~\eqref{eq:dist-S0}, we note that if $d(\psi_0, S_0)^2 > C_0 \beta$, then from~\eqref{eq:cuvette} and our choice of $\alpha_0$ we have
$$
I_{f_0}(\psi_0) - I_{f_0}(S_0) \geq C_0 \beta \, \frac{\Lambda_0}2 - \beta(2+\alpha_0).
$$
Therefore, if $C_0$ is large enough depending on $\Lambda_0$, we obtain a contradiction to the minimality of $I_{f_0}(\psi_0)$. Thus~\eqref{eq:dist-S0} is satisfied, which completes the proof of Lemma~\ref{lem:sol-psi0}. \qed

\subsection{Proof of Lemma~\ref{lem:exp-dec-eta-0}}
\label{sub:exp-decay}

The proof is based on the differential equation for the function $\eta_0$ in~\eqref{eq:N-0}. When the function $f_0$ is in $\boL_\sigma^\infty(\R)$, the functions $\langle f_0, \psi_0 \rangle_\C$ and $g_0$ in this formula decay exponentially. The functions $\psi_0'$ and $N_0$ indeed belong to $H^1(\R)$, so that both the functions $\psi_0$ and $\psi'_0$ are bounded on $\R$, and we can check that
\begin{equation}
\label{eq:exp-estim-h0}
\big\| \langle f_0, \psi_0 \rangle_\C \big\|_{L_\sigma^\infty(\R)} \leq \big\| \psi_0 \big\|_{L^\infty(\R)} \, \big\| f_0 \big\|_{L_\sigma^\infty(\R)},
\end{equation}
as well as
\begin{equation}
\label{eq:exp-estim-g0}
\big\| g_0 \big\|_{L_\sigma^\infty(\R)} \leq \frac{2}{\sigma} \big\| \psi_0' \big\|_{L^\infty(\R)} \, \big\| f_0 \big\|_{L_\sigma^\infty(\R)}.
\end{equation}
The two previous bounds depend on the function $\psi_0$, which is not uniformly controlled at this stage. In order to derive~\eqref{eq:exp-dec-0}, we need to drop this possible dependence. In this direction, we first establish

\begin{step}
\label{S1:exp-decay}
There exists a positive number $M_0$, depending only on $\Lambda_0$, such that
\begin{equation}
\label{eq:unif-psi-0}
\big\| \psi_0 \big\|_{L^\infty(\R)} + \big\| \psi_0' \big\|_{L^\infty(\R)} \leq M_0.
\end{equation}
Moreover, we also have
\begin{equation}
\label{eq:pertub-0-petit}
\big\| \eta_0 \big\|_{L^\infty(\R)} + \big\| \varepsilon_0' \big\|_{L^\infty(\R)} \leq M_0 \beta^\frac{1}{2}.
\end{equation}
\end{step}

The proof is based on the bound for the function $\psi_0$ in~\eqref{eq:dist-S0}, which means that this function is a perturbation of the black soliton $S_0$ of order $\beta^{1/2}$. We first use this property in order to control uniformly the function $\psi_0$. Since $N_0$ is in $H^1(\R)$, we are allowed to compute
\begin{equation}
\label{eq:unif-bound-eta-0}
N_0(x)^2 = 4 \int_x^{+ \infty} N_0(t) \, \langle \psi_0(t), \psi_0'(t) \rangle_\C \, dt \leq 4 \big\| N_0 \big\|_{L^2(\R)} \big\| \psi_0 \big\|_{L^\infty(\R)} \big\| \psi_0' \big\|_{L^2(\R)},
\end{equation}
for any $x \in \R$. Since $N_0 = 1 - |\psi_0|^2$, we derive from~\eqref{eq:dist-S0} that
\begin{align*}
\big\| \psi_0 \big\|_{L^\infty(\R)}^4 & \leq 2 + 2 \big\| N_0 \big\|_{L^\infty(\R)}^2 \\
& \leq 2 + 8 \Big( \big\| 1 - S_0^2 \big\|_{L^2(\R)} + \sqrt{C_0 \beta_0} \Big) \, \Big( \big\| S_0' \big\|_{L^2(\R)} + \sqrt{C_0 \beta_0} \Big) \, \big\| \psi_0 \big\|_{L^\infty(\R)},
\end{align*}
so that
$$
\big\| \psi_0 \big\|_{L^\infty(\R)} \leq M_0,
$$
for a positive number $M_0$, depending only $\Lambda_0$. In view of~\eqref{eq:deriv-eta}, we also have
$$
\big\| \psi_0' \big\|_{L^\infty(\R)} \leq \frac{1}{\sqrt{2}} \big\| N_0 \|_{L^\infty(\R)} + \sqrt{2} \big\| g_0 \|_{L^\infty(\R)}^\frac{1}{2} \leq \frac{1}{\sqrt{2}} \big\| N_0 \|_{L^\infty(\R)} + \sqrt{2} \big\| f_0 \|_{L^1(\R)}^\frac{1}{2} \big\| \psi_0' \big\|_{L^\infty(\R)}^\frac{1}{2},
$$
so that by~\eqref{eq:cond-f0},
$$
\big\| \psi_0' \big\|_{L^\infty(\R)} \leq M_0,
$$
for a further positive number $M_0$, again also depending on $\Lambda_0$.

In order to establish~\eqref{eq:pertub-0-petit}, we argue as in~\eqref{eq:unif-bound-eta-0}. Since $\eta_0' = - 2 \langle \psi_0, \varepsilon_0' \rangle_\C - 2 \langle \varepsilon_0, S_0' \rangle_\C$, with $S_0' = (1 - S_0^2)/\sqrt{2}$, we can compute
$$
\eta_0(x)^2 \leq 2 \big\| \eta_0 \big\|_{L^2(\R)} \Big( 2 \big\| \psi_0 \big\|_{L^\infty(\R)}\, \big\| \varepsilon_0' \big\|_{L^2(\R)} + \sqrt{2} \big\| 1 - S_0^2 \big\|_{L^\infty(\R)}^\frac{1}{2} \, \big\| (1 - S_0^2)^\frac{1}{2} \varepsilon_0 \big\|_{L^2(\R)} \Big),
$$
so that, by~\eqref{eq:dist-S0},
\begin{equation}
\label{eq:unif-N-0}
\big\| \eta_0 \big\|_{L^\infty}^2 \leq 2 \big( 2 M_0 + \sqrt{2} \big) d(\psi_0, S_0)^2 \leq 2 C_0 \big( 2 M_0 + \sqrt{2} \big) \beta.
\end{equation}
Concerning the derivative $\varepsilon_0'$, we recall from~\eqref{eq:dist-S0} that
\begin{equation}
\label{eq:unif-deriv-eps-0}
\big\| \varepsilon_0' \big\|_{L^2}^2 \leq C_0 \beta.
\end{equation}
In view of~\eqref{eq:eps-0}, we also have
$$
\big\| \varepsilon_0'' \big\|_{L^2(\R)} \leq \big\| (1 - S_0^2)^\frac{1}{2} \big\|_{L^\infty(\R)} \, \big\| (1 - S_0^2)^\frac{1}{2} \varepsilon_0 \big\|_{L^2(\R)} + \big\| \psi_0 \big\|_{L^\infty(\R)} \, \big\| \eta_0 \big\|_{L^2(\R)} + \big\| f_0 \big\|_{L^2(\R)},
$$
so that, by~\eqref{eq:cond-f0},~\eqref{eq:dist-S0} and~\eqref{eq:unif-psi-0},
$$
\big\| \varepsilon_0'' \big\|_{L^2(\R)} \leq \Big( \big( M_0 + 1 \big) C_0^\frac{1}{2} + 1 \Big) \beta^\frac{1}{2}.
$$
Combining this inequality with~\eqref{eq:unif-deriv-eps-0}, applying the Sobolev embedding theorem, and then adding~\eqref{eq:unif-N-0} provide~\eqref{eq:pertub-0-petit} for a possibly larger number $M_0$, depending only on $\Lambda_0$.

With the estimates in Step~\ref{S1:exp-decay} at hand, we can invoke~\eqref{eq:N-0} to show that the functions $N_0$ and $\eta_0$ decay exponentially at infinity. We first address the case where the decay rate $\sigma$ is small enough.

\begin{step}
\label{S2:exp-decay}
Assume that $\sigma < \sqrt{2}/2$. There exists a positive number $\tilde{\beta}_0$, such that if $0 < \beta \leq \tilde{\beta}_0$, then the functions $N_0$ and $\eta_0$ are in $L_\sigma^\infty(\R)$, and there exists a positive number $C_\sigma > 0$, depending only on $\sigma$ and $\Lambda_0$, such that
\begin{equation}
\label{eq:exp-eta-0-sigma}
\big\| \eta_0 \big\|_{L_\sigma^\infty(\R)} \leq C_\sigma \Big( \beta^\frac{1}{2} + \big\| f_0 \big\|_{L_\sigma^\infty(\R)} \Big).
\end{equation}
\end{step}

In order to prove Step~\ref{S2:exp-decay}, we first derive from~\eqref{eq:N-0} that the function $\goN_0 := \eta_0^2$ satisfies
$$
- \goN_0'' + \big( 4 - 12 (1- S_0^2) \big) \goN_0 = 2 \eta_0 h_0 + 6 \eta_0 \goN_0 - 2 \big( \eta_0' \big)^2,
$$
where we have set $h_0 := 2 (\langle f_0, \psi_0 \rangle_\C + g_0)$. Going back to~\eqref{eq:pertub-0-petit}, we can bound the first two terms in the right-hand side of this identity by
$$
2 \eta_0 h_0 + 6 \eta_0 \goN_0 \leq \Big( \frac{1}{2} + 6 M_0 \beta^\frac{1}{2} \Big) \goN_0 + 2 h_0^2 \leq \goN_0 + 2 h_0^2,
$$
when $\beta \leq \tilde{\beta}_0 : = 1/(144 \, M_0^2)$. Since $1 - S_0(x)^2 \to 0$ as $x \to \pm \infty$, we can also find a universal positive number $R_0$ such that
$$
12 \, (1- S_0(x)^2) \leq 1,
$$
when $|x| \geq R_0$. Hence we are led to
$$
- \goN_0''(x) + 2 \goN_0(x) \leq 2 h_0(x)^2.
$$
Arguing as for the variation of parameters, we obtain
$$
- \Big( \big( \goN_0(x) e^{\sqrt{2} x} \big)' e^{- 2 \sqrt{2} x} \Big)' \leq 2 h_0(x)^2 e^{- \sqrt{2} x},
$$
for $x \geq R_0$.

We next invoke~\eqref{eq:exp-estim-h0} and~\eqref{eq:exp-estim-g0} in order to check that
$$
2 h_0(x)^2 \leq A_\sigma \, \| f_0 \|_{L_\sigma^\infty(\R)}^2 \, e^{- 2 \sigma |x|},
$$
where $A_\sigma := 8 ( \| \psi_0 \|_{L^\infty(\R)} + \| \psi_0' \|_{L^\infty(\R)}/\sigma )^2$ only depends on $\sigma$ and $\Lambda_0$ by Step~\ref{S1:exp-decay}. This gives
$$
- \Big( \big( \goN_0(x) e^{\sqrt{2} x} \big)' e^{- 2 \sqrt{2} x} \Big)' \leq A_\sigma \, \| f_0 \|_{L_\sigma^\infty(\R)}^2 \, e^{- (\sqrt{2} + 2 \sigma) x}.
$$
Observing that
$$
\big( \goN_0(x) e^{\sqrt{2} x} \big)' e^{- 2 \sqrt{2} x} = \big( \goN_0'(x) + \sqrt{2} \goN_0(x) \big) e^{- \sqrt{2} x} \to 0,
$$
as $x \to + \infty$, we can integrate the previous inequality in order to obtain
$$
\big( \goN_0(x) e^{\sqrt{2} x} \big)' \leq \frac{A_\sigma}{\sqrt{2} + 2 \sigma} \, \| f_0 \|_{L_\sigma^\infty(\R)}^2 \, e^{(\sqrt{2} - 2 \sigma) x}.
$$
Since $\sigma < 1/\sqrt{2}$, and $\goN_0(R_0) \leq M_0^2 \beta$ by~\eqref{eq:pertub-0-petit}, we conclude that
\begin{equation}
\label{eq:bound-goN-0}
\begin{split}
\goN_0(x) & \leq \Big( \goN_0(R_0) e^{\sqrt{2} R_0} + \frac{A_\sigma}{2 - 4 \sigma^2} \, \big( e^{(\sqrt{2} - 2 \sigma) x} - e^{(\sqrt{2} - 2 \sigma) R_0} \big) \, \| f_0 \|_{L_\sigma^\infty(\R)}^2 \Big) e^{- \sqrt{2} x} \\
& \leq \Big( M_0^2 \, \beta \, e^{\sqrt{2} R_0} + \frac{A_\sigma}{2 - 4 \sigma^2} \, \| f_0 \|_{L_\sigma^\infty(\R)}^2 \Big) e^{- 2 \sigma x}.
\end{split}
\end{equation}
A similar estimate holds for $x \leq - R_0$, so that the function $\goN_0$ belongs to $L_{2 \sigma}^\infty([- R_0, R_0]^c)$. Since $\goN_0 = \eta_0^2$, and $\eta_0$ is already known to be continuous on $\R$, the function $\eta_0$, and then $N_0$, are in $L_\sigma^\infty(\R)$. Concerning the bound in~\eqref{eq:exp-eta-0-sigma}, we already derive from~\eqref{eq:bound-goN-0} that
$$
\big\| \eta_0 \big\|_{L_\sigma^\infty([- R_0, R_0]^c)} \leq C_\sigma \Big( \beta^\frac{1}{2} + \| f_0 \|_{L_\sigma^\infty(\R)} \Big),
$$
for some positive number $C_\sigma$ depending only on $\sigma$ and $\Lambda_0$. On the other hand, it follows from~\eqref{eq:pertub-0-petit} that
$$
\big\| \eta_0 \big\|_{L_\sigma^\infty([- R_0, R_0])} \leq \big\| \eta_0 \big\|_{L^\infty(\R)} e^{\sigma R_0} \leq M_0 \, \beta^\frac{1}{2} \, e^{\sigma R_0},
$$
which is sufficient to obtain~\eqref{eq:exp-eta-0-sigma} for a larger positive number $C_\sigma$.

We now deal with the case where the decay rate $\sigma$ is more than $\sqrt{2}/2$.

\begin{step}
\label{S3:exp-decay}
Assume that $\sqrt{2}/2 \leq \sigma < \sqrt{2}$. The functions $N_0$ and $\eta_0$ are in $L_\sigma^\infty(\R)$, and there exists a positive number $C_\sigma > 0$, depending only on $\sigma$ and $\Lambda_0$, such that
$$
\big\| \eta_0 \big\|_{L_\sigma^\infty(\R)} \leq C_\sigma \Big( \beta^\frac{1}{2} + \big\| f_0 \big\|_{L_\sigma^\infty(\R)} \Big)\Big( 1 + \beta^\frac{1}{2} + \big\| f_0 \big\|_{L_\sigma^\infty(\R)} \Big).
$$
\end{step}

The proof is also based on~\eqref{eq:N-0}. For $\sigma < \sqrt{2}$, the function $f_0$ is in $L_{\sigma/2}^\infty(\R)$, with $\sigma/2 < \sqrt{2}/2$. By Step~\ref{S2:exp-decay}, the functions $\eta_0$ and $N_0$ also belong to $L_{\sigma/2}^\infty(\R)$, and $\eta_0$ satisfies~\eqref{eq:exp-eta-0-sigma} with $\sigma$ replaced by $\sigma/2$. In particular, we are allowed to deduce from~\eqref{eq:exp-estim-h0},~\eqref{eq:exp-estim-g0}, and Step~\ref{S1:exp-decay} that
\begin{equation}
\label{eq:exp-nonlin1}
\big\| 3 \eta_0^2 + h_0 \big\|_{L_\sigma^\infty(\R)} \leq 3 \big\| \eta_0 \big\|_{L_\frac{\sigma}{2}^\infty(\R)}^2 + 2 M_0 \Big( 1 + \frac{2}{\sigma} \Big) \big\| f_0 \big\|_{L_\sigma^\infty(\R)}.
\end{equation}
On the other hand, we derive from~\eqref{eq:pertub-0-petit} that
\begin{equation}
\label{eq:exp-nonlin2}
\big\| 6 (1 - S_0^2) \eta_0 \big\|_{L_\sigma^\infty(\R))} \leq 6 \big\| \eta_0 \big\|_{L^\infty(\R)} \big\| 1 - S_0^2 \big\|_{L_{\sqrt{2}}^\infty(\R)} \leq 24 M_0 \beta^\frac{1}{2}.
\end{equation}
Going back to~\eqref{eq:N-0}, we can infer from the variation of parameters that
$$
\eta_0(x) = \eta_0(0) e^{- \sqrt{2} x} + \int_0^x \int_z^{+ \infty} \big( 6 (1 - S_0(t)^2) \eta_0(t) + 3 \eta_0(t)^2 + h_0(t) \big) \, e^{\sqrt{2} (2 z - t - x)} \, dt \,dz,
$$
for any $x \geq 0$. Since $\sigma < \sqrt{2}$, we deduce from~\eqref{eq:pertub-0-petit} that
$$
\big| \eta_0(x) \big| \leq M_0 \, \beta^\frac{1}{2} \, e^{- \sqrt{2} x} + \frac{1}{2 - \sigma^2} \, \big\| 6 (1 - S_0^2) \eta_0 + 3 \eta_0^2 + h_0 \big\|_{L_\sigma^\infty(\R)} \, e^{- \sigma x}.
$$
We can argue similarly for $x \leq 0$ in order to obtain that the function $\eta_0$, and then $N_0$, are in $L_\sigma^\infty(\R)$, with
$$
\big\| \eta_0 \big\|_{L_\sigma^\infty(\R)} \leq 25 M_0 \, \beta^\frac{1}{2} + \frac{3 C_\frac{\sigma}{2}}{2 - \sigma^2} \Big( \beta^\frac{1}{2} + \big\| f_0 \big\|_{L_\sigma^\infty(\R)} \Big)^2 + \frac{2 M_0 (2 + \sigma)}{\sigma (2 - \sigma^2)}\, \big\| f_0 \big\|_{L_\sigma^\infty(\R)},
$$
by~\eqref{eq:exp-nonlin1},~\eqref{eq:exp-nonlin2} and Step~\ref{S2:exp-decay}. This completes the proof of Step~\ref{S3:exp-decay}.

We finally conclude the proof of Lemma~\ref{lem:exp-dec-eta-0} by deriving the exponential decay of the functions $\psi_0'$ and $\varepsilon_0'$.

\begin{step}
\label{S4:exp-decay}
Conclusion.
\end{step}

Our control on the function $\varepsilon_0'$ is based on~\eqref{eq:eps-0}. This formula indeed provides
$$
\big\| \varepsilon_0'' \big\|_{L_\sigma^\infty(\R)} \leq \big\| (1 - S_0^2) \varepsilon_0 \big\|_{L_\sigma^\infty(\R)} + \big\| \psi_0 \big\|_{L^\infty(\R)} \big\| \eta_0 \big\|_{L_\sigma^\infty(\R)} + \big\| f_0 \big\|_{L_\sigma^\infty(\R)}.
$$
Since $\varepsilon(0) = 0$, we can write
$$
\big( 1 - S_0(x)^2 \big) \, \varepsilon_0(x) \leq 4 |x| e^{- \sqrt{2} |x|} \big\| \varepsilon_0' \big\|_{L^\infty(\R)}
$$
for any $x \in \R$, so that by~\eqref{eq:pertub-0-petit},
$$
\big\| (1 - S_0^2) \varepsilon_0 \big\|_{L_\sigma^\infty(\R)} \leq \frac{4 M_0}{e (\sqrt{2} - \sigma)} \beta^\frac{1}{2}.
$$
As a consequence, we deduce again from~\eqref{eq:pertub-0-petit} that
\begin{equation}
\label{eq:unif-second-eps-0}
\big\| \varepsilon_0'' \big\|_{L_\sigma^\infty(\R)} \leq \frac{4 M_0}{e (\sqrt{2} - \sigma)} \beta^\frac{1}{2} + M_0 \big\| \eta_0 \big\|_{L_\sigma^\infty(\R)} + \big\| f_0 \big\|_{L_\sigma^\infty(\R)}.
\end{equation}
By Steps~\ref{S2:exp-decay} and~\ref{S3:exp-decay}, the right-hand side of this inequality is finite, so that the function $\varepsilon_0''$ is in $L_\sigma^\infty(\R)$. In this case, it is integrable on $\R$. Since $\varepsilon_0'(x) \to 0$ as $x \to \pm \infty$, we can write the derivative $\varepsilon_0'$ as
$$
\varepsilon_0'(x) = - \int_x^{+ \infty} \varepsilon_0''(t) \, dt = \int_{- \infty}^x \varepsilon_0''(t) \, dt,
$$
so that $\varepsilon_0'$ is also in $L_\sigma^\infty(\R)$, with
$$
\big\| \varepsilon_0' \big\|_{L_\sigma^\infty(\R)} \leq \frac{1}{\sigma} \, \big\| \varepsilon_0'' \big\|_{L_\sigma^\infty(\R)}.
$$
The estimate for $\varepsilon_0'$ in~\eqref{eq:exp-dec-0} follows (for a further positive number $C_\sigma$) by combining with~\eqref{eq:unif-second-eps-0}, Step~\ref{S2:exp-decay} and Step~\ref{S3:exp-decay}. Finally we also have
$$
\big\| \varepsilon_0 \big\|_{L^\infty(\R)} \leq \big\| \varepsilon_0' \big\|_{L^1(\R)} \leq \frac{2}{\sigma} \, \big\| \varepsilon_0' \big\|_{L_\sigma^\infty(\R)},
$$
due to the property that $\varepsilon(0) = 0$. This provides the estimate for $\varepsilon_0$ in~\eqref{eq:exp-dec-0} (for another positive number $C_\sigma$), and concludes the proof of Lemma~\ref{lem:exp-dec-eta-0}. \qed

\subsection{Proof of Lemma~\ref{lem:control-V-N}}
\label{sub:control-V-N}

The proof is based on equation~\eqref{eq:eps-0} for $\varepsilon_0 = \psi_0 - S_0$. Going back to Lemma~\ref{lem:exp-dec-eta-0}, we recall that the derivative $\varepsilon_0'$ and the function $\eta_0$ are in $L_\sigma^\infty(\R)$. As a consequence, the function $\varepsilon_0$ is bounded on $\R$. Since the functions $1 - S_0^2$ and $f_0$ are also in $L_\sigma^\infty(\R)$, we deduce from~\eqref{eq:eps-0} that the second order derivative $\varepsilon_0''$ belongs to $L_\sigma^\infty(\R)$. In conclusion, the functions $V := \varepsilon_1 - \varepsilon_0$ and $U := (\varepsilon_1 + \varepsilon_0)/2$ are bounded on $\R$, and their first and second order derivatives are in $L_\sigma^\infty(\R)$. Note that the functions $N := \eta_1 - \eta_0$ and $M := (\eta_1 + \eta_0)/2$ also belong to this space.

Invoking again~\eqref{eq:eps-0}, we check that the difference $V$ satisfies the equation
\begin{equation}
\label{eq:V}
- V'' - (1 - S_0^2) V = - F + (S_0 + U) N + V M,
\end{equation}
where we have set $F = f_1 - f_0$. Multiplying this equation by $V$, we are allowed to integrate by parts in order to obtain
$$
\int_\R \Big( |V'|^2 - (1 - S_0^2) |V|^2 \Big) = \int_\R \Big( - \langle F, V \rangle_\C + N \langle S_0 + U, V \rangle_\C + M |V|^2 \Big).
$$
Since $N = \eta_1 - \eta_0 = - 2 \langle S_0 + U, V \rangle_\C$, we can rewrite this identity as
\begin{equation}
\label{eq:int-V}
\int_\R \Big( |V'|^2 - (1 - S_0^2) |V|^2 + \frac{N^2}{2} \Big) = \int_\R \Big( - \langle F, V \rangle_\C + M |V|^2 \Big),
\end{equation}
and we can bound the right-hand side as follows
\begin{equation}
\label{eq:rhs-V}
\int_\R \Big( - \langle F, V \rangle_\C + M |V|^2 \Big) \leq \| F \|_{L_\tau^2(\R)} \| V \|_{L_{- \tau}^2(\R)} + \| M \|_{L_\sigma^\infty(\R)} \| V \|_{L_{- \tau}^2(\R)}^2,
\end{equation}
where $\tau = \sigma/2$.

At this stage, we can bound the norm $\| V \|_{L_{- \tau}^2(\R)}$ by the norm $\| V' \|_{H(\R)}$. This claim follows from the fact that $V(0) = \varepsilon_1(0) - \varepsilon_0(0) = \psi_1(0) - \psi_0(0) = 0$. As a consequence, we can write
$$
V(x) = \int_0^x V'(y) \, dy,
$$
for any $x \in \R$, so that
\begin{equation}
\label{eq:unif-V}
|V(x)|^2 \leq |x| \, \| V' \|_{L^2(\R)}^2.
\end{equation}
This inequality is then enough to establish that
$$
\| V \|_{L_{- \tau}^2(\R)} \leq \frac{1}{\sqrt{2} \tau} \| V' \|_{L^2(\R)} \leq \frac{\sqrt{2}}{\sigma} \| V \|_{H(\R)}.
$$
 Introducing this bound into~\eqref{eq:rhs-V}, and adding the inequality
$$
\| F \|_{L_\tau^2(\R)} \leq \sqrt{\frac{2}{\sigma}} \| F \|_{L_\sigma^\infty(\R)},
$$
we are led to
\begin{equation}
\label{eq:estim-rhs-V}
\int_\R \Big( - \langle F, V \rangle_\C + M |V|^2 \Big) \leq \frac{1}{\kappa \sigma} \| F \|_{L_\sigma^\infty(\R)}^2 + \frac{2}{\sigma^2} \Big( \frac{\kappa}{2} + \| M \|_{L_\sigma^\infty(\R)} \Big) \| V \|_{H(\R)}^2,
\end{equation}
for any positive number $\kappa$.

We now turn to the left-hand side of~\eqref{eq:int-V}. We control it by using the coercivity of the quadratic form $Q_0$ introduced in the proof of Lemma~\ref{lem:coer-S0}. Going back to~\eqref{eq:coer-Q0}, we have
\begin{equation}
\label{eq:coer-V}
\int_\R \Big( |V'|^2 - (1 - S_0^2) |V|^2 \Big) \geq \mu_1 \Big( \Big\| V_1 - \frac{\langle V_1, S_0 \rangle_{\goH(\R)}}{\| S_0 \|_{\goH(\R)}^2} S_0 \Big\|_{\goH(\R)}^2 + \Big\| V_2 - \frac{\langle V_2, S_0 \rangle_{\goH(\R)}}{\| S_0 \|_{\goH(\R)}^2} S_0 \Big\|_{\goH(\R)}^2 \Big),
\end{equation}
where $V_1$ and $V_2$ stand for the real and imaginary parts of $V$, and $\mu_1$ is the first positive eigenvalue of the operator $\boQ_0$. Since the imaginary parts of $\psi_1$ and $\psi_0$ are even, so is the function $V_2$, and we have
\begin{equation}
\label{eq:scal-S0-V2}
\langle V_2, S_0 \rangle_{\goH(\R)} = 0.
\end{equation}
Concerning the real part $V_1$, we argue as in the proof of Lemma~\ref{lem:coer-S0}. We first derive from the formula $N = - 2 S_0 V_1 - 2 \langle U, V \rangle_\C$ and the Cauchy-Schwarz inequality that
$$
\frac{1}{2} \int_{\R^2} N^2 \geq \frac{1}{2} \int_\R (1 - S_0^2) N^2 \geq \int_\R (1 - S_0^2) S_0^2 V_1^2 - 2 \int_\R (1 - S_0^2) \langle U, V \rangle_\C^2.
$$
Using the H\"older inequality, we next write
$$
\int_\R (1 - S_0^2) \langle U, V \rangle_\C^2 \leq \bigg( \int_{\R} (1 - S_0^2) |U|^4 \bigg)^\frac{1}{2} \, \bigg( \int_{\R} (1 - S_0^2) |V|^4 \bigg)^\frac{1}{2}.
$$
Since $U(0) = (\varepsilon_1(0) + \varepsilon_0(0))/2 = (\psi_1(0) + \psi_0(0))/2 - S_0(0) = 0$, we obtain as in the proof of~\eqref{eq:unif-V} that
$$
|U(x)|^4 \leq |x|^2 \, \| U' \|_{L^2(\R)}^4.
$$
Using the inequality $1 - S_0(x)^2 \leq 4 e^{- \sqrt{2} x}$, this gives
$$
\bigg( \int_{\R} (1 - S_0^2) |U|^4 \bigg)^\frac{1}{2} \leq \frac{1}{\sqrt{4 \sqrt{2}}} \, \| U' \|_{L^2(\R)}^2.
$$
Arguing similarly for the function $V$, we obtain
$$
\int_\R (1 - S_0^2) \langle U, V \rangle_\C^2 \leq \frac{1}{4 \sqrt{2}} \, \| U' \|_{L^2(\R)}^2 \, \| V' \|_{L^2(\R)}^2,
$$
which eventually gives the upper bound
$$
\int_\R (1 - S_0^2) \langle U, V \rangle_\C^2 \leq \frac{1}{4 \sqrt{2} \sigma} \, \| U' \|_{L_\sigma^\infty(\R)}^2 \, \| V \|_{H(\R)}^2,
$$
and then the lower bound
\begin{equation}
\label{eq:low-N2}
\frac{1}{2} \int_{\R^2} N^2 \geq \int_\R (1 - S_0^2) S_0^2 V_1^2 - \frac{1}{2 \sqrt{2} \sigma} \, \| U' \|_{L_\sigma^\infty(\R)}^2 \, \| V \|_{H(\R)}^2.
\end{equation}
Recall at this stage that $S_0'' + (1 - S_0^2) = 0$, so that
$$
\langle V_1, S_0 \rangle_{\goH(\R)} = 2 \int_\R V_1 S_0 (1 - S_0^2).
$$
Hence by the Cauchy-Schwarz inequality in $L^2 \big( (1 - S_0(x)^2) \, dx \big)$, we obtain
$$
\langle V_1, S_0 \rangle_{\goH(\R)}^2 \leq 8 \sqrt{2} \int_\R V_1^2 S_0^2 (1 - S_0^2).
$$
Combining with~\eqref{eq:low-N2} finally gives
$$
 \frac{1}{2} \int_\R N^2 \geq \frac{1}{8 \sqrt{2}} \langle V_1, S_0 \rangle_{\goH(\R)}^2 - \frac{1}{2 \sqrt{2} \sigma} \, \| U' \|_{L_\sigma^\infty(\R)}^2 \, \| V \|_{H(\R)}^2.
$$
We next gather this inequality with~\eqref{eq:coer-V} and~\eqref{eq:scal-S0-V2}. Using again the fact that the quadratic form $Q_0$ is non-negative, this provides
\begin{align*}
\int_\R \Big( |V'|^2 - (1 - S_0^2) |V|^2 + \frac{N^2}{2} \Big) \geq & \mu_1 \Big( \Big\| V_1 - \frac{\langle V_1, S_0 \rangle_{\goH(\R)}}{\| S_0 \|_{\goH(\R)}^2} S_0 \Big\|_{\goH(\R)}^2 + \big\| V_2 \big\|_{\goH(\R)}^2 \Big) + \frac{1}{4} \int_\R N^2 \\
& + \frac{\langle V_1, S_0 \rangle_{\goH(\R)}^2}{12 \| S_0 \|_{\goH(\R)}^2} - \frac{1}{4 \sqrt{2} \sigma} \, \| U' \|_{L_\sigma^\infty(\R)}^2 \, \| V \|_{H(\R)}^2.
\end{align*}

In view of~\eqref{eq:int-V} and~\eqref{eq:estim-rhs-V}, we are led to the inequality
$$
\mu \Big( \| V \|_{H(\R)}^2 + \int_\R N^2 \Big) \leq \frac{1}{\kappa \sigma} \| F \|_{L_\sigma^\infty(\R)}^2 + \frac{2}{\sigma^2} \Big( \frac{\kappa}{2} + \| M \|_{L_\sigma^\infty(\R)} + \frac{\sigma}{8 \sqrt{2}} \, \| U' \|_{L_\sigma^\infty(\R)}^2\Big) \| V \|_{H(\R)}^2,
$$
for $\mu = \min \{ \mu_1, 1/12 \}$. It now remains to use Lemma~\ref{lem:exp-dec-eta-0} in order to bound the norms of $M$ and $U'$ in the right-hand side of this estimate. Under the condition in~\eqref{eq:cond-exp-f0}, we first have
$$
\| f_0 \|_{L^1(\R)} + \| f_0 \|_{L^2(\R)} + \| f_1 \|_{L^1(\R)} + \| f_1 \|_{L^2(\R)} \leq \Big( \frac{2}{\sigma} + \frac{1}{\sqrt{\sigma}} \Big) \big( \| f_0 \|_{L_\sigma^\infty(\R)} + \| f_1 \|_{L_\sigma^\infty(\R)} \big) \leq \gamma \Big( \frac{2}{\sigma} + \frac{1}{\sqrt{\sigma}} \Big).
$$
In particular, when $\beta(\gamma) := \gamma (2/\sigma + 1/\sqrt{\sigma}) \leq \beta_\sigma$, we are in position to apply Lemma~\ref{lem:exp-dec-eta-0} in order to obtain
$$
\big\| \eta_0 \big\|_{L_\sigma^\infty(\R)}^2 + \big\| \varepsilon_0' \big\|_{L_\sigma^\infty(\R)}^2 + \big\| \eta_1 \big\|_{L_\sigma^\infty(\R)}^2 + \big\| \varepsilon_1' \big\|_{L_\sigma^\infty(\R)}^2 \leq C_\sigma \Big( \beta(\gamma) + \gamma^2 \Big) \Big( 1 + \beta(\gamma) + \gamma^2 \Big),
$$
where $C_\sigma$ denotes, here as in the sequel, a positive number depending only on $\sigma$. In view of the definition of the functions $M$ and $U$, this gives
$$
\| M \|_{L_\sigma^\infty(\R)} + \| U' \|_{L_\sigma^\infty(\R)} \leq C_\sigma \Big( \beta(\gamma) + \gamma^2 \Big)^\frac{1}{2} \Big( 1 + \beta(\gamma) + \gamma^2 \Big)^\frac{1}{2}.
$$
In particular, we can choose the values of the numbers $\kappa$ and $\gamma_\sigma$ small enough so that, for $\gamma \leq \gamma_\sigma$, we have
$$
\frac{2}{\sigma^2} \Big( \frac{\kappa}{2} + \| M \|_{L_\sigma^\infty(\R)} + \frac{\sigma}{8 \sqrt{2}} \, \| U' \|_{L_\sigma^\infty(\R)}^2\Big) \leq \frac{\mu}{2}.
$$
With this choice, we conclude that
$$
\| V \|_{H(\R)}^2 + \| N \|_{L^2(\R)}^2 \leq \frac{2}{\kappa \mu \sigma} \| F \|_{L_\sigma^\infty(\R)}^2,
$$
which completes the proof of Lemma~\ref{lem:control-V-N} for $A_\sigma = 2/(\kappa \mu \sigma)^{1/2}$. \qed

\subsection{Proof of Lemma~\ref{lem:control-exp-V-N}}
\label{sub:exp-V-N}

The proof follows from controlling the differences $V' := \varepsilon_1' - \varepsilon_0'$ and $N := \eta_1 - \eta_0$ in the spaces $L_\sigma^\infty(\R)$. As in the proof of Lemma~\ref{lem:control-V-N}, this control relies on the differential equations which these functions satisfy. Concerning the function $N$, we derive from~\eqref{eq:N-0} the differential equation
\begin{equation}
\label{eq:N}
- N'' + 2 N = 6 (1 - S_0^2) N + 6 M N + 2 \langle F, S_0 + U \rangle_{\C} + \langle H, V \rangle_{\C} + 2 G_0 := \boN,
\end{equation}
with
\begin{align*}
G_0(x) = & \int_{- \infty}^x \big( \langle F(t), S_0'(t) + U'(t) \rangle_\C + \langle H(t), V'(t) \rangle_\C \big) \, dt \\
= & - \int_x^{+ \infty} \big( \langle F(t), S_0'(t) + U'(t) \rangle_\C + \langle H(t), V'(t) \rangle_\C \big) \, dt,
\end{align*}
for any $x \in \R$. Here, we have set, as above, $U := (\varepsilon_1 + \varepsilon_0)/2$, $M := (\eta_1 + \eta_0)/2$ and $F = f_1 - f_0$, as well as $H := (f_1 + f_0)/2$. We first deal with the function $N$, which we need to control in $H^1(\R)$ before bounding it in $L_\sigma^\infty(\R)$.

\setcounter{step}{0}
\begin{step}
\label{S1:N'}
There exists a positive number $B_1$, depending only on $\sigma$, such that
$$
\big\| N' \big\|_{L^2(\R)} \leq B_1 \big\| f_1 - f_0 \big\|_{L_\sigma^\infty(\R)}.
$$
\end{step}

The proof follows the lines of the proof of Lemma~\ref{lem:control-V-N}. Multiplying~\eqref{eq:N} by the function $N$ and integrating by parts first gives
$$
\int_\R \Big( (N')^2 + 2 N^2 \Big) = \int_\R \Big( 6 (1 - S_0^2) N^2 + 6 M N^2 + 2 \langle F, S_0 + U \rangle_{\C} N + \langle H, V \rangle_{\C} N + 2 G_0 N \Big),
$$
so that
\begin{equation}
\label{eq:cont-N'}
\begin{split}
\| N' \|_{L^2(\R)}^2 \leq \| N \|_{L^2(\R)} \Big( & 6 \| N \|_{L^2(\R)} \big( 1 + \| M \|_{L^\infty(\R)} \big) + 2 \| F \|_{L^2(\R)} \big( 1 + \| U \|_{L^\infty(\R)} \big) \\
& + \| \langle H, V \rangle_{\C} \|_{L^2(\R)} + 2 \| G_0 \|_{L^2(\R)} \Big).
\end{split}
\end{equation}
In this inequality, we observe that
$$
\| M \|_{L^\infty(\R)} \leq \frac{1}{2} \Big( \| \eta_1 \|_{L_\sigma^\infty(\R)} + \| \eta_0 \|_{L_\sigma^\infty(\R)} \Big),
$$
while
$$
\| F \|_{L^2(\R)} \leq \frac{1}{\sqrt{\sigma}} \| F \|_{L_\sigma^\infty(\R)}.
$$
Since $U(0) = (\varepsilon_1(0) + \varepsilon_0(0))/2 = 0$, we also have
$$
U(x) = \int_0^x U'(t) \, dt,
$$
so that
\begin{equation}
\label{eq:U-unif}
\| U \|_{L^\infty(\R)} \leq \frac{1}{\sigma} \| U' \|_{L_\sigma^\infty(\R)} \leq \frac{1}{2 \sigma} \Big( \| \varepsilon_1' \|_{L_\sigma^\infty(\R)} + \| \varepsilon_0' \|_{L_\sigma^\infty(\R)} \Big).
\end{equation}
Similarly, we infer from the formula $V(0) = \varepsilon_1(0) - \varepsilon_0(0) = 0$ that
\begin{equation}
\label{eq:V-sqrt}
|V(x)| \leq |x|^\frac{1}{2} \| V' \|_{L^2(\R)},
\end{equation}
so that
$$
\| \langle H, V \rangle_{\C} \|_{L^2(\R)} \leq \frac{1}{\sqrt{2} \sigma} \| H \|_{L_\sigma^\infty(\R)} \, \| V' \|_{L^2(\R)} \leq \frac{1}{2 \sqrt{2} \sigma} \Big( \| f_1 \|_{L_\sigma^\infty(\R)} + \| f_0 \|_{L_\sigma^\infty(\R)} \Big) \| V \|_{H(\R)}.
$$
Concerning the function $G_0$, we check that
\begin{equation}
\label{eq:G0-unif}
\big| G_0(x) \big| \leq \| F \|_{L_\sigma^\infty(\R)} \Big( \frac{2 \sqrt{2} e^{- (\sigma + \sqrt{2}) |x|}}{\sigma + \sqrt{2}} + \| U' \|_{L_\sigma^\infty(\R)} \frac{e^{- 2 \sigma |x|}}{2 \sigma} \Big) + \| H \|_{L_\sigma^\infty(\R)} \, \| V' \|_{L^2(\R)} \frac{e^{- \sigma |x|}}{\sqrt{2 \sigma}},
\end{equation}
so that
\begin{align*}
\big\| G_0 \big\|_{L^2(\R)} \leq \, & \| F \|_{L_\sigma^\infty(\R)} \Big( \frac{2 \sqrt{2}}{(\sigma + \sqrt{2})^\frac{3}{2}} + \frac{1}{2 (2 \sigma)^\frac{3}{2}} \big( \| \varepsilon_1' \|_{L_\sigma^\infty(\R)} + \| \varepsilon_0' \|_{L_\sigma^\infty(\R)} \big) \Big) \\
& + \frac{1}{2 \sqrt{2} \sigma} \Big( \| f_1 \|_{L_\sigma^\infty(\R)} + \| f_0 \|_{L_\sigma^\infty(\R)} \Big) \| V \|_{H(\R)}.
\end{align*}
Gathering all the previous estimates and using~\eqref{eq:exp-dec-0} and~\eqref{eq:cond-exp-f0}, we can bound~\eqref{eq:cont-N'} by
$$
\| N' \|_{L^2(\R)}^2 \leq B \| N \|_{L^2(\R)} \Big( \| N \|_{L^2(\R)} + \| F \|_{L_\sigma^\infty(\R)} + \| V \|_{H(\R)} \Big),
$$
where $B$ denotes, here as in the sequel, a positive number depending only on $\sigma$. Step~\ref{S1:N'} is then a direct consequence of~\eqref{eq:V-N}.

We next control the function $N$ in $L_\sigma^\infty(\R)$.

\begin{step}
\label{S2:exp-N}
Given any number $0 < \tau < \sigma$, there exists a positive number $B_2$, depending only on $\sigma$ and $\tau$, such that
$$
\big\| N \big\|_{L_\tau^\infty(\R)} \leq B_2 \big\| f_1 - f_0 \big\|_{L_\sigma^\infty(\R)}.
$$
\end{step}

Observe first that the right-hand side $\boN$ in~\eqref{eq:N} exponentially decays at infinity due to the exponential decay of the functions $\varepsilon_1'$, $\varepsilon_0'$, $\eta_1$ and $\eta_0$ in Lemma~\ref{lem:exp-dec-eta-0}, the boundedness of the functions $\varepsilon_1$, $\varepsilon_0$, and the exponential decay of the functions $f_1$ and $f_0$. Applying the variation of parameters as in the proof of Lemma~\ref{lem:exp-dec-eta-0}, we are therefore allowed to derive from~\eqref{eq:N} that
$$
N(x) = \frac{1}{2 \sqrt{2}} \bigg( \int_x^{+ \infty} e^{\sqrt{2}(x - t)} \boN(t) \, dt + \int_{- \infty}^x e^{\sqrt{2}(t - x)} \boN(t) \, dt \bigg),
$$
for any $x \in \R$. In case the function $\boN$ is in $L_\tau^\infty(\R)$, we infer from this formula that
\begin{equation}
\label{eq:N-elliptic}
\| N \|_{L_\tau^\infty(\R)} \leq \frac{1}{2 - \tau^2} \| \boN \|_{L_\tau^\infty(\R)}.
\end{equation}
Hence proving Step~\ref{S2:exp-N} reduces to establish that the function $\boN$ can be controlled by the function $F$ in the space $L_\tau^\infty(\R)$.

In order to do so, we estimate each term in the definition of the function $\boN$. Concerning the two first terms, we observe that
\begin{align*}
\big\| 6 ( 1 - S_0^2 + M) N \big\|_{L_\sigma^\infty(\R)} & \leq 6 \Big( \| 1 - S_0^2 \|_{L_\sigma^\infty(\R)} + \| M \|_{L_\sigma^\infty(\R)} \Big) \| N \|_{L^\infty(\R)} \\
& \leq 6 \Big( 4 + \frac{1}{2} \big( \| \eta_1 \|_{L_\sigma^\infty(\R)} + \| \eta_0 \|_{L_\sigma^\infty(\R)} \big) \Big) \| N \|_{L^\infty(\R)}.
\end{align*}
We next combine Lemma~\ref{lem:exp-dec-eta-0} and Step~\ref{S1:N'} with the Sobolev embedding theorem in order to obtain
\begin{equation}
\label{eq:N-1}
\big\| 6 ( 1 - S_0^2 + M) N \big\|_{L_\sigma^\infty(\R)} \leq B_\sigma \big\| F \big\|_{L_\sigma^\infty(\R)}.
\end{equation}
Here as in the sequel, $B_\sigma$ denotes a positive number depending only on $\sigma$. Similarly, we derive from Lemma~\ref{lem:exp-dec-eta-0} and~\eqref{eq:U-unif} that
\begin{equation}
\label{eq:N-2}
\big\| 2 \langle F, S_0 + U \rangle_{\C} \big\|_{L_\sigma^\infty(\R)} \leq 2 \big( 1 + \| U \|_{L^\infty(\R)} \big) \big\| F \big\|_{L_\sigma^\infty(\R)} \leq B_\sigma \| F \|_{L_\sigma^\infty(\R)}.
\end{equation}
Concerning the term depending on the scalar product $\langle H, V \rangle_{\C}$, we argue as in the proof of Step~\ref{S1:N'}. Using~\eqref{eq:V-sqrt} and the fact that $\tau < \sigma$, we obtain
$$
\big\| \langle H, V \rangle_{\C} \big\|_{L_\tau^\infty(\R)} \leq \big\| H \|_{L_\sigma^\infty(\R)} \, \big\| V \big\|_{L_{\tau - \sigma}^\infty(\R)} \leq B_{\sigma, \tau} \big( \| f_1 \|_{L_\sigma^\infty(\R)} + \| f_0 \|_{L_\sigma^\infty(\R)} \big) \big\| V' \big\|_{L^2(\R)},
$$
where $B_{\sigma, \tau}$ denotes a positive number depending only on $\sigma$ and $\tau$. By condition~\eqref{eq:cond-exp-f0} and Lemma~\ref{lem:control-V-N}, we then get
\begin{equation}
\label{eq:N-3}
\big\| \langle H, V \rangle_{\C} \big\|_{L_\tau^\infty(\R)} \leq B_{\sigma, \tau} \| F \|_{L_\sigma^\infty(\R)}.
\end{equation}
Finally, we bound the term depending on $G_0$ by using~\eqref{eq:G0-unif}. This inequality indeed provides
\begin{align*}
\big\| G_0 \big\|_{L_\sigma^\infty(\R)} \leq \| F \|_{L_\sigma^\infty(\R)} \Big( & \frac{2 \sqrt{2}}{\sigma + \sqrt{2}} + \frac{1}{4 \sigma} \big( \| \varepsilon_1' \|_{L_\sigma^\infty(\R)} + \| \varepsilon_0' \|_{L_\sigma^\infty(\R)} \big) \Big) \\
& + \frac{1}{2 \sqrt{2 \sigma}} \big( \| f_1 \|_{L_\sigma^\infty(\R)} + \| f_0 \|_{L_\sigma^\infty(\R)} \big) \| V' \|_{L^2(\R)},
\end{align*}
so that, by Lemma~\ref{lem:exp-dec-eta-0}, condition~\eqref{eq:cond-exp-f0} and Lemma~\ref{lem:control-V-N},
\begin{equation}
\label{eq:N-4}
\big\| G_0 \big\|_{L_\sigma^\infty(\R)} \leq B \| F \|_{L_\sigma^\infty(\R)}.
\end{equation}
Collecting the estimates in~\eqref{eq:N-1},~\eqref{eq:N-2},~\eqref{eq:N-3} and~\eqref{eq:N-4}, and using the fact that $\tau < \sigma$, we conclude that
$$
\big\| \boN \big\|_{L_\tau^\infty(\R)} \leq B_{\sigma, \tau} \| F \|_{L_\sigma^\infty(\R)},
$$
which is enough to complete the proof of Step~\ref{S2:exp-N}.

We then deal with the function $V'$.

\begin{step}
\label{S3:exp-V'}
Given any number $0 < \tau < \sigma$, there exists a positive number $B_3$, depending only on $\sigma$ and $\tau$, such that
$$
\big\| V' \big\|_{L_\tau^\infty(\R)} \leq B_3 \big\| f_1 - f_0 \big\|_{L_\sigma^\infty(\R)}.
$$
\end{step}

In view of Lemma~\ref{lem:exp-dec-eta-0}, we first observe that the derivative $V'$ is in $L_\tau^\infty(\R)$. In case $V''$ is in $L_\tau^\infty(\R)$, it follows that
$$
V'(x) = - \int_x^{+ \infty} V''(t) \, dt = \int_{- \infty}^x V''(t) \, dt,
$$
so that
\begin{equation}
\label{eq:V'-elliptic}
\big\| V' \big\|_{L_\tau^\infty(\R)} \leq \frac{1}{\tau} \big\| V'' \big\|_{L_\tau^\infty(\R)}.
\end{equation}
As a consequence, it is enough to control the second order derivative $V''$ by the function $F$ in the space $L_\tau^\infty(\R)$. In this direction, we estimate each term in the formula for $V''$ given by~\eqref{eq:V}, that is
$$
V'' = - (1 - S_0^2) V + F - (S_0 + U) N - V M.
$$
For the first one, we write
$$
\big\| (1 - S_0^2) V \big\|_{L_\sigma^\infty(\R)} \leq 4 \big\| V \big\|_{L_{\sigma - \sqrt{2}}^\infty(\R)},
$$
so that by~\eqref{eq:V-sqrt} and Lemma~\ref{lem:control-V-N},
\begin{equation}
\label{eq:V'-1}
\big\| (1 - S_0^2) V \big\|_{L_\sigma^\infty(\R)} \leq \frac{2 \sqrt{2}}{\sqrt{(\sqrt{2} - \sigma) e}} \, \big\| V \big\|_{H(\R)} \leq B_\sigma \big\| F \big\|_{L_\sigma^\infty(\R)}.
\end{equation}
For the third one, we compute
$$
\big\| (S_0 + U) N \big\|_{L_\tau^\infty(\R)} \leq \big( 1 + \| U \|_{L^\infty(\R)} \big) \big\| N \big\|_{L_\tau^\infty(\R)},
$$
and we combine~\eqref{eq:U-unif}, Lemma~\ref{lem:exp-dec-eta-0} and Step~\ref{S2:exp-N} in order to obtain
$$
\big\| (S_0 + U) N \big\|_{L_\tau^\infty(\R)} \leq B_{\sigma, \tau} \big\| F \big\|_{L_\sigma^\infty(\R)}.
$$
Finally, we estimate the last term by
$$
\big\| V M \big\|_{L_\tau^\infty(\R)} \leq \big\| M \big\|_{L_\sigma^\infty(\R)} \, \big\| V \big\|_{L_{\tau - \sigma}^\infty(\R)} \leq \frac{1}{2} \Big( \big\| \eta_0 \big\|_{L_\sigma^\infty(\R)} + \big\| \eta_1 \big\|_{L_\sigma^\infty(\R)} \Big) \big\| V \big\|_{L_{\tau - \sigma}^\infty(\R)}.
$$
Arguing as for the first term and invoking Lemma~\ref{lem:exp-dec-eta-0}, we again have
$$
\big\| V M \big\|_{L_\tau^\infty(\R)} \leq B_{\sigma, \tau} \big\| F \big\|_{L_\sigma^\infty(\R)}.
$$
We finally complete the proof of Step~\ref{S3:exp-V'} by applying the three previous estimates to~\eqref{eq:V}, using the fact that $\tau < \sigma$.

We are now in position to conclude the proof of Lemma~\ref{lem:control-exp-V-N}.

\begin{step}
Conclusion.
\end{step}

We first control uniformly the function $V$. Applying Step~\ref{S3:exp-V'} and the fact that $V(0) = 0$, we indeed have
\begin{equation}
\label{eq:V-unif}
\big\| V \big\|_{L^\infty(\R)} \leq \frac{2}{\sigma} \big\| V' \big\|_{L_\frac{\sigma}{2}^\infty(\R)} \leq B_\sigma \big\| F \big\|_{L_\sigma^\infty(\R)}.
\end{equation}

We next use this estimate to improve the control on the function $N$. In view of~\eqref{eq:V-unif}, we can bound the scalar product $\langle H, V \rangle_{\C}$ in~\eqref{eq:N} by
$$
\big\| \langle H, V \rangle_{\C} \big\|_{L_\sigma^\infty(\R)} \leq \big\| H \|_{L_\sigma^\infty(\R)} \, \big\| V \big\|_{L^\infty(\R)} \leq B_\sigma \big( \| f_1 \|_{L_\sigma^\infty(\R)} + \| f_0 \|_{L_\sigma^\infty(\R)} \big) \big\| F \big\|_{L_\sigma^\infty(\R)},
$$
so that by condition~\eqref{eq:cond-exp-f0},
$$
\big\| \langle H, V \rangle_{\C} \big\|_{L_\sigma^\infty(\R)} \leq B_\sigma \| F \|_{L_\sigma^\infty(\R)}.
$$
Combining this inequality with~\eqref{eq:N-1},~\eqref{eq:N-2} and~\eqref{eq:N-4}, we now conclude that
$$
\big\| \boN \big\|_{L_\sigma^\infty(\R)} \leq B_\sigma \| F \|_{L_\sigma^\infty(\R)},
$$
so that by~\eqref{eq:N-elliptic} (with $\tau$ now equal to $\sigma$),
$$
\big\| N \big\|_{L_\sigma^\infty(\R)} \leq B_\sigma \| F \|_{L_\sigma^\infty(\R)}.
$$

Using this improved estimate, we next control the third term in the formula for $V'$ by
$$
\big\| (S_0 + U) N \big\|_{L_\sigma^\infty(\R)} \leq \big( 1 + \| U \|_{L^\infty(\R)} \big) \big\| N \big\|_{L_\sigma^\infty(\R)},
$$
so that, again by~\eqref{eq:U-unif} and Lemma~\ref{lem:exp-dec-eta-0},
\begin{equation}
\label{eq:V'-3-bis}
\big\| (S_0 + U) N \big\|_{L_\sigma^\infty(\R)} \leq B_\sigma \big\| F \big\|_{L_\sigma^\infty(\R)}.
\end{equation}
Moreover, we also deduce from~\eqref{eq:V-unif} and Lemma~\ref{lem:exp-dec-eta-0} that
$$
\big\| V M \big\|_{L_\sigma^\infty(\R)} \leq \frac{1}{2} \Big( \big\| \eta_0 \big\|_{L_\sigma^\infty(\R)} + \big\| \eta_1 \big\|_{L_\sigma^\infty(\R)} \Big) \big\| V \big\|_{L^\infty(\R)} \leq B_\sigma \big\| F \big\|_{L_\sigma^\infty(\R)}.
$$
Combining with~\eqref{eq:V'-1} and~\eqref{eq:V'-3-bis}, we are led to
$$
\big\| V'' \big\|_{L_\sigma^\infty(\R)} \leq B_\sigma \big\| F \big\|_{L_\sigma^\infty(\R)},
$$
and it follows from~\eqref{eq:V'-elliptic} (with $\tau$ replaced by $\sigma$) that
$$
\big\| V' \big\|_{L_\sigma^\infty(\R)} \leq B_\sigma \big\| F \big\|_{L_\sigma^\infty(\R)}.
$$
Since $\psi_1 - \psi_0 = V$ and $N_1 - N_0 = N$, this concludes the proof of Lemma~\ref{lem:control-exp-V-N}. \qed

\section{Analysis in the higher Fourier sectors}
\label{sec:psi-neq-0}

The analysis is based on the property that the operator $T_0$ defined by
$$
T_0(\psi) := - \psi'' - \psi (1 - |\psi_0|^2) + 2 \langle \psi_0, \psi \rangle_\C \, \psi_0,
$$
is a perturbation of the operator $L_0$ under the assumptions of Lemma~\ref{lem:invert-T-k}. We indeed compute
$$
T_0(\psi) = L_0(\psi) - \eta_0 \, \psi + 2 \langle \varepsilon_0, \psi \rangle_\C \, S_0 + 2 \langle S_0 + \varepsilon_0, \psi \rangle_\C \, \varepsilon_0,
$$
for any function $\psi \in H^2(\R)$, so that
\begin{equation}
\label{eq:cont-T-L-0}
\begin{split}
\big\| T_0(\psi) - L_0(\psi) \big\|_{L^2(\R)} & \leq \Big( \big\| \eta_0 \big\|_{L^\infty(\R)} + 2 \big\| \varepsilon_0 \big\|_{L^\infty(\R)} \big( 2 + \big\| \varepsilon_0 \big\|_{L^\infty(\R)} \big) \Big) \big\| \psi \big\|_{L^2(\R)} \\
& \leq \mu_0 \big( 5 + 2 \mu_0 \big) \big\| \psi \big\|_{L^2(\R)},
\end{split}
\end{equation}
when the function $\psi_0$ satisfies the uniform conditions in~\eqref{eq:unif-eta-eps-0}. Due to the fact that the operator $L_0$ has a unique negative eigenvalue $- 1/2$, it follows from~\eqref{eq:cont-T-L-0} that the invertibility properties of the operators $T_j = T_0 + (\pi j/d)^2$ when $d$ is close to $d_k$ and $\mu_0$ is small enough depend on the fact that $j < k$, $j = k$ or $j > k$. As a consequence, we split the analysis into three steps corresponding to these three conditions. Gathering these three steps eventually provides the statements in Lemma~\ref{lem:invert-T-k}.

\subsection{Analysis of the operators $T_j$ for $j > k$}

The analysis is based on the following lemma.

\begin{lemma}
\label{lem:exp-coer}
Consider a one-variable matrix-valued function $M \in L^\infty(\R, \boM_2(\R))$ and the corresponding bilinear form
$$
\boB(U, V) := \int_{\Omega_d} \Big( \langle \nabla U, \nabla V \rangle_{\C} + \langle M(U), V \rangle_{\C} \Big),
$$
on $H^1(\Omega_d, \C)$.\footnote{In this formula, the complex number $M(U)$ is naturally defined by the property that
$$
\begin{pmatrix} \text{Re}(M(U)) \\ \text{Im}(M(U)) \end{pmatrix} = M \begin{pmatrix} \text{Re}(U) \\ \text{Im}(U) \end{pmatrix}.
$$} For $k \geq 1$, set
$$
X_k := \big\{ \psi \in L^2(\Omega_d, \C) \text{ s.t. } \psi_j = 0 \text{ for any } 0 \leq j \leq k \big\},
$$
and assume that
$$
\boQ(U) := \boB(U, U) \geq \kappa \|U\|_{H^1}^2
$$
for any $U\in \boH^1(\Omega_d, \C) \cap X_k$ and some $\kappa > 0$.

Then, given any function $F \in \boX_k$, there exists a unique function $U \in \boH^2(\Omega_d, \C) \cap X_k$ such that
\begin{equation}
\label{eq:M}
- \Delta U + M(U) = F.
\end{equation}
Moreover, there exists a number $C > 0$, depending only on $\kappa$ and the uniform norm of the function $M$, such that
\begin{equation}
\label{up:M-L2}
\big\| U \|_{H^2(\Omega_d)} \leq C \, \big\| F \big\|_{L^2(\Omega_d)}.
\end{equation}
When the function $F$ is additionally in $L_\sigma^\infty(\Omega_d, \C)$ for some number $\sigma > 0$, the solution $U$ is in $L_\tau^\infty(\Omega_d, \C)$
for any $0 < \tau < \min\{ \sigma, \sqrt{\kappa} \}$, with
\begin{equation}
\label{up:M}
\big\| U \|_{L_\tau^\infty(\Omega_d)} \leq C \, \big\| F \big\|_{L_\sigma^\infty(\Omega_d)},
\end{equation}
for a further number $C > 0$, depending only on $d$, $\tau$, $\kappa - \tau^2$, $\sigma - \tau$ and the uniform norm of the function $M$.
\end{lemma}

\begin{proof}
Observe that $\boH^1(\Omega_d, \C) \cap X_k$ is naturally endowed with an Hilbert space structure as a closed subspace of $\boH^1(\Omega_d, \C)$. The bilinear form $\boB$ is coercive on this subspace and also continuous by boundedness of the function $M$. When $F$ is in $\boX_k$, we can apply the Lax-Milgram theorem in order to find a unique function $U \in \boH^1(\Omega_d, \C) \cap X_k$ such that
\begin{equation}
\label{eq:U-def}
\boB(U, V) = \int_{\Omega_d} \langle F, V \rangle_\C,
\end{equation}
for any $V \in \boH^1(\Omega_d, \C) \cap X_k$. Since the Laplacian operator and the multiplication operator by the matrix $M$ stabilize the subspace $X_k$, this identity is sufficient to establish that the function $U$ is a solution to~\eqref{eq:M}. The estimate in~\eqref{up:M-L2} then follows from standard elliptic theory.

Assume additionally that $F$ is in $L_\sigma^\infty(\Omega_d, \C)$. We can introduce an even bounded Lipschitz function $f : \R \to \R$ and set $\tilde{U}(x, y) = e^{f(x)} U(x, y)$, as well as $\tilde{V}(x, y) = e^{2 f(x)} U(x, y)$. Since the multiplication operators by the functions $e^f$ and $e^{2 f}$ also stabilize the subspace $X_k$, we deduce from the bounded Lipschitz nature of the function $f$ that the functions $\tilde{U}$ and $\tilde{V}$ are in $\boH^1(\Omega_d, \C) \cap X_k$. Inserting the function $\tilde{V}$ in~\eqref{eq:U-def}, we obtain
$$
\int_{\Omega_d} \Big( \big| \nabla \tilde{U} \big|^2 - (f')^2 \big| \tilde{U} \big|^2 + \big\langle M(\tilde{U}), \tilde{U} \big\rangle_\C \Big) = \boB \big( U, \tilde{V} \big) = \int_{\Omega_d} \langle \tilde{F}, \tilde{U} \rangle_\C,
$$
with $\tilde{F}(x, y) = e^{f(x)} F(x, y)$. Invoking the coercivity of the bilinear form $\boB$ on $\boH^1(\Omega_d, \C) \cap X_k$, and the Lipschitz nature of the function $f$, we are led to
\begin{equation}
\label{eq:tilde-U}
\big( \kappa - \| f' \|_{L^\infty(\R)}^2 \big) \, \big\| \tilde{U} \big\|_{H^1(\Omega_d)} \leq \big\| \tilde{F} \big\|_{L^2(\Omega_d)}.
\end{equation}

Now assume that $F$ is in $L_\sigma^\infty(\Omega_d, \C)$ and let $f(x) = f_n(x) = \min \{ \tau |x|, n\}$ for $0 < \tau <\sigma$ and some integer $n$. In this case, we have
\begin{equation}
\label{L2inf}
\big\| \tilde{F} \big\|_{L^2(\Omega_d)} \leq \big\| F \big\|_{L_\tau^2(\Omega_d)} \leq \sqrt{\frac{2 d}{\sigma - \tau}} \, \big\| F \big\|_{L_\sigma^\infty(\Omega_d)}.
\end{equation}
When $\tau^2 < \kappa$, we infer from~\eqref{eq:tilde-U} that
$$
\big\| \tilde{U} \big\|_{H^1(\Omega_d)} \leq \frac{\sqrt{2 d}}{(\kappa - \tau^2) \sqrt{\sigma - \tau}} \, \big\| F \big\|_{L_\sigma^\infty(\Omega_d)},
$$
and we can take the limit $n \to \infty$ in order to check that the function $U$, as well its partial derivatives $\partial_x U$ and $\partial_y U$, are in $L_\tau^2(\Omega_d)$, with
\begin{equation}
\label{eq:U-L-tau-2}
\big\| U \big\|_{L_\tau^2(\Omega_d)}^2 + \big\| \partial_x U \big\|_{L_\tau^2(\Omega_d)}^2 + \big\| \partial_y U \big\|_{L_\tau^2(\Omega_d)}^2 \leq \frac{2 d}{(\kappa - \tau^2)^2 (\sigma - \tau)} \, \big\| F \big\|_{L_\sigma^\infty(\Omega_d)}^2.
\end{equation}
Next we deduce from~\eqref{eq:M} and~\eqref{L2inf} that $\Delta U$ is also in $L_\tau^2(\Omega_d)$, with
\begin{equation}
\label{eq:Delta-U-L-tau-2}
\big\| \Delta U \big\|_{L_\tau^2(\Omega_d)} \leq \Big( 1 + \frac{\sqrt{2 d} \, \| M \|_{L^\infty(\R)}}{(\kappa - \tau^2) \sqrt{\sigma - \tau}} \Big) \, \big\| F \big\|_{L_\sigma^\infty(\Omega_d)}.
\end{equation}
The two previous estimates are sufficient to establish that the function $\check{U}(x, y) = e^{\tau \sqrt{x^2 + 1}} \linebreak[0] U(x, y)$ is in $H^2(\Omega_d, \C)$, so that by the Sobolev embedding theorem, it is bounded on $\Omega_d$. This shows that the function $U$ is indeed in $L_\tau^\infty(\Omega_d)$, and the estimate in~\eqref{up:M} follows from~\eqref{eq:U-L-tau-2} and~\eqref{eq:Delta-U-L-tau-2}.
\end{proof}

We apply Lemma~\ref{lem:exp-coer} to the self-adjoint bilinear form $\boB$ corresponding to the operator $T$, namely
\begin{equation}
\label{def:B-T}
\boB \big( \psi_1, \psi_2 \big) := \int_{\Omega_d} \Big( \langle \nabla \psi_1, \nabla \psi_2 \rangle_\C - \big( 1 - |\psi_0|^2 \big) \langle \psi_1, \psi_2 \rangle_\C + 2 \langle \psi_0, \psi_1 \rangle_\C \, \langle \psi_0, \psi_2 \rangle_\C \Big),
\end{equation}
for any functions $(\psi_1, \psi_2) \in H^1(\Omega_d, \C)^2$. The matrix-valued function $M$ in this formula is given by
$$
M = \begin{pmatrix} |\psi_0|^2 + 2 \text{Re}(\psi_0)^2 - 1 & 2 \text{Re}(\psi_0) \text{Im}(\psi_0) \\ 2 \text{Re}(\psi_0) \text{Im}(\psi_0) & |\psi_0|^2 + 2 \text{Im}(\psi_0)^2 - 1 \end{pmatrix}.
$$
When the function $\psi_0 \in \boU(\mu_0)$ satisfies the condition in~\eqref{eq:unif-eta-eps-0}, the function $M$ is uniformly bounded on $\Omega_d$ and its uniform norm only depends on the number $\mu_0$. In particular the self-adjoint bilinear form $\boB$ is continuous on $H^1(\Omega_d, \C)$ (with a norm depending only on $\mu_0$).

Moreover it follows from the Parseval formula that, for any $\psi\in X_k$,
$$
\int_{\Omega_d} |\partial_y \psi|^2 \geq \frac{\pi^2 (k + 1)^2}{d^2} \int_{\Omega_d} |\psi|^2.
$$
Arguing as for~\eqref{eq:cont-T-L-0}, we are led to
\begin{equation}
\label{eq:cont-B-j-0}
\boB(\psi, \psi) \geq \boQ_0(\psi) + \Big( \frac{\pi^2 (k + 1)^2}{d^2} - \mu_0 \big( 5 + 2 \mu_0 \big) \Big) \big\| \psi \big\|_{L^2(\Omega_d)}^2,
\end{equation}
where
$$
\boQ_0(\psi) := \int_{\Omega_d} \Big( |\partial_x \psi|^2 - \big( 1 - |S_0|^2 \big) |\psi|^2 + 2 \langle S_0, \psi \rangle_\C^2 \Big),
$$
stands for the quadratic form (on $H^1(\Omega_d, \C)$) corresponding to the operator $L_0$. Since this operator has a unique negative eigenvalue equal to $- 1/2$, we deduce from~\eqref{def:d-k} that
$$
\boQ_0(\psi) \geq - \frac{1}{2} \big\| \psi \big\|_{L^2(\Omega_d)}^2 = - \frac{\pi^2 k^2}{d_k^2} \big\| \psi \big\|_{L^2(\Omega_d)}^2.
$$
Hence we have
$$
\boB(\psi, \psi) \geq \Big( \frac{\pi^2 (k + 1)^2}{d^2} - \frac{\pi^2 k^2}{d_k^2} - \mu_0 \big( 5 + 2 \mu_0 \big) \Big) \big\| \psi \big\|_{L^2(\Omega_d)}^2.
$$
On the other hand we also deduce from applying~\eqref{eq:unif-eta-eps-0} to~\eqref{def:B-T} that
\begin{equation}
\label{eq:defect-coer-B}
\boB(\psi, \psi) \geq \big\| \nabla \psi \big\|_{L^2(\Omega_d)}^2 - (1 + \mu_0) \big\| \psi \big\|_{L^2(\Omega_d)}^2,
\end{equation}
so that, for any number $0 < \theta < 1$,
$$
\boB(\psi, \psi) \geq \theta \big\| \nabla \psi \big\|_{L^2(\Omega_d)}^2 + \Big( (1 - \theta) \Big( \frac{\pi^2 (k + 1)^2}{d^2} - \frac{\pi^2 k^2}{d_k^2} - \mu_0 \big( 5 + 2 \mu_0 \big) \Big) - \theta \big( 1 + \mu_0 \big) \Big) \big\| \psi \big\|_{L^2(\Omega_d)}^2.
$$
At this stage, we notice that
$$
(1 - \theta) \Big( \frac{\pi^2 (k + 1)^2}{d^2} - \frac{\pi^2 k^2}{d_k^2} - \mu_0 \big( 5 + 2 \mu_0 \big) \Big) - \theta \big( 1 + \mu_0 \big) \to (1 - \theta) \frac{(2 k + 1) \pi^2}{d_k^2} - \theta,
$$
as $d \to d_k$ and $\mu_0 \to 0$. For $\theta$ small enough, we obtain that, for $d$ close enough to $d_k$ and $\mu_0$ small enough,
$$
\boB(\psi, \psi) \geq \min \Big\{ \frac{(2 k + 1) \pi^2}{2 d_k^2},\theta \Big\} \, \big\| \psi \big\|_{H^1(\Omega_d)}^2.
$$
This coercivity estimate allows us to apply Lemma~\ref{lem:exp-coer} to conclude that, if the hypothesis of Lemma~\ref{lem:invert-T-k} are satisfied, then its conclusions hold under the additional assumption that $g\in X_k$.

Note here that, when a function $\psi_0 \in \boU(\mu_0)$ satisfies the condition in~\eqref{eq:cond-eta-eps-0} for a number $0 < \sigma < \sqrt{2}$, we have
\begin{equation}
\label{eq:control-mu-nu}
\| \eta_0 \|_{L^\infty(\R)} \leq \| \eta_0 \|_{L_\sigma^\infty(\R)} \leq \nu, \quad \text{ and } \quad \| \varepsilon_0 \|_{L^\infty(\R)} \leq \nu.
\end{equation}
In particular, we can decrease the value of the number $\nu_0$ in Lemma~\ref{lem:invert-T-k} so that the function $\psi_0$ satisfies the condition in~\eqref{eq:unif-eta-eps-0} for any $0 < \nu \leq \nu_0$. Moreover, the dependence on the number $\mu_0$ in the previous computations can be replaced by a dependence on the number $\nu_0$ as in the statement of Lemma~\ref{lem:invert-T-k}.

\subsection{Analysis of the operator $T_k$}

For $k \geq 1$, we set
$$
Z_k := \big\{ \psi \in L^2(\Omega_d, \C) \text{ s.t. } \psi_\ell = 0 \text{ for any } \ell \neq k \big\},
$$
and we focus on the restriction of the self-adjoint bilinear form $\boB$ to the space $\boH^1(\Omega_d, \C) \cap Z_k$. Recall that this bilinear form is continuous on $H^1(\Omega_d, \C)$, with norm depending only on $\mu_0$. Arguing as for~\eqref{eq:cont-B-j-0}, we obtain
$$
\boB(\phi, \phi) \geq \boQ_0(\phi) + \Big( \frac{\pi^2 k^2}{d^2} - \mu_0 \big( 5 + 2 \mu_0 \big) \Big) \big\| \phi \big\|_{L^2(\Omega_d)}^2,
$$
for any function $\phi \in \boH^1(\Omega_d, \C) \cap Z_k$. When $\phi$ is also in $H_k$, the quadratic form $\boQ_0(\phi)$ is non-negative, so that
$$
\boB(\phi, \phi) \geq \Big( \frac{\pi^2 k^2}{d^2} - \mu_0 \big( 5 + 2 \mu_0 \big) \Big) \big\| \phi \big\|_{L^2(\Omega_d)}^2.
$$
Using~\eqref{eq:defect-coer-B}, we obtain
\begin{equation}
\label{eq:coer-B-k}
\boB(\phi, \phi) \geq \frac{1}{4} \big\| \nabla \phi \big\|_{L^2(\Omega_d)}^2 + \frac{1}{4} \Big( \frac{3 \pi^2 k^2}{d^2} - 1 - 16 \mu_0 - 6 \mu_0^2 \Big) \big\| \phi \big\|_{L^2(\Omega_d)}^2 \geq \frac{1}{16} \big\| \phi \big\|_{H^1(\Omega_d)}^2,
\end{equation}
for $d$ close enough to $d_k$ and $\mu_0$ small enough. Given a function $g \in \boL^2(\Omega_d, \C) \cap Z_k$, we can therefore invoke the Lax-Milgram theorem in order to find a unique function $\psi \in \boH^1(\Omega_d, \C) \cap Z_k \cap H_k$ such that
\begin{equation}
\label{eq:sol-psi-k}
\boB(\psi, \phi) = \int_{\Omega_d} \langle g, \phi \rangle_\C, \quad \forall \phi \in \boH^1(\Omega_d, \C) \cap Z_k \cap H_k.
\end{equation}
This identity can be rephrased as the fact that there exists a unique function $\psi \in \boH^1(\Omega_d, \C) \cap Z_k \cap H_k$ such that
\begin{equation}
\label{eq:eq-psi-k}
- \Delta \psi - \big( 1 - |\psi_0|^2 \big) \psi + 2 \langle \psi_0, \psi \rangle_\C \, \psi_0 - g = \frac{\chi_k}{\| \chi_k \|_{L^2(\Omega_d)}^2} \bigg( \boB(\psi, \chi_k) - \int_{\Omega_d} \langle g, \chi_k \rangle_\C \bigg).
\end{equation}
By standard elliptic regularity theory, we check that the function $\psi$ is actually in $H^2(\Omega_d, \C)$, and the previous equation can be expressed as
$$
\pi_k \big( T(\psi) - g \big) = 0.
$$
It then follows from~\eqref{eq:coer-B-k},~\eqref{eq:sol-psi-k},~\eqref{eq:eq-psi-k} and standard elliptic regularity theory that
\begin{equation}
\label{eq:estim-inv-boT-k}
\big\| \psi \big\|_{H^2(\Omega_d)} \leq C \big\| g \big\|_{L^2(\Omega_d)},
\end{equation}
for some number $C$ depending only on $\mu_0$.

When the function $g$ is additionally in $L_\sigma^\infty(\Omega_d)$, we argue as in the proof of Lemma~\ref{lem:exp-coer} in order to control the function $\psi$ in $L_\tau^\infty(\Omega_d)$ for any $0 < \tau < \sigma$. We introduce an even bounded Lipschitz function $f : \R \to \R$, and consider the functions $\tilde{\psi}(x, y) := e^{f(x)} \psi(x, y)$ and $\xi(x, y) := e^{2 f(x)} \psi(x, y)$, which both remain in $\boH^1(\Omega_d, \C) \cap Z_k$. We also introduce the numbers
$$
\lambda_k := \frac{\langle \tilde{\psi}, \chi_k \rangle_{L^2(\Omega_d)}}{\| \chi_k \|_{L^2(\Omega_d)}^2}, \quad \text{ and } \quad \mu_k := \frac{\langle \xi, \chi_k \rangle_{L^2(\Omega_d)}}{\| \chi_k \|_{L^2(\Omega_d)}^2},
$$
so that $\pi_k(\tilde{\psi}) = \tilde{\psi} - \lambda_k \chi_k$ and $\pi_k(\xi) = \xi - \mu_k \chi_k$. Inserting the function $ \pi_k(\xi) = \xi - \mu_k \chi_k$ into~\eqref{eq:sol-psi-k}, we obtain
$$
\boB(\psi, \xi) = \mu_k \boB(\psi, \chi_k) + \int_{\Omega_d} \langle g, \xi \rangle_\C - \mu_k \int_{\Omega_d} \langle g, \chi_k \rangle_\C.
$$
We next check that
$$
\boB(\psi, \xi) = \boB \big( \tilde{\psi}, \tilde{\psi} \big) - \int_{\Omega_d} \big( f' \big)^2 \big| \tilde{\psi} \big|^2,
$$
so that
$$
\boB \big( \tilde{\psi}, \tilde{\psi} \big) = \mu_k \boB(\psi, \chi_k) + \int_{\Omega_d} \big( f' \big)^2 \big| \tilde{\psi} \big|^2 + \int_{\Omega_d} \big\langle e^f g, \tilde{\psi} \big\rangle_\C - \mu_k \int_{\Omega_d} \langle g, \chi_k \rangle_\C.
$$
Since
$$
\frac{1}{16} \big\| \pi_k(\tilde{\psi}) \big\|_{H^1(\Omega_d)}^2 \leq \boB \big( \pi_k(\tilde{\psi}), \pi_k(\tilde{\psi}) \big) = \boB \big( \tilde{\psi}, \tilde{\psi} \big) - 2 \lambda_k \boB \big( \tilde{\psi}, \chi_k \big) + \lambda_k^2 \boB \big( \chi_k, \chi_k \big),
$$
by~\eqref{eq:coer-B-k}, we obtain
\begin{align*}
\frac{1}{16} \big\| \tilde{\psi} \big\|_{H^1(\Omega_d)}^2 \leq & \frac{1}{16} \Big( \big\| \pi_k(\tilde{\psi}) \big\|_{H^1(\Omega_d)}^2 + 2 \lambda_k \big\langle \tilde{\psi}, \chi_k \big\rangle_{H^1(\Omega_d)} - \lambda_k^2 \big\| \chi_k \big\|_{H^1(\Omega_d)}^2 \Big) \\
\leq & \mu_k \boB(\psi, \chi_k) + \int_{\Omega_d} \big( f' \big)^2 \big| \tilde{\psi} \big|^2 + \int_{\Omega_d} \big\langle e^f g, \tilde{\psi} \big\rangle_\C - \mu_k \int_{\Omega_d} \langle g, \chi_k \rangle_\C \\
& - 2 \lambda_k \boB \big( \tilde{\psi}, \chi_k \big) + \lambda_k^2 \boB \big( \chi_k, \chi_k \big) + \frac{\lambda_k}{8} \big\langle \tilde{\psi}, \chi_k \big\rangle_{H^1(\Omega_d)} - \frac{\lambda_k^2}{16} \big\| \chi_k \big\|_{H^1(\Omega_d)}^2.
\end{align*}
At this stage we first check that
$$
\big| \lambda_k \big| \leq \frac{\| \psi \|_{L^2(\Omega_d)} \, \| e^f \chi_k \|_{L^2(\Omega_d)}}{\| \chi_k \|_{L^2(\Omega_d)}^2}, \quad \text{ and } \quad \big| \mu_k \big| \leq \frac{\| \psi \|_{L^2(\Omega_d)} \, \| e^{2 f} \chi_k \|_{L^2(\Omega_d)}}{\| \chi_k \|_{L^2(\Omega_d)}^2}.
$$
From~\eqref{eq:unif-eta-eps-0} and~\eqref{def:B-T} we also have, for any $\phi_1$, $\phi_2\in H^1(\Omega_d,\C)$,
$$
\big| \boB(\phi_1, \phi_2) \big| \leq \big( 3 + 5 \mu_0 + 2 \mu_0^2 \big) \, \big\| \phi_1 \big\|_{H^1(\Omega_d)} \, \big\| \phi_2 \big\|_{H^1(\Omega_d)},
$$
Combining these estimates with~\eqref{eq:estim-inv-boT-k}, and using the Young inequality $2 a b \leq a^2 + b^2$, we can find a number $C$, depending only and (continuously) on $k$ and $\mu_0$, such that
$$
\Big( \frac{1}{32} - \big\| f' \big\|_{L^\infty(\R)}^2 \Big) \big\| \tilde{\psi} \big\|_{H^1(\Omega_d)}^2 \leq C \bigg( \big\| e^f g \big\|_{L^2(\Omega_d)}^2 + \big\| g \big\|_{L^2(\Omega_d)}^2 \Big( \big\| e^f \chi_k \big\|_{L^2(\Omega_d)}^2 + \big\| e^{2 f} \chi_k \big\|_{L^2(\Omega_d)} \Big) \bigg).
$$
Assuming that $0 < \sigma_0 < 1/(2 \sqrt{2})$, and letting $f(x) = f_n(x) = \min \{ \tau |x|, n\}$ for $0 < \tau <\sigma$ and some integer $n$, we observe that
$$
\big\| e^f \chi_k \big\|_{L^2(\Omega_d)}^2 + \big\| e^{2 f} \chi_k \big\|_{L^2(\Omega_d)} \leq C,
$$
for a further number $C$ depending only on $k$ and $\sigma_0$. Arguing as for~\eqref{L2inf}, we are led to
$$
\Big( \frac{1}{32} - \tau^2 \Big) \big\| \tilde{\psi} \big\|_{H^1(\Omega_d)}^2 \leq C \, \big\| e^f g \big\|_{L^2(\Omega_d)}^2 \leq \frac{2 C d}{\sigma - \tau} \, \big\| g \big\|_{L_\sigma^\infty(\Omega_d)}^2.
$$
Decreasing if necessary the value of $\sigma_0$ so that $\sigma_0^2 < 1/32$, we conclude that
$$
\big\| \tilde{\psi} \big\|_{H^1(\Omega_d)}\leq C \, \big\| g \big\|_{L_\sigma^\infty(\Omega_d)},
$$
the number $C$ now depending also on $\delta_1$, $\sigma$ and $\tau$. Taking the limit $n \to \infty$, we deduce that
$$
\big\| \psi \big\|_{L_\tau^2(\Omega_d)} + \big\| \nabla \psi \big\|_{L_\tau^2(\Omega_d)} \leq C \, \big\| g \big\|_{L_\sigma^\infty(\Omega_d)}.
$$
Since $\psi$ is in $Z_k$, we finally invoke the one-variable Sobolev embedding theorem in order to conclude that
\begin{equation}
\label{eq:estim-inv-boT-exp-k}
\big\| \psi \big\|_{L_\tau^\infty(\Omega_d)} \leq C \big\| g \big\|_{L_\sigma^\infty(\Omega_d)}.
\end{equation}
Note that we can argue as in~\eqref{eq:control-mu-nu} in order to replace the dependence on $\mu_0$ in the constant $C$ by a dependence on $\nu_0$.

The estimates~\eqref{eq:estim-inv-boT-k} and~\eqref{eq:estim-inv-boT-exp-k} show that the conclusion of Lemma~\ref{lem:invert-T-k} hold under the additional assumption that $g\in Z_k$.

\subsection{Analysis of the operators $T_j$ for $j < k$}

When $j<k$, the analysis is a little more involved due to the fact that the operator $T_j$ has a negative eigenvalue. Fix a number $1 \leq j < k$, and consider a function $g_j \in \boL^2(\R, \C)$. In order to solve the equation $T_j(\psi_j) = g_j$, we decompose the function $g_j$ as $g_j = i \beta \chi_0 + \pi_0(g_j)$, with $\beta \in \R$. Here the notation $\pi_0$ refers to the orthogonal projection on $H_0$, which is given by
$$
\pi_0(f) := f - \frac{\langle f, i \chi_0 \rangle_{L^2(\R)}}{\| \chi_0 \|_{L^2(\R)}^2} \, i \chi_0, \quad \forall f \in L^2(\R, \C).
$$
Similarly, we look for a solution $\psi_j$ of the form $\psi_j = i \lambda \chi_0 + z$, with $z = \pi_0(\psi_j)$ and $\lambda \in \R$. By definition of the operator $T_j$, and since $L_0(i \chi_0) = - i \chi_0/2$, the equation $T_j(\psi_j) = g_j$ is then equivalent to the system
\begin{equation}
\label{eq:T-j-equiv}
\begin{cases}
\lambda \Big( \frac{\pi^2 j^2}{d^2} - \frac{1}{2} + \alpha_0 \Big) = \beta - \frac{1}{\| \chi_0 \|_{L^2(\R)}^2} \, \langle T_0(z) - L_0(z), i \chi_0 \rangle_{L^2(\R)}, \\
\pi_0 \Big( T_0(z) + \frac{\pi^2 j^2}{d^2} z - g_j + \lambda \big( T_0(i \chi_0) - L_0(i \chi_0) \big) \Big) = 0.
\end{cases}
\end{equation}
In this system, we have set $T_0(i \chi_0) - L_0(i \chi_0) = i \alpha_0 \chi_0 + \pi_0(T_0(i \chi_0) - L_0(i \chi_0))$, with $\alpha_0 \in \R$. Going back to~\eqref{eq:cont-T-L-0}, we observe that
$$
\big| \alpha_0 \big| = \bigg| \frac{\langle T_0(i \chi_0) - L_0(i \chi_0), i \chi_0 \rangle_{L^2(\R)}}{\| \chi_0 \|_{L^2(\R)}^2} \bigg| \leq \mu_0 \big( 5 + 2 \mu_0 \big).
$$
Since $1/2 = \pi^2 k^2/d_k^2$, we can choose $d - d_k$ and $\mu_0$ small enough such that
\begin{equation}
\label{eq:bound-invert-lambda}
\frac{1}{2} - \frac{\pi^2 j^2}{d^2} - \alpha_0 \geq \frac{(2 k - 1) \pi^2}{2 d_k^2} \geq \frac{1}{4 k} > 0.
\end{equation}
In this case, the unique solution of the first equation in~\eqref{eq:T-j-equiv} is given by
\begin{equation}
\label{eq:lambda}
\lambda = \frac{1}{\frac{\pi^2 j^2}{d^2} - \frac{1}{2} + \alpha_0} \Big( \beta - \frac{\langle T_0(z) - L_0(z), i \chi_0 \rangle_{L^2(\R)}}{\| \chi_0 \|_{L^2(\R)}^2} \Big),
\end{equation}
and we can insert this expression into the second equation in~\eqref{eq:T-j-equiv} in order to obtain
\begin{equation}
\label{eq:z}
\begin{split}
\pi_0 \bigg( T_0(z) + \frac{\pi^2 j^2}{d^2} z - \frac{\langle T_0(z) - L_0(z), i \chi_0 \rangle_{L^2(\R)}}{\| \chi_0 \|_{L^2(\R)}^2 \big( \frac{\pi^2 j^2}{d^2} - \frac{1}{2} + \alpha_0 \big)} & \Big( T_0(i \chi_0) - L_0(i \chi_0) \Big) \bigg) \\
= \pi_0 \bigg( g_j - & \frac{\beta}{\frac{\pi^2 j^2}{d^2} - \frac{1}{2} + \alpha_0} \Big( T_0(i \chi_0) - L_0(i \chi_0) \Big) \bigg).
\end{split}
\end{equation}
In order to invert this last equation, we now introduce the self-adjoint bilinear form
\begin{align*}
\tilde{\boB}_j \big( z_1, z_2 \big) := \int_\R \Big( \langle z_1', z_2' \rangle_\C & + \Big( \frac{\pi^2 j^2}{d^2} - 1 + |\psi_0|^2 \Big) \langle z_1, z_2 \rangle_\C + 2 \langle \psi_0, z_1 \rangle_\C \, \langle \psi_0, z_2 \rangle_\C \Big) \\
- \frac{1}{\| \chi_0 \|_{L^2(\R)}^2 \big( \frac{2 \pi^2 j^2}{d^2} - 1 + 2 \alpha_0 \big)} \Big( & \langle T_0(z_1) - L_0(z_1), i \chi_0 \rangle_{L^2(\R)} \, \langle T_0(i \chi_0) - L_0(i \chi_0), z_2 \rangle_{L^2(\R)} \\
& + \langle T_0(z_2) - L_0(z_2), i \chi_0 \rangle_{L^2(\R)} \, \langle T_0(i \chi_0) - L_0(i \chi_0), z_1 \rangle_{L^2(\R)} \Big),
\end{align*}
for any functions $(z_1, z_2) \in (\boH^1(\R) \cap H_0)^2$. This bilinear form is continuous on the subspace $H^1(\R) \cap H_0$. Since
\begin{equation}
\label{eq:cont-error-form}
\bigg| \frac{\langle T_0(z) - L_0(z), i \chi_0 \rangle_{L^2(\R)} \, \langle T_0(i \chi_0) - L_0(i \chi_0), z \rangle_{L^2(\R)}}{\| \chi_0 \|_{L^2(\R)}^2 \big( \frac{\pi^2 j^2}{d^2} - \frac{1}{2} + \alpha_0 \big)} \bigg| \leq 4 k \mu_0^2 (5 + 2 \mu_0)^2 \, \big\| z \big\|_{L^2(\R)}^2,
\end{equation}
by~\eqref{eq:cont-T-L-0} and~\eqref{eq:bound-invert-lambda}, it also satisfies
$$
\tilde{\boB}_j \big( z, z \big) \geq \Big( \frac{\pi^2 j^2}{d^2} - \mu_0 \big( 5 + 2 \mu_0 \big) \big( 1 + 4 k \mu_0 (5 + 2 \mu_0) \big) \Big) \big\| z \big\|_{L^2(\R)}^2 \geq \frac{\pi^2}{2 d_k^2} \big\| z \big\|_{L^2(\R)}^2,
$$
when $z \in \boH^1(\R) \cap H_0$, and for $d - d_k$ and $\mu_0$ small enough. Moreover it follows from the definition of $\tilde{B}_j$,~\eqref{eq:unif-eta-eps-0} and~\eqref{eq:cont-error-form} that
$$
\tilde{\boB}_j \big( z, z \big) \geq \big\| z' \big\|_{L^2(\R)}^2 + \Big( \frac{\pi^2}{d^2} - 1 - \mu_0 - 4 k \mu_0^2 (5 + 2 \mu_0)^2 \Big) \big\| z \big\|_{L^2(\R)}^2 \geq \big\| z' \big\|_{L^2(\R)}^2 + \Big( \frac{\pi^2}{2 d_k^2} - 1 \Big) \big\| z \big\|_{L^2(\R)}^2,
$$
again for $d - d_k$ and $\mu_0$ small enough. We conclude that
\begin{equation}
\label{eq:coer-boB-tilde}
\tilde{\boB}_j \big( z, z \big) \geq \frac{\pi^2}{4 d_k^2} \big\| z \big\|_{H^1(\R)}^2,
\end{equation}
and we can derive from the Lax-Milgram theorem the existence of a unique solution $z_j \in \boH^1(\R) \cap H_0$ to~\eqref{eq:z}. Then we can go back to~\eqref{eq:lambda} in order to construct a unique solution $\psi_j := i \lambda_j \chi_0 + z_j \in \boH^1(\R)$ to the equation $T_j(\psi_j) = g_j$.

Moreover, it follows from~\eqref{eq:cont-T-L-0},~\eqref{eq:bound-invert-lambda} and~\eqref{eq:coer-boB-tilde} that the function $z$ satisfies
\begin{equation}
\label{eq:estim-z-j}
\big\| z_j \big\|_{H^1(\R)} \leq \frac{4 d_k^2}{\pi^2} \Big( \big\| g_j \big\|_{L^2(\R)} + 4 k \mu_0 (5 + 2 \mu_0) \, \big\| \beta \chi_0 \big\|_{L^2(\R)} \Big).
\end{equation}
Similarly we can estimate~\eqref{eq:lambda} so as to obtain
\begin{equation}
\label{eq:estim-lambda-j}
|\lambda_j| \leq 4 k \Big( |\beta| + \frac{\mu_0 (5 + 2 \mu_0)}{\| \chi_0 \|_{L^2(\R)}} \, \big\| z_j \big\|_{L^2(\R)} \Big).
\end{equation}
Using the definition of the number $\beta$, we conclude that there exists a number $C > 0$, depending only on $k$ and $\mu_0$, such that
\begin{equation}
\label{eq:estim-inv-boT-j<k}
\big\| \psi_j \big\|_{H^1(\R)} \leq C \big\| g_j \big\|_{L^2(\R)}.
\end{equation}
The fact that the function $\psi_j$ is actually in $H^2(\R)$ follows from applying standard elliptic theory to the equation $T_j(\psi_j) = g_j$, as well as the fact that we can replace the $H^1$-norm by the $H^2$-norm in the previous inequality.

When the function $g_j$ is additionally in $L_\sigma^\infty(\R)$, we argue as before to control the function $\psi_j$ in $L_\tau^\infty(\R)$ for any $0 < \tau < \sigma$. We introduce an even bounded Lipschitz function $f : \R \to \R$ and consider the functions $\tilde{z}_j := e^f z_j$, as well as $\zeta_j := e^{2 f} z_j$. We also set
$$
\mu_j := \frac{\langle \tilde{z}_j, i \chi_0 \rangle_{L^2(\R)}}{\| \chi_0 \|_{L^2(\R)}^2}, \quad \text{ and } \quad \nu_j := \frac{\langle \zeta_j, i \chi_0 \rangle_{L^2(\R)}}{\| \chi_0 \|_{L^2(\R)}^2},
$$
so that $\pi_0(\tilde{z}_j) = \tilde{z}_j - i \mu_j \chi_0$ and $\pi_0(\zeta_j) = \zeta_j - i \nu_j \chi_0$. Observe in particular that
$$
\big| \mu_j \big| \leq C \, \big\| g_j \big\|_{L^2(\R)} \, \big\| e^f \chi_0 \|_{L^2(\R)}, \quad \text{ and } \quad \big| \nu_j \big| \leq C \, \big\| g_j \big\|_{L^2(\R)} \, \big\| e^{2 f} \chi_0 \|_{L^2(\R)},
$$
for some number $C$, depending only on $k$ and $\mu_0$. Recalling that
$$
\tilde{\boB}_j(z_j, z) = \int_\R \langle g_j, z \rangle_\C, \quad \forall z \in \boH^1(\R) \cap H_0,
$$
and choosing $z = \pi_0(\zeta_j)$ in this identity, we can write
\begin{align*}
\tilde{\boB}_j \big( z_j, \zeta_j \big) = \nu_j \tilde{\boB}_j \big( z_j, i \chi_0 \big) + \int_\R \langle g_j, \zeta_j \rangle_\C - \nu_j \int_\R \langle g_j, i \chi_0 \rangle_\C.
\end{align*}
Since
\begin{align*}
\tilde{\boB}_j \big( \pi_0(\tilde{z}_j), \pi_0(\tilde{z}_j) \big) = & \tilde{\boB}_j \big( \tilde{z}_j, \tilde{z}_j \big) - 2 \mu_j \tilde{\boB}_j \big( \tilde{z}_j, i \chi_0 \big) + \mu_j^2 \tilde{\boB}_j \big( i \chi_0, i \chi_0 \big) \\
= & \tilde{\boB}_j \big( z_j, \zeta_j \big) + \int_\R \big( f' \big)^2 \big| \tilde{z}_j \big| - 2 \mu_j \tilde{\boB}_j \big( \tilde{z}_j, i \chi_0 \big) + \mu_j^2 \tilde{\boB}_j \big( i \chi_0, i \chi_0 \big),
\end{align*}
we can invoke~\eqref{eq:coer-boB-tilde} in order to obtain
\begin{align*}
\frac{\pi^2}{4 d_k^2} \big\| \tilde{z}_j \big\|_{H^1(\R)}^2 \leq & \frac{\pi^2 \, \mu_j}{2 d_k^2} \big\langle \tilde{z}_j, i \chi_0 \big\rangle_{H^1(\R)} - \frac{\pi^2 \, \mu_j^2}{4 d_k^2} \big\| \chi_0 \big\|_{H^1(\R)}^2 + \nu_j \tilde{\boB}_j \big( z_j, i \chi_0 \big) + \int_\R \langle e^f g_j, \tilde{z}_j \rangle_\C \\
& - \nu_j \int_\R \langle g_j, i \chi_0 \rangle_\C + \int_\R \big( f' \big)^2 \big| \tilde{z}_j \big| - 2 \mu_j \tilde{\boB}_j \big( \tilde{z}_j, i \chi_0 \big) + \mu_j^2 \tilde{\boB}_j \big( i \chi_0, i \chi_0 \big).
\end{align*}
Combining~\eqref{eq:unif-eta-eps-0},~\eqref{eq:cont-T-L-0}, and~\eqref{eq:bound-invert-lambda}, we check that
$$
\Big| \tilde{\boB}_j \big( z, \zeta \big) \Big| \leq \Big( \frac{\pi^2 k^2}{d^2} + 4 + 5 \mu_0 + 2 \mu_0^2 + 4 k \mu_0^2 (5 + 2 \mu_0)^2 \Big) \, \big\| z \big\|_{H^1(\R)} \, \big\| \zeta \big\|_{H^1(\R)}.
$$
Arguing as in the case $j = k$, we conclude that we can find a number $C > 0$, depending only on $k$, $\mu_0$ and $\delta_1$, such that
$$
\Big( \frac{\pi^2}{4 d_k^2} - \big\| f' \big\|_{L^\infty(\R)}^2 \Big) \big\| \tilde{z}_j \big\|_{H^1(\R)}^2 \leq C \bigg( \big\| e^f g_j \big\|_{L^2(\R)}^2 + \big\| g_j \big\|_{L^2(\R)}^2 \Big( \big\| e^f \chi_0 \big\|_{L^2(\R)}^2 + \big\| e^{2 f} \chi_0 \big\|_{L^2(\R)} \Big) \bigg).
$$
When $\sigma_0 < \min \{ 1/(2 \sqrt{2}), \pi/(2 d_k) \}$ and $0 < \tau < \sigma$, we derive as for~\eqref{eq:estim-inv-boT-k} from this estimate that
$$
\big\| z_j \big\|_{L_\tau^\infty(\R)} \leq C \big\| g_j \big\|_{L_\sigma^\infty(\R)},
$$
for a number $C$, depending also on $\delta_1$, $\sigma$ and $\tau$.

Recall here that, by~\eqref{eq:estim-z-j} and~\eqref{eq:estim-lambda-j}, we have
$$
\big\| i \lambda_j \chi_0 \big\|_{L_\tau^\infty(\R)} \leq \big| \lambda_j \big| \big\| \chi_0 \big\|_{L_\tau^\infty(\R)} \leq C \| g_j \|_{L^2(\R)},
$$
for a further number $C > 0$, depending on $k$, $\mu_0$ and $\tau$. In particular we can conclude that
\begin{equation}
\label{eq:estim-boT-exp-j<k}
\big\| \psi_j \big\|_{L_\tau^\infty(\R)} \leq C \big\| g_j \big\|_{L_\sigma^\infty(\R)}.
\end{equation}
The number $C$ in this inequality depends as before on $k$, $\delta_1$, $\nu_0$, $\sigma_0$, $\sigma$ and $\tau$ (using~\eqref{eq:control-mu-nu} for replacing the dependence on $\mu_0$ by a dependence on $\nu_0$).

Estimates~\eqref{eq:estim-inv-boT-j<k} and~\eqref{eq:estim-boT-exp-j<k} show that the conclusions of Lemma~\ref{lem:invert-T-k} hold under the additional assumption that $g\in Z_j$, i.e. that $g(x,y) = g_j(x)\cos(\pi jy/d)$, for some $j<k$.

\subsection{End of the proof of Lemma~\ref{lem:invert-T-k}}

Assume the hypothesis of Lemma~\ref{lem:invert-T-k} are satisfied. To solve the equation $\pi_k(T(w) - g)=0$, we project it on the Fourier sectors $j>k$, $j = k$ and $j<k$. The previous analysis shows that in each of these three sectors the equation can be solved and that the solutions satisfy the estimates in~\eqref{eq:proj-inv-L2} and~\eqref{eq:exp-bound-invert-T}. Adding these three solutions gives the desired solution, which will satisfy the required estimates. \qed

\section{Fixed point argument}
\label{sec:fixed-point}

\subsection{Proof of Lemma~\ref{lem:non-linearity}}

Note first that the maps $f_0(\psi_0, w)$ and $g(\psi_0, w)$ are in $\boL^\infty(\R)$, respectively in $\boL^\infty(\Omega_d)$, when the functions $\psi_0$ and $w$ belong to $\boL^\infty(\R)$, respectively in $\boL^\infty(\Omega_d)$. Let us next write
\begin{align*}
f_0 \big( \tilde{\psi}_0, \tilde{w} \big) & - f_0 \big( \psi_0, w \big) \\ & = \frac{1}{d} \int_0^d
\Big( 2 \, \langle \tilde{\psi}_0, \tilde{w} \rangle_\C \, (\tilde{w} - w)+2 w \, \langle \tilde{\psi}_0 - \psi_0, {w} \rangle_\C\, + 2 w \, \langle \tilde{\psi}_0, \tilde{w} - w \rangle_\C \\
& + |w|^2 (\tilde{\psi}_0 - \psi_0) + |w|^2 (\tilde{w} - w) + (\tilde{\psi}_0 + \tilde{w}) (|\tilde{w}|^2 - |w|^2) \Big) \, dy.
\end{align*}
When we take the $L^\infty_{2\sigma}$-norm of this expression, the terms in $\tilde{\psi}_0 - \psi_0$ are less than
$$
3 \big\| \psi_0 - \tilde{\psi}_0 \big\|_{L^\infty(\R)} \, \big\| w^2 e^{2\sigma |\cdot|} \big\|_{L^\infty(\Omega_d)}
\leq 3\big\| \psi_0 - \tilde{\psi}_0 \big\|_{L^\infty(\R)} \,\big\| w \big\|_{L_\sigma^\infty(\Omega_d)}^2,
$$
and the terms in $\tilde{w}- w$ are less than
\begin{align*}
\big\| (\tilde{w} - w) & e^{\sigma |\cdot|} \big\|_{L^\infty(\Omega_d)} \big(
2 \big\| \tilde{\psi}_0 \big\|_{L^\infty(\R)} \big\| \tilde{w}e^{\sigma |\cdot|} \big\|_{L^\infty(\Omega_d)}
+2 \big\| \tilde{\psi}_0 \big\|_{L^\infty(\R)} \big\| {w}e^{\sigma |\cdot|} \big\|_{L^\infty(\Omega_d)} \\
& +\big\| w^2 e^{\sigma |\cdot|} \big\|_{L^\infty(\Omega_d)} + \big(\big\| \tilde{\psi}_0 \big\|_{L^\infty(\R)}+
\big\| \tilde{w} \big\|_{L^\infty(\Omega_d))}\big)\big(\big\| {w}e^{\sigma |\cdot|} \big\|_{L^\infty(\Omega_d)}+\big\| \tilde{w}e^{\sigma |\cdot|} \big\|_{L^\infty(\Omega_d)}\big)\big).
\end{align*}
Since $\| \tilde{w} \big\|_{L^\infty(\Omega_d)} \leq \| \tilde{w} \big\|_{L_\sigma^\infty(\Omega_d)}$, the previous inequality guarantees that the function $f_0(\tilde{\psi}_0, \tilde{w}) - f_0(\psi_0, w)$ is in $L_{2 \sigma}^\infty(\R)$ and satisfies the estimate in~\eqref{eq:lips-f-0-g}. The same holds for the function $g(\tilde{\psi}_0, \tilde{w}) - g(\psi_0, w)$ given the definition of $g$. \qed

\subsection{Proof of Proposition~\ref{prop:exist-proj-sol}}

Recall first that the set $\boY_\sigma^\infty(\R)$ is defined as the intersection of the set $\boX_0(\R)$ with the vector space $W_{0, \sigma}^{1, \infty}(\R)$. Due to the uniform exponential decay of their derivative, any function $\psi$ in $\boY_\sigma^\infty(\R)$ has limits at $\pm \infty$, and since $\psi$ is moreover in $\boX_0(\R)$, these limits are of modulus one, that is of the form $e^{i \theta}$ for some number $\theta \in \R$. We now rely on this property in order to handle the proof of Proposition~\ref{prop:exist-proj-sol}.

Given any number $\theta \in (- \pi/4, \pi/4)$, we denote by $S_\theta : \R \to \C$ a smooth map belonging to $\boX_0(\R)$, depending smoothly on $\theta$, and such that
$$
S_\theta(x) = e^{i \, \theta \, \text{sign}(x)} \, S_0(x),
$$
for any $|x| > 1$. Note in particular that we can assume that the map $\theta \mapsto 1 - |S_\theta|^2$ is smooth from $(- \pi/4, \pi/4)$ to $L^2(\R)$. We then consider the maps
$$
(\theta, \varepsilon) \mapsto A(\theta, \varepsilon) = S_\theta + \varepsilon,
$$
$$
\psi \mapsto B_1(\psi) = \psi'' + \psi (1 - |\psi|^2), \quad (\psi, w) \mapsto B_2(\psi,w) = \pi_k \big( T(\psi, w) \big),
$$
and
$$
(\psi, w) \mapsto C(\psi, w) = \Big( f_0(\psi, w), \pi_k \big( g(\psi, w) \big) \Big).
$$

Assuming that $0 < \sigma \leq \sigma_0/2 < \sqrt{2}/2$, we check from the previous definitions that the map $B_1 \circ A$ is smooth from $(- \pi/4, \pi/4) \times (\boL_\sigma^\infty(\R) \cap \boW_{0, \sigma}^{2, \infty}(\R))$ to $\boL_\sigma^\infty(\R)$, while the map $B_2 \circ A$ is smooth from $(- \pi/4, \pi/4) \times (\boL_\sigma^\infty(\R) \cap \boW_{0, \sigma}^{2, \infty}(\R)) \times (\boW_\sigma^{2,\infty}(\Omega_d) \cap H_k)$ to $\boL_\sigma^\infty(\Omega_d) \cap H_k$. Here as in the sequel, the composition with the map $A$ only acts on the variable $\psi$. Note that in order to prove the previous claims, we use the property that any function $w \in W_{0, \sigma}^{2,\infty}(\Omega_d)$ is actually in $L_\sigma^\infty(\Omega_d)$ due to the Poincar\'e-Wirtinger inequality with exponentially weights in the $x$ variable.

From Lemma~\ref{lem:non-linearity}, we similarly check that the map $\boC = C \circ A$ is smooth from $(- \pi/4, \pi/4) \times \boL_\sigma^\infty(\R) \times \boL_\sigma^\infty(\Omega_d)$ to $\boL_{2 \sigma}^\infty(\R) \times (\boL_{2 \sigma}^\infty(\Omega_d) \cap H_k)$. Moreover, still from Lemma~\ref{lem:non-linearity}, by restricting $\boC$ to a small enough neighbourhood of the origin, we can make its Lipschitz norm in this neighbourhood as small as we wish.

Now, we infer from Lemmas~\ref{lem:exp-dec-eta-0} and~\ref{lem:control-V-N} that $B_1$ has an inverse from a neighbourhood of $0$ in $\boL_\sigma^\infty(\R)$ to a neighbourhood of $S_0$ in $\boY_\sigma^\infty(\R)$. When $|d - d_k| < \delta_1$ and $0 < \sigma \leq \sigma_0/2 $, it follows from Lemma~\ref{lem:invert-T-k} that if $\psi$ is close enough to $S_0$ in $\boY_\sigma^\infty(\R)$, then $B_2(\psi, \cdot)$ has an inverse from a neighbourhood of $0$ in $\boL_{2 \sigma}^\infty(\Omega_d) \cap H_k$ to a neighbourhood of $0$ in $\boL_\sigma^\infty(\Omega_d) \cap H_k$. Moreover, by Lemmas~\ref{lem:control-exp-V-N} and~\ref{lem:invert-T-k}, both inverses are Lipschitz. Therefore, admitting for a moment that the map $A$ has a Lipschitz inverse from a neighbourhood of $S_0$ in $\boY_\sigma^\infty(\R)$ to a neighbourhood of $(0, 0)$ in $(- \pi/4, \pi/4) \times \boL_\sigma^\infty(\R)$, we deduce that the map
$$
\boB(\theta, \varepsilon, w) = \big( B_1(S_\theta + \varepsilon), B_2(S_\theta + \varepsilon, w) \big),
$$
has a Lipschitz inverse from a neighbourhood of $0$ in $\boL_\sigma^\infty(\R) \times (\boL_{2 \sigma}^\infty(\Omega_d) \cap H_k)$ to a neighbourhood of $(0, 0, 0)$ in $(- \pi/4, \pi/4) \times \boL_\sigma^\infty(\R) \times (\boL_\sigma^\infty(\Omega_d) \cap H_k)$. Moreover, since $\psi = S_\theta + \varepsilon$ satisfies~\eqref{eq:psi-0} and $w$ satisfies~\eqref{eq:proj-inv-L1}, this inverse is in fact Lipschitz with values into $(- \pi/4, \pi/4) \times (L_\sigma^\infty(\R) \cap W_\sigma^{2,\infty}(\R)) \times W_{0, \sigma}^{2, \infty}(\Omega_d)$.

We may then conclude that $\boB$, since it is smooth and has a Lipschitz inverse, has a smooth inverse defined on a neighbourhood of the origin in $\boL_\sigma^\infty(\R) \times (\boL_{2 \sigma}^\infty(\Omega_d) \cap H_k)$ to a neighbourhood of $(0, 0, 0)$ in $(- \pi/4, \pi/4) \times (\boL_\sigma^\infty(\R) \cap \boW_\sigma^{2, \infty}(\R)) \times (\boW_{0, \sigma}^{2,\infty}(\Omega_d) \cap H_k)$, and then to a neighbourhood of $(- \pi/4, \pi/4) \times \boL_\sigma^\infty(\R) \times (\boL_\sigma^\infty(\Omega_d) \cap H_k)$ after embedding the spaces $W_{0, \sigma}^{2, \infty}(\Omega_d)$ into $L_\sigma^\infty(\Omega_d)$.

To justify that $A$ has a Lipschitz inverse from a neighbourhood of $S_0$ in $\boY_\sigma^\infty(\R)$ to a neighbourhood of $(0, 0)$ in $(- \pi/4, \pi/4) \times \boL_\sigma^\infty(\R)$, we note that, if $(\psi_1,\psi_2) \in \boY_\sigma^\infty(\R)^2$ then, since $\psi_1(0) = \psi_2(0) = 0$, their limits as $x \to + \infty$ are
$$
e^{i \theta_1} = \int_0^{+ \infty} \psi_1'(x) \, dx, \quad \text{ and } \quad e^{i \theta_2} = \int_0^{+\infty}\psi_2'(x) \, dx,
$$
so that
$$
\big| \theta_1 - \theta_2 \big| \leq \frac{C}{\sigma} \big\| \psi_1' - \psi_2' \big\|_{L_\sigma^\infty},
$$
by bounding $|\psi_1'(x) - \psi_2'(x)|$ by $e^{- \sigma |x|} \, \| \psi_1' - \psi_2' \|_{L_\sigma^\infty}$. As a consequence, we can bound from above $\| \varepsilon_1 - \varepsilon_2 \|_{L_\sigma^\infty}$, where we have set $\varepsilon_i =\psi_i - S_{\theta_i}$.

It remains to apply the Picard fixed point theorem with parameter. For any small enough $\lambda \in \R$, we define
\begin{equation}
\label{def:Xi-lambda}
\Xi_\lambda \big( \theta, \varepsilon, w \big) := \boB^{-1} \circ \boC \big( \theta, \varepsilon, w + \lambda \chi_k \big).
\end{equation}
This map is well-defined and smooth from a neighbourhood of the origin in $(- \pi/4, \pi/4) \times \boL_\sigma^\infty(\R) \times (\boL_\sigma^\infty(\Omega_d) \cap H_k)$ to itself. We have used here the fact that, speaking loosely, $\boC$ maps $L_\sigma^\infty$ to $L_{2 \sigma}^\infty$. Moreover, by restricting the neighbourhood if necessary we may make the Lipschitz constant of $\Xi_\lambda$ as small as we wish, so that it has a unique fixed point $(\theta_\lambda,\varepsilon_\lambda, W_\lambda)$ for any small enough value of $\lambda$. Moreover, since $\Xi_\lambda$ depends also smoothly on $\lambda$, this fixed point is a smooth function of $\lambda$ with values into $(- \pi/4, \pi/4) \times (\boL_\sigma^\infty(\R) \cap \boW_\sigma^{2, \infty}(\R)) \times (\boW_{0, \sigma}^{2,\infty}(\Omega_d) \cap H_k)$.

The functions $\Psi_0^\lambda = S_{\theta_\lambda} +\varepsilon_\lambda$ and $W_\lambda$ are then the unique solutions to~\eqref{eq:proj-syst-0} and~\eqref{eq:proj-syst-1} in a small neighbourhood of $(S_0, 0)$ in $\boY_\sigma^\infty(\R) \times (\boL_\sigma^\infty(\Omega_d) \cap H_k)$, and they depend smoothly on $\lambda$ with values into $W_{0, \sigma}^{2, \infty}(\R) \times W_{0, \sigma}^{2, \infty}(\Omega_d)$. Note also that, due to the facts that the map $\theta \mapsto 1 - |S_\theta|^2$ is smooth from $(- \pi/4, \pi/4)$ to $L^2(\R)$, and the functions $\lambda \mapsto \theta_\lambda$ and $\lambda \mapsto \varepsilon_\lambda$ are smooth with values into $(- \pi/4, \pi/4)$, respectively $\boL_\sigma^\infty(\R)$, the function $\lambda \mapsto 1 - |\Psi_0^\lambda|^2$ is smooth with values into $L^2(\R)$. This concludes the proof of Proposition~\ref{prop:exist-proj-sol}. \qed

\section{Differentiability properties}
\label{sec:diff-I}

In this section, we rely on the smoothness of the maps $\lambda \mapsto (\Psi_0^\lambda, W^\lambda)$ and $\lambda \mapsto 1 - |\Psi_0^\lambda|^2$ to describe the derivatives of the function $J$ as in Lemma~\ref{lem:diff-I}, and also to expand the energy $E(\Psi_{k, d})$ as in~\eqref{eq:DL-E}.

\subsection{Proof of Lemma~\ref{lem:diff-I}}

The smoothness of the function $J$ follows from the smoothness of the maps $(d, \lambda) \mapsto (\psi_0^\lambda, W^\lambda)$. Going back to the proof of Proposition~\ref{prop:exist-proj-sol}, we observe that the map $\Xi_\lambda$ in~\eqref{def:Xi-lambda} also depends on the width $d$ through the fact that the functions $w$ and $\chi_k$ are defined on the strip $\Omega_d$. Rescaling distances by $d$, that is applying the change of variables $z = d_k \, y/d$, we can impose these functions to be defined on a unique domain $\Omega_{d_k}$. Since this change of variables is smooth, we check that the map $\Xi_\lambda$ depends smoothly on the pair $(d, \lambda)$. Applying the Picard fixed point theorem with parameter as in the proof of Proposition~\ref{prop:exist-proj-sol} then guarantees that we can find two numbers $0 < \delta_2 \leq \delta_1$ and $0 < \rho_2 \leq \rho_1$ such that the map $(d, \lambda) \mapsto (\psi_0^\lambda, \tilde{W}_\lambda)$ is smooth from $(d_k - \delta_2, d_k + \delta_2) \times (- \rho_2, \rho_2)$ to $W_{0, \sigma}^{2, \infty}(\R) \times W_{0, \sigma}^{2, \infty}(\Omega_{d_k})$. Here as in the sequel, we have used the tilde symbol to highlight the fact that the function $\tilde{W}^\lambda$ now depends on $(x, z) \in \R \times \Omega_{d_k}$. It is then enough to apply the smooth reverse change of variables $y = d \, z/d_k$ to deduce from the definition of the function $J$ that it is smooth on $(d_k - \delta_2, d_k + \delta_2) \times (- \rho_2, \rho_2)$.

The fact that the function $J$ is odd in its second variable $\lambda$ relies on the property that the conjugate function $\psi$ of a solution to~\eqref{eq:GL} remains a solution to~\eqref{eq:GL}. Due to the property that the function $\chi_k$ takes purely imaginary values, the conjugated pair $(\overline{\psi_0^\lambda}, \overline{W^\lambda})$ is a solution to~\eqref{eq:proj-syst-0} and~\eqref{eq:proj-syst-1} for $\lambda$ replaced by $- \lambda$. Using the uniqueness of the solutions of these equations provided by Proposition~\ref{prop:exist-proj-sol}, we deduce that
$$
\psi_0^{(- \lambda)} = \overline{\psi_0^\lambda}, \quad \text{ and } \quad W^{(- \lambda)} = \overline{W^\lambda},
$$
for any $- \rho_1 < \lambda < \rho_1$. Introducing this identity into the definition of the function $J$, invoking the definition of the function $g$, and again the fact that the function $\chi_k$ takes purely imaginary values, we obtain that $J$ is indeed odd in the variable $\lambda$.

We next compute the values of the functions $J$ and its first order derivatives for $(d, \lambda) = (d_k, 0)$. We first observe that for $\lambda = 0$, the pair $(S_0, 0)$ is a solution to~\eqref{eq:proj-syst-0} and~\eqref{eq:proj-syst-1}. Hence by uniqueness it is equal to the pair $(\Psi_0^0, W^0)$ for any $d_k - \delta_2 < d < d_k + \delta_2$. This is sufficient to check that
$$
J(d, 0) = 0,
$$
so that
$$
\partial_d^\ell J(d_k, 0) = 0,
$$
for any $\ell \geq 1$.

When $\lambda \neq 0$, we call $Z_k^\lambda := W^\lambda + \lambda \chi_k$, and we recall that
$$
J(d, \lambda) := \big\langle - \Delta Z_k^\lambda - Z_k^\lambda \big( 1 - |\Psi_0^\lambda|^2 \big) + 2 \, \langle \Psi_0^\lambda, Z_k^\lambda \rangle_\C \, \Psi_0^\lambda - g\big( \Psi_0^\lambda, Z_k^\lambda \big), \chi_k \big\rangle_{L^2(\Omega_d)},
$$
where $g$ is defined by~\eqref{def:g}. Since the map $\lambda \mapsto (\psi_0^\lambda, W^\lambda)$ is smooth with values into $W_{0, \sigma}^{2, \infty}(\R) \times W_{0, \sigma}^{2, \infty}(\Omega_{d_k})$ by Proposition~\ref{prop:exist-proj-sol}, we are allowed to compute the derivative with respect to $\lambda$
\begin{equation}
\label{def:dIlam}
\begin{split}
\partial_\lambda J(d, \lambda) := \big\langle & - \Delta \big( \partial_\lambda Z_k^\lambda \big) - \partial_\lambda Z_k^\lambda \big( 1 - |\Psi_0^\lambda|^2 \big)+ 2 \, \langle \Psi_0^\lambda, \partial_\lambda Z_k^\lambda \rangle_\C \, \Psi_0^\lambda \\
& + 2 \, \langle \partial_\lambda \Psi_0^\lambda, Z_k^\lambda \rangle_\C \, \Psi_0^\lambda+ 2 \, \langle \Psi_0^\lambda, Z_k^\lambda \rangle_\C \, \partial_\lambda \Psi_0^\lambda
+2 \, \langle \partial_\lambda \Psi_0^\lambda, \Psi_0^\lambda \rangle_\C \, Z_k^\lambda \\& - D_1 g\big( \Psi_0^\lambda, Z_k^\lambda)( \partial_\lambda \Psi_0^\lambda) - D_2 g\big( \Psi_0^\lambda, Z_k^\lambda \big)( \partial_\lambda Z_k^\lambda), \chi_k \big\rangle_{L^2(\Omega_d)}.
\end{split}
\end{equation}
When we want to evaluate this quantity at $(d_k,0)$, we use the previous computed special values
$$
\Psi_0^* := {\Psi_0^\lambda}_{|(d_k,0)} = S_0, \quad \text{ and } \quad Z_k^* := {Z_k^\lambda}_{|(d_k,0)}= (W^\lambda + \lambda \chi_k)_{|(d_k, 0)} = 0.
$$
Given the expression of the function $g$, we check that $D_1 g(S_0, 0) = D_2 g(S_0, 0) = 0$. Calling $\goL_0$ the linearized operator at $S_0$ given by
$$
\goL_0(\psi) := - \Delta \psi - \psi (1 - S_0^2) + 2 \langle S_0, \psi \rangle_\C \, S_0,
$$
this implies that
$$
\partial_\lambda J(d_k, 0) = \big\langle \goL_0 \big( {\partial _\lambda Z_k^\lambda}_{|(d_k, 0)} \big), \chi_k \big\rangle_{L^2(\Omega_d)}.
$$
This quantity is equal to zero because $\chi_k$ is in the kernel of $\goL_0$. Since $J(d, \lambda)$ is odd in the variable $\lambda$, we point out that the second derivative $\partial_{\lambda, \lambda} J(d_k, 0)$ is also equal to zero.

The next step is to compute the derivative $\partial^2_{d, \lambda} J(d_k, 0)$ so we have to differentiate ~\eqref{def:dIlam} with respect to $d$. For this purpose, we use the change of variables $y = d z/ d_k$ in order to drop the dependence on $d$ of the domain $\Omega_d$. We write down explicitly this dependence through the identity $\tilde{W}^\lambda(x, z) = W^\lambda(x, d z/d_k)$, where the function $\tilde{W}^\lambda$ is now defined in $\Omega_{d_k}$. We similarly set
$$
\tilde{\chi}_k(x, z) := i \chi_0(x) \cos \Big( \frac {k \pi z}{d_k} \Big),
$$
so that the function $\tilde{\chi}_k$ no longer contains dependence on the parameter $d$. Calling $\tilde{Z}_k^\lambda= \tilde{W}^\lambda + \lambda \tilde{\chi}_k$, we obtain
\begin{equation}
\label{def:dIlamd}
\begin{split}
\partial_\lambda J(d, \lambda) := \frac{d}{d_k} \big\langle & - \partial_{x, x} \big( \partial_\lambda \tilde{W}^\lambda + \tilde{\chi}_k \big) -\frac {d_k^2}{d^2} \partial_{z, z} \big( \partial_\lambda \tilde{W}^\lambda + \tilde{\chi}_k \big) - \partial_\lambda \tilde{Z}_k^\lambda \big( 1 - |\Psi_0^\lambda|^2 \big) \\
& + 2 \, \langle \Psi_0^\lambda, \partial_\lambda \tilde{Z}_k^\lambda \rangle_\C \, \Psi_0^\lambda + 2 \, \langle \partial_\lambda \Psi_0^\lambda, \tilde{Z}_k^\lambda \rangle_\C \, \Psi_0^\lambda + 2 \, \langle \Psi_0^\lambda, \tilde{Z}_k^\lambda \rangle_\C \, \partial_\lambda \Psi_0^\lambda \\
& + 2 \, \langle \partial_\lambda \Psi_0^\lambda, \Psi_0^\lambda \rangle_\C \, \tilde{Z}_k^\lambda - D_1 g \big( \Psi_0^\lambda, \tilde{Z}_k^\lambda \big) (\partial_\lambda \Psi_0^\lambda) \\
& - D_2 g \big( \Psi_0^\lambda, \tilde{Z}_k^\lambda \big) (\partial_\lambda \tilde{Z}_k^\lambda), \tilde{\chi}_k \big\rangle_{L^2(\Omega_{d_k})}.
\end{split}
\end{equation}
Before differentiating further with respect to $d$, we are going to prove that at $(d_k,0)$, the first order derivatives of the smooth function $(\lambda, d) \mapsto (\psi_0^\lambda, W^\lambda)$ are given by
\begin{equation}
\label{ddlambdapsi}
\partial_\lambda \Psi_0^* := \partial_\lambda {\Psi_0^\lambda}_{|(d_k, 0)} = 0, \quad \partial_d \Psi_0^* := \partial_d {\Psi_0^\lambda}_{|(d_k, 0)} = 0,
\end{equation}
\begin{equation}
\label{ddlambdaW}
\partial_\lambda \tilde{W}^* := {\partial_\lambda \tilde{W}^\lambda}_{|(d_k, 0)} = 0, \quad \partial_d \tilde{W}^* := {\partial_d \tilde{W}^\lambda}_{|(d_k, 0)} = 0.
\end{equation}
For this purpose, we need to write the equations satisfied by these functions. Let us differentiate~\eqref{eq:proj-syst-0} with respect to $\lambda$ and evaluate it at $(d_k, 0)$, then we find
$$
- \big( \partial_\lambda {\Psi_0^*} \big)'' - \partial_\lambda {\Psi_0^*} \big( 1 - |S_0|^2 \big) + 2 \, \langle S_0, \partial_\lambda {\Psi_0^*}\rangle_\C \, S_0 = 0.
$$
Since $\Psi_0^*$ does not depend on the variable $z$, this means that $L_0(\partial_\lambda {\Psi_0^*}) = 0$, so that there are real numbers $\alpha$ and $\beta$ such that
$$
\partial_\lambda {\Psi_0^*} = \alpha S_0' + i \beta S_0.
$$
Since the real part of $\Psi_0^\lambda$ is odd and its imaginary part is even, it is the same with $\partial_\lambda {\Psi_0^\lambda}$. Since $S_0$ is odd and $S_0'$ is even, then $\alpha = \beta = 0$ and $\partial_\lambda {\Psi_0^*} = 0$. The same reasoning, differentiating with respect to $d$, leads to $\partial_d {\Psi_0^*} = 0$, so that~\eqref{ddlambdapsi} does hold.

For the function $\tilde{W}^\lambda$, we go back to~\eqref{eq:proj-syst-1}, implement the change of variables $y = d z/d_k$, and differentiate the resulting equation with respect to $d$ or $\lambda$. Since the rescaled orthogonal projection $\tilde{\pi}_k$ given by~\eqref{def:pi-k} for $d = d_k$ is independent of $d$ or $\lambda$, we can compute the derivative with respect to $\lambda$ at $(d_k, 0)$ and find
$$
\tilde{\pi}_k \Big( - \Delta \big( \partial_\lambda \tilde{W}^* + \tilde{\chi}_k \big) - \big( \partial_\lambda \tilde{W}^* + \tilde{\chi}_k \big) \big( 1 - |S_0|^2 \big) + 2 \, \langle S_0, \partial_\lambda \tilde{W}^* + \tilde{\chi}_k \rangle_\C \, S_0 \Big) = 0.
$$
Since $\goL_0(\tilde{\chi}_k) = 0$, and the operators $\goL_0$ and $\tilde{\pi}_k$ commute, we deduce that the function $\tilde{\pi}_k(\partial_\lambda \tilde{W}^*)$ is in the kernel of $\goL_0$. Moreover, the image of the projection $\tilde{\pi}_k$ is included in the orthogonal of the kernel, so this implies
$$
\tilde{\pi}_k \big( \partial_\lambda \tilde{W}^* \big) = 0.
$$
Since $\tilde{\pi}_k(\tilde{W}^\lambda) = \tilde{W}^\lambda$, we have $\tilde{\pi}_k(\partial_\lambda \tilde{W}^*) = \partial_\lambda \tilde{W}^*$ and therefore $\partial_\lambda \tilde{W}^* = 0$. The same reasoning holds when differentiating with respect to $d$ to get~\eqref{ddlambdaW}.

Differentiating~\eqref{def:dIlamd} with respect to $d$, and using the fact that $\tilde{\Psi}_0^* = S_0$ and $\tilde{W}^* = 0$, we next find at $(d_k, 0)$,
\begin{equation}
\begin{split}
\label{dlamIlam}
\partial_{d, \lambda} J(d_k, 0) = & \frac 1{d_k} \partial_\lambda J(d_k, 0) + \frac {2}{d_k} \big\langle \partial_{z, z} (\partial_\lambda \tilde{W}^* + \tilde{\chi}_k), \tilde{\chi}_k \big\rangle_{L^2(\Omega_{d_k})} + \big\langle \goL_0 (\partial_{d, \lambda} \tilde{W}^*), \tilde{\chi}_k \big\rangle_{L^2(\Omega_{d_k})} \\
& + 2 \big\langle \langle S_0, \partial_d \Psi_0^* \rangle_\C \, (\partial_\lambda \tilde{W}^* + \tilde{\chi}_k) + \langle \partial_d \Psi_0^*, \partial_\lambda \tilde{W}^* + \tilde{\chi}_k \rangle_\C \, S_0 + \langle S_0, \partial_d \tilde{W}^* \rangle_\C \, \partial_\lambda \Psi_0^* \\
& + \langle S_0, \partial_\lambda \tilde{W}^* + \tilde{\chi}_k \rangle_\C \, \partial_d \Psi_0^* + \langle \partial_\lambda \Psi_0^*, S_0 \rangle_\C \, \partial_d \tilde{W}^* + \langle \partial_\lambda \Psi_0^*, \partial_d \tilde{W}^* \rangle_\C \, S_0, \tilde{\chi}_k \big\rangle_{L^2(\Omega_{d_k})} \\
& - \big\langle D_1 \tilde{g}\big( S_0, 0 \big)(\partial_{d, \lambda} \Psi_0^*) + D_2 \tilde{g} \big( S_0, 0 \big)(\partial_{d, \lambda} \tilde{W}^*) + D_{1, 1} \tilde{g} \big( S_0, 0 \big)(\partial_\lambda \Psi_0^*, \partial_d \Psi_0^*) \\
& + D_{1, 2} \tilde{g} \big( S_0, 0 \big)(\partial_\lambda \Psi_0^*, \partial_d \tilde{W}^*) + D_{1, 2} \tilde{g} \big( S_0, 0 \big)(\partial_d \Psi_0^*, \partial_\lambda \tilde{W}^* + \tilde{\chi}_k) \\
& + D_{2, 2} \tilde{g}\big( S_0, 0 \big) (\partial_d \tilde{W}_\lambda^*, \partial_\lambda \tilde{W}^* + \tilde{\chi_k}), \tilde{\chi}_k \big\rangle_{L^2(\Omega_{d_k})}.
\end{split}
\end{equation}
In this formula, given two functions $\psi_0 : \R \to \C$ and $\tilde{w} : \Omega_{d_k} \to \C$, we have set
$$
\tilde{g} \big( \psi_0, \tilde{w} \big) = - 2 \big\langle \psi_0, \tilde{w} \big\rangle_\C \tilde{w} - \big| \tilde{w} \big|^2 \big( \psi_0 + \tilde{w} \big) + \tilde{f}_0 \big( \psi_0, \tilde{w} \big),
$$
where
$$
\tilde{f}_0\big( \psi_0, \tilde{w} \big)(x) := \frac{1}{d_k} \int_0^{d_k} \Big( 2 \big\langle \psi_0(x), \tilde{w}(x, z) \big\rangle_\C \tilde{w}(x, z) + \big| \tilde{w}(x, z) \big|^2 \big( \psi_0(x) + \tilde{w}(x, z) \big) \Big) \, dz.
$$
When $(d, \lambda) = (d_k, 0)$, the first term in~\eqref{dlamIlam} is equal to zero by the previous computation, the third term is also zero because $\tilde{\chi}_k$ is in the kernel of the self-adjoint operator $\goL_0$. Moreover, we can check that
$$
D_1 \tilde{g}(S_0, 0) = D_2 \tilde{g}(S_0, 0) = 0, \quad \text{ and } \quad D_{1, 1} \tilde{g}(S_0, 0) = D_{1, 2} \tilde{g}(S_0, 0) = 0,
$$
whereas
\begin{equation}
\label{eq:D22-g}
\begin{split}
D_{2, 2} \tilde{g} \big( S_0, 0 \big)(H, \tilde{L}) = & - 2 \Big( \big\langle S_0, H \big\rangle_\C \tilde{L} + \big\langle S_0, \tilde{L} \big\rangle_\C H + \big\langle H, \tilde{L} \big\rangle_\C S_0 \Big) \\
& + \frac{2}{d_k} \int_0^{d_k} \Big[ \big\langle S_0, H \big\rangle_\C \tilde{L} + \big\langle S_0, \tilde{L} \big\rangle_\C H + \big\langle H, \tilde{L} \big\rangle_\C S_0 \Big](\cdot, z) \, dz.
\end{split}
\end{equation}
Using the values at $(d_k, 0)$ given by~\eqref{ddlambdapsi} and~\eqref{ddlambdaW}, we find
$$
\partial_{d, \lambda} J(d_k, 0) = - \frac {2}{d_k} \int_{\Omega_{d, k}} \big| \partial_z \tilde{\chi}_k \big|^2 = - \frac{2 k^2 \pi^2}{d_k^2} \int_\R \chi_0^2 = - 2 \sqrt{2},
$$
which implies the first equality in~\eqref{eq:der-neq-0}.

We finally compute the third order derivative $\partial_{\lambda, \lambda, \lambda} J(d_k, 0)$. In this direction, we first derive from~\eqref{def:dIlam} that
\begin{align*}
& \partial_{\lambda, \lambda} J(d, \lambda) = \big\langle - \Delta (\partial_{\lambda, \lambda} W^\lambda) - \partial_{\lambda, \lambda} W^\lambda (1 - |\psi_0^\lambda|^2) + 2 \langle \Psi_0^\lambda, \partial_{\lambda, \lambda} W^\lambda \rangle_\C \Psi_0^\lambda + 4 \langle \partial_\lambda \Psi_0^\lambda, \Psi_0^\lambda \rangle_\C \, (\partial_\lambda W_\lambda \\
& + \chi_k) + 4 \langle \partial_\lambda \Psi_0^\lambda, \partial_\lambda W^\lambda + \chi_k \rangle_\C \, \Psi_0^\lambda + 4 \langle \Psi_0^\lambda, \partial_\lambda W^\lambda + \chi_k \rangle_\C \, \partial_\lambda \Psi_0^\lambda + 2 \langle \partial_{\lambda, \lambda} \Psi_0^\lambda, \Psi_0^\lambda \rangle_\C \, (W^\lambda + \lambda \chi_k) \\
& + 2 \langle \partial_{\lambda, \lambda} \Psi_0^\lambda, W^\lambda + \lambda \chi_k \rangle_\C \, \Psi_0^\lambda + 4 \langle \partial_\lambda \Psi_0^\lambda, W^\lambda + \lambda \chi_k \rangle_\C \, \partial_\lambda \Psi_0^\lambda + 2 \langle \Psi_0^\lambda, W^\lambda + \lambda \chi_k \rangle_\C \, \partial_{\lambda, \lambda} \Psi_0^\lambda \\
& + 2 |\partial_\lambda \Psi_0^\lambda|^2 \, (W^\lambda + \lambda \chi_k) - D_1 g\big( \Psi_0^\lambda, W^\lambda + \lambda \chi_k \big)(\partial_{\lambda, \lambda} \Psi_0^\lambda) - D_2 g \big( \Psi_0^\lambda, W^\lambda + \lambda \chi_k \big)(\partial_{\lambda, \lambda} W^\lambda) \\
& - D_{1, 1} g \big( \Psi_0^\lambda, W^\lambda + \lambda \chi_k \big)(\partial_\lambda \Psi_0^\lambda, \partial_\lambda \Psi_0^\lambda) - 2 D_{1, 2} g \big( \Psi_0^\lambda, W^\lambda + \lambda \chi_k \big)(\partial_\lambda \Psi_0^\lambda, \partial_\lambda W^\lambda + \chi_k) \\
& - D_{2, 2} g\big( \Psi_0^\lambda, W^\lambda + \lambda \chi_k \big) (\partial_\lambda W^\lambda + \chi_k, \partial_\lambda W^\lambda + \chi_k), \chi_k\big\rangle_{L^2(\Omega_d)}.
\end{align*}
Differentiating again this quantity with respect to $\lambda$, we find at $(d_k, 0)$
\begin{equation}
\label{iddd}
\begin{split}
\partial_{\lambda, \lambda,\lambda} J(d_k, 0) = & \big\langle \goL_0 (\partial_{\lambda, \lambda, \lambda} W^*) + 6 \langle \partial_{\lambda, \lambda} \Psi_0^*, S_0 \rangle_\C \, \chi_k + 6 \langle \partial_{\lambda, \lambda} \Psi_0^*, \chi_k \rangle_\C \, S_0 + 6 \langle \chi_k, S_0 \rangle_\C \, \partial_{\lambda, \lambda} \Psi_0^* \\
& - 3 D_{2, 2} g \big( S_0, 0 \big)(\chi_k,\partial_{\lambda,\lambda} W^* ) - 2 D_{2, 2, 2} g \big( S_0, 0 \big)(\chi_k, \chi_k, \chi_k), \chi_k \big\rangle_{L^2(\Omega_{d_k})},
\end{split}
\end{equation}
where we have set as before
$$
\partial_{\lambda, \lambda} \Psi_0^* := \partial_{\lambda, \lambda} {\Psi_0^\lambda}_{|(d_k, 0)}, \quad \partial_{\lambda, \lambda} W^* := {\partial_{\lambda, \lambda} W^\lambda}_{|(d_k, 0)}, \quad \text{ and } \quad \partial_{\lambda, \lambda,\lambda} W^* := {\partial_{\lambda, \lambda, \lambda} W^\lambda}_{|(d_k, 0)}.
$$
In order to obtain the simplified formula in~\eqref{iddd}, we have used the fact that $\Psi_0^* = S_0$, $\partial_\lambda \Psi_0^* = 0$ and $W^* = \partial_\lambda W^* = 0$, as well as the identities $D_1 g(S_0, 0) = D_2 g(S_0, 0) = D_{1, 1} g(S_0, 0) = D_{1, 2} g(S_0, 0) = 0$. We next recall that $\goL_0 (\chi_k) = 0$, and we observe that $\langle S_0, \chi_k \rangle_\C = 0$ because $S_0$ takes real values, while $\chi_k$ takes purely imaginary ones. Using this property further, we check from (the rescaled version of)~\eqref{eq:D22-g} that
$$
\big\langle D_{2, 2} g \big( S_0, 0 \big)(\chi_k, \partial_{\lambda, \lambda} W^* ), \chi_k \big\rangle_{L^2(\Omega_{d_k})} = - 2 \int_{\Omega_{d_k}} \langle S_0, \partial_{\lambda, \lambda} W^* \rangle_\C \, |\chi_k|^2,
$$
while we can compute similarly that
$$
\big\langle D_{2, 2, 2} g \big( S_0, 0 \big) (\chi_k, \chi_k, \chi_k), \chi_k \big\rangle_{L^2(\Omega_{d_k})} = - 6 \int_{\Omega_{d_k}} |\chi_k|^4.
$$
Therefore, we find from~\eqref{iddd} that
\begin{equation}
\label{eq:iddd}
\partial_{\lambda, \lambda,\lambda} J(d_k, 0) = 6 \int_{\Omega_{d_k}} \Big( 2 |\chi_k|^4 + \langle S_0, \partial_{\lambda,\lambda} W^* \rangle_\C \, |\chi_k|^2 + \langle S_0, \partial_{\lambda, \lambda} \Psi_0^* \rangle_\C \, |\chi_k|^2 \Big).
\end{equation}
We finally compute the derivatives $\partial_{\lambda, \lambda} \Psi_0^*$ and $\partial_{\lambda, \lambda} W^*$ using the equations which they satisfy.

Concerning the derivative $U := \partial_{\lambda, \lambda} \Psi_0^*$, we use the formulae in~\eqref{ddlambdapsi} and~\eqref{ddlambdaW} to derive from~\eqref{eq:proj-syst-0} that this derivative satisfies
\begin{equation}
\label{eq:dll-Psi0*}
- U'' - U \big( 1 - S_0^2 \big) + 2 \big\langle S_0, U \big\rangle_\C S_0 = - 2 D_{2, 2} f_0 \big( S_0, 0 \big)(\chi_k, \chi_k) = - |\chi_0|^2 \, S_0,
\end{equation}
that is
\begin{equation}
\label{eqU}
- U''(x) - \frac{U(x)}{\cosh^2 \big( \frac{x}{\sqrt{2}} \big)} + 2 \text{Re} \big( U(x) \big) \tanh^2 \Big( \frac{x}{\sqrt{2}} \Big) = - \frac{\sinh \big( \frac{x}{\sqrt{2}} \big)}{\cosh^3 \big( \frac{x}{\sqrt{2}} \big)}.
\end{equation}
The imaginary part of $U$ is even and vanishes at the origin, while the solutions of the equation satisfied by this imaginary part are spanned by the functions $x \mapsto \tanh(x/\sqrt{2})$ and $x \mapsto x \tanh(x/\sqrt{2}) - \sqrt{2}$. As a consequence, the imaginary part of $U$ is identically equal to zero. On the other hand, the function
$$
U_0(x) = - \frac{x}{2 \sqrt{2} \cosh^2 \big( \frac{x}{\sqrt{2}} \big)},
$$
is a special solution of~\eqref{eqU}. The solutions corresponding to the homogeneous equation for the real part of $U$ are spanned by the functions
$$
U_1(x) = \frac{1}{\cosh^2 \big( \frac{x}{\sqrt{2}})}, \quad \text{ and } \quad U_2(x) = 6 \tanh \Big( \frac{x}{\sqrt{2}} \Big) + 4 \sinh \Big( \frac{x}{\sqrt{2}} \Big) \cosh \Big( \frac{x}{\sqrt{2}} \Big) + \frac{3 \sqrt{2} x}{\cosh^2 \big( \frac{x}{\sqrt{2}} \big)}.
$$
Since the solution which we are looking for is odd, vanishes at the origin and tends to $0$ at $\pm \infty$, we conclude that
$$
\partial_{\lambda, \lambda} \Psi_0^*(x) = U_0(x) = - \frac{x}{2 \sqrt{2} \cosh^2 \big( \frac{x}{\sqrt 2} \big)},
$$
and we then check that
$$
6 \int_\R S_0 \, \partial_{\lambda, \lambda} \Psi_0^* \, |\chi_0|^2 = - \frac{3}{4} \int_\R |\chi_0|^4.
$$
In particular, we are led to
\begin{equation}
\label{iddd2}
6 \int_{\Omega_{d_k}} \Big( 2 |\chi_k|^4 + \langle S_0, \partial_{\lambda, \lambda} \Psi_0^* \rangle_\C \, |\chi_k|^2 \Big) = 3 d_k \int_\R \Big( 3 |\chi_0|^4 + 2 S_0 \, \partial_{\lambda, \lambda} \Psi_0^* \, |\chi_0|^2 \Big) = \frac{33 \, d_k}{4} \int_\R |\chi_0|^4.
\end{equation}

We next turn to the derivative $V := \partial_{\lambda,\lambda} W^*$. Similarly, we use the formulae in~\eqref{ddlambdapsi} and~\eqref{ddlambdaW} to derive from~\eqref{eq:proj-syst-1} that this derivative satisfies
\begin{equation}
\label{eq:dll-W^*}
\pi_k \Big( - \Delta V(x, y) - V(x, y) \big( 1 - |S_0(x)|^2 \big) + 2 \, \langle S_0(x), V(x, y) \rangle_\C \, S_0(x) + \chi_0(x)^2 \, S_0(x) \, \cos \Big( \frac{2 \pi k y}{d_k} \Big) \Big) = 0.
\end{equation}
In view of Proposition~\ref{prop:exist-proj-sol}, the derivative $V$ is in $\boH^2(\Omega_{d_k}, \C) \cap H_k$, while the map $(x, y) \mapsto \chi_0(x)^2 S_0(x) \cos(2 \pi k y/d_k)$ is in $\boL^2(\Omega_{d_k}, \C)$. Decomposing the operator $T$ (for $\psi_0 = S_0$) as in~\eqref{eq:dec-boT}, and invoking the invertibility properties of the operators $T_k$ resulting from Lemma~\ref{lem:invert-T-k}, we deduce that the derivative $V$ necessarily writes as
\begin{equation}
\label{eqdllw}
\partial_{\lambda,\lambda} W^*(x, y) = v(x) \, \cos \Big( \frac{2 \pi k y}{d_k} \Big),
\end{equation}
where the real-valued function $v$ is the unique solution in $H^2(\R)$ of the differential equation
$$
- v'' + 4 v - 3 \chi_0^2 v = - S_0 \, \chi_0^2.
$$
Note here that the Lax-Milgram theorem guarantees that this equation has a unique solution $v$ in $H^1(\R)$, which is in $H^2(\R)$ by standard elliptic theory. Note also that we can estimate this solution using the fact that
$$
0 \leq \int_\R |v|^2 \leq \int_\R \Big( |v'|^2 + 4 |v|^2 - 3 |\chi_0|^2 \, |v|^2 \Big) = - \int_\R S_0 \, v \, |\chi_0|^2,
$$
so that by the Cauchy-Schwarz inequality,
\begin{equation}
\label{eq:estim-int-v}
\int_\R |v|^2 \leq \int_\R S_0^2 \, |\chi_0|^4, \quad \text{ and } \quad 0 \leq - \int_\R S_0 \, v \, |\chi_0|^2 \leq \int_\R S_0^2 \, |\chi_0|^4 \leq \int_\R |\chi_0|^4 .
\end{equation}
In particular, going back to the expression of the derivative $\partial_{\lambda,\lambda} W^*$ in~\eqref{eqdllw}, we obtain
$$
6 \int_{\Omega_{d_k}} \langle S_0, \partial_{\lambda,\lambda} W^* \rangle_\C \, |\chi_k|^2 = 3 d_k \int_\R S_0 \, v \, |\chi_0|^2 \geq - 3 d_k \int_\R |\chi_0|^4.
$$
Combining this inequality with~\eqref{eq:iddd} and~\eqref{iddd2}, we are led to
$$
\partial_{\lambda, \lambda,\lambda} J(d_k, 0) = \omega \, d_k,
$$
with
$$
\omega := \frac{33}{4} \int_\R |\chi_0|^4 + 3 \int_\R S_0 \, v \, |\chi_0|^2 \geq \frac{21}{4} \int_\R |\chi_0|^4 > 0.
$$
This concludes the proof of Lemma~\ref{lem:diff-I}. \qed

\subsection{Expansion of the energy $E(\Psi_{k, d})$}
\label{sub:DL-E}

In view of~\eqref{def:Psi-d}, we can decompose the energy $E(\Psi_{k, d})$ as
\begin{equation}
\label{eq:exp-E-Psikd}
\begin{split}
E(\Psi_{k, d}) = & E \big( \Psi_0^{\bm{\lambda}(d)} \big) + \int_0^d \int_\R \Big\langle - \big( \Psi_0^{\bm{\lambda}(d)} \big)'' - \Psi_0^{\bm{\lambda}(d)} \big( 1 - |\Psi_0^{c}|^2 \big), \bm{\lambda}(d) \chi_k + W^{\bm{\lambda}(d)} \Big\rangle_\C \\
& + \int_0^d \int_\R \bigg( \frac{1}{2} \Big| \bm{\lambda}(d)\nabla \chi_k + \nabla W^{\bm{\lambda}(d)} \Big|^2 + \Big\langle \Psi_0^{\bm{\lambda}(d)}, \bm{\lambda}(d)\chi_k + W^{\bm{\lambda}(d)} \Big\rangle_\C^2 \\
& \quad \quad - \frac{1}{2} \Big( 1 - |\Psi_0^{\bm{\lambda}(d)}|^2 \Big) \Big| \bm{\lambda}(d)\chi_k + W^{\bm{\lambda}(d)} \Big|^2 + \frac{1}{4} \Big| \bm{\lambda}(d) \chi_k + W^{\bm{\lambda}(d)} \Big|^4 \\
& \quad \quad + \Big\langle \Psi_0^{\bm{\lambda}(d)}, \bm{\lambda}(d)\chi_k + W^{\bm{\lambda}(d)} \Big\rangle_\C \Big| \bm{\lambda}(d) \chi_k + W^{\bm{\lambda}(d)} \Big|^2 \bigg).
\end{split}
\end{equation}
Due to the property that $\psi_0^{\bm{\lambda}(d)}$ only depends on the $x$-variable, and the integrals of the functions $\chi_k$ and $W^{\bm{\lambda}(d)}$ with respect to the $y$-variable vanish, the integral in the first line of the previous decomposition is zero.

Concerning the energy quantity $E(\Psi_0^{\bm{\lambda}(d)})$, we use the fact that $S_0$ solves~\eqref{eq:GL} to write it as
$$
E \big( \Psi_0^{\bm{\lambda}(d)} \big) = E(S_0) + \frac{d}{2} \int_\R \Big( \big| (\Psi_0^{\bm{\lambda}(d)} - S_0)' \big|^2 - \big( 1 - S_0^2 \big) \big| \Psi_0^{\bm{\lambda}(d)} - S_0 \big|^2 + \frac{1}{2} \big( S_0^2 - |\Psi_0^{\bm{\lambda}(d)}|^2 \big)^2 \Big).
$$
We then deduce from the smoothness of the map $\lambda \mapsto \Psi_0^\lambda$ and~\eqref{ddlambdapsi} that
$$
\Psi_0^\lambda = S_0 + \frac{\lambda^2}{2} \partial_{\lambda, \lambda} \Psi_0^* + \lambda (d - d_k) \partial_{d, \lambda} \Psi_0^* + \frac{(d - d_k)^2}{2} \partial_{d, d} \Psi_0^* + \boO \Big( (d - d_k)^3 + |\lambda|^3 \Big),
$$
this expansion holding in $W_{0, \sigma}^{2, \infty}(\R)$ for $0 \leq \sigma \leq \sigma_0/2$. As a consequence of~\eqref{eq:equiv-lambda}, we obtain
\begin{equation}
\label{eq:DL-Psi0}
\Psi_0^{\bm{\lambda}(d)} = S_0 + \frac{\Lambda^2 (d - d_k)}{2 d_k} \partial_{\lambda, \lambda} \Psi_0^* + \boO \Big( (d - d_k)^\frac{3}{2} \Big),
\end{equation}
with $\Lambda = \sqrt{12 \sqrt{2}/\omega}$ as before. Relying on the smoothness of the map $\lambda \mapsto 1 - |\Psi_0^\lambda|^2$ in Proposition~\ref{prop:exist-proj-sol}, we similarly deduce that
\begin{equation}
\label{eq:DL-eta0}
1 - \big| \Psi_0^{\bm{\lambda}(d)} \big|^2 = 1 - S_0^2 - \frac{\Lambda^2 (d - d_k)}{d_k} \big\langle S_0, \partial_{\lambda \lambda} \Psi_0^* \big\rangle_\C + \boO \Big( (d - d_k)^\frac{3}{2} \Big),
\end{equation}
this asymptotics now holding in $L^2(\R)$. Inserting these identities in the previous expansion of the energy $E(\Psi_0^{\bm{\lambda}(d)})$, we are led to
\begin{align*}
E \big( \Psi_0^{\bm{\lambda}(d)} \big) = & E(S_0) + \frac{\Lambda^4 (d - d_k)^2}{8 d_k^2} \int_\R \Big( \big| (\partial_{\lambda \lambda} \Psi_0^*)' \big|^2 - \big( 1 - S_0^2 \big) \big| \partial_{\lambda \lambda} \Psi_0^* \big|^2 + 2 \big\langle S_0, \partial_{\lambda \lambda} \Psi_0^* \big\rangle_\C^2 \Big) \\
& + \boO \Big( (d - d_k)^\frac{5}{2} \Big).
\end{align*}
Going back to~\eqref{eq:dll-Psi0*}, we conclude that
$$
E \big( \Psi_0^{\bm{\lambda}(d)} \big) = E(S_0) - \frac{\Lambda^4 (d - d_k)^2}{8 d_k} \int_\R |\chi_0|^2 \,\big\langle S_0, \partial_{\lambda \lambda} \Psi_0^* \big\rangle_\C + \boO \Big( (d - d_k)^\frac{5}{2} \Big).
$$

We now deal with the integral in the second line of~\eqref{eq:exp-E-Psikd}, which we expand in the following four terms
$$
I_1 = \frac{\bm{\lambda}(d)^2}{2} \int_0^d \int_\R \Big( \big| \nabla \chi_k \big|^2 - \big( 1 - |\Psi_0^{\bm{\lambda}(d)}|^2 \big) \big| \chi_k \big|^2 + 2 \big\langle \Psi_0^{\bm{\lambda}(d)}, \chi_k \big\rangle_\C^2 \Big),
$$
\begin{align*}
I_2 = \bm{\lambda}(d) \int_0^d \int_\R \Big( \big\langle \nabla \chi_k, \nabla W^{\bm{\lambda}(d)} \big\rangle_\C - \big( 1 & - |\Psi_0^{\bm{\lambda}(d)}|^2 \big) \big\langle \chi_k, W^{\bm{\lambda}(d)} \big\rangle_\C \\
& + 2 \big\langle \Psi_0^{\bm{\lambda}(d)}, \chi_k \big\rangle_\C \big\langle \Psi_0^{\bm{\lambda}(d)}, W^{\bm{\lambda}(d)} \big\rangle_\C \Big),
\end{align*}
$$
I_3 = \frac{1}{2} \int_0^d \int_\R \Big( \big| \nabla W^{\bm{\lambda}(d)} \big|^2 - \big( 1 - |\Psi_0^{\bm{\lambda}(d)}|^2 \big) \big| W^{\bm{\lambda}(d)} \big|^2 + 2 \big\langle \Psi_0^{\bm{\lambda}(d)}, W^{\bm{\lambda}(d)} \big\rangle_\C^2 \Big),
$$
and
$$
I_4 = \int_0^d \int_\R \bigg( \Big\langle \Psi_0^{\bm{\lambda}(d)}, \bm{\lambda}(d) \chi_k + W^{\bm{\lambda}(d)} \Big\rangle_\C \Big| \bm{\lambda}(d) \chi_k + W^{\bm{\lambda}(d)} \Big|^2 + \frac{1}{4} \Big| \bm{\lambda}(d) \chi_k + W^{\bm{\lambda}(d)} \Big|^4 \bigg).
$$
Concerning the integral $I_1$, we deduce from~\eqref{def:chi-k} that
$$
I_1 = \frac{\bm{\lambda}(d)^2 \, d}{4} \int_\R \Big( \big| \chi_0' \big|^2 + \frac{\pi^2 k^2}{d^2} |\chi_0|^2 - \big( 1 - |\Psi_0^{\bm{\lambda}(d)}|^2 \big) \big| \chi_0 \big|^2 + 2 \big\langle \Psi_0^{\bm{\lambda}(d)}, i \chi_0 \big\rangle_\C^2 \Big),
$$
Combining the fact that $\chi_0$ is an eigenvector of the operator $L_0^-$ for the eigenvalue $-1/2$ with the expression of $d_k$ given by~\eqref{def:d-k}, we obtain
$$
I_1 = \frac{\bm{\lambda}(d)^2 \, d}{4} \int_\R \bigg( \Big( \frac{\pi^2 k^2}{d^2} - \frac{\pi^2 k^2}{d_k^2} \Big) |\chi_0|^2 - \big( S_0^2 - |\Psi_0^{\bm{\lambda}(d)}|^2 \big) \big| \chi_0 \big|^2 + 2 \big\langle \Psi_0^{\bm{\lambda}(d)}, i \chi_0 \big\rangle_\C^2 \bigg),
$$
In view of~\eqref{eq:equiv-lambda},~\eqref{eq:DL-Psi0} and~\eqref{eq:DL-eta0}, and since $\langle S_0, i \chi_0 \rangle_\C = 0$, we are led to
$$
I_1 = \frac{\Lambda^2 (d - d_k)^2}{4 d_k} \int_\R \Big( - 1 + \Lambda^2 \, \big\langle S_0, \partial_{\lambda \lambda} \Psi_0^* \big\rangle_\C \Big) \, \big| \chi_0 \big|^2 + \boO \Big( (d - d_k)^\frac{5}{2} \Big).
$$

In order to estimate the integrals $I_k$ for $2 \leq k \leq 4$, we next expand the map $W^{\bm{\lambda}(d)}$ with respect to $d- d_k$. More precisely, we consider as before the map $\tilde{W}^{\bm{\lambda}(d)}(x, z) = W^{\bm{\lambda}(d)}(x, d z/d_k)$. Combining Proposition~\ref{prop:exist-proj-sol} and the smoothness of the previous change of variables, we check that the map $(d, \lambda) \mapsto \tilde{W}^\lambda$ is smooth with values in $W_{0, \sigma}^{2, \infty}(\Omega_{d_k})$ for $0 \leq \sigma \leq \sigma_0/2$. As a consequence of the fact that $\tilde{W}^0 = 0$ and~\eqref{ddlambdaW}, we can expand it as
$$
\tilde{W}^\lambda = \frac{\lambda^2}{2} \partial_{\lambda, \lambda} \tilde{W}^* + \lambda (d - d_k) \partial_{d, \lambda} \tilde{W}^* + \frac{(d - d_k)^2}{2} \partial_{d, d} \tilde{W}^* + \boO \Big( (d - d_k)^3 + |\lambda|^3 \Big),
$$
this expansion holding in $W_{0, \sigma}^{2, \infty}(\Omega_{d_k})$. In view of~\eqref{eq:equiv-lambda}, this gives
\begin{equation}
\label{eq:DL-W}
\tilde{W}^{\bm{\lambda}(d)} = \frac{\Lambda^2 (d - d_k)}{2 d_k} \partial_{\lambda, \lambda} \tilde{W}^* + \boO \Big( (d - d_k)^\frac{3}{2} \Big).
\end{equation}
Applying the change of variables $y = d z/d_k$ to the integral $I_4$, and using~\eqref{eq:DL-Psi0} and~\eqref{eq:DL-W}, we are first led to
\begin{align*}
I_4 = \frac{\bm{\lambda}(d)^3 \, d}{d_k} \int_0^{d_k} \int_\R \langle S_0, \tilde{\chi}_k \rangle_\C |\tilde{\chi}_k|^2 & + \frac{\bm{\lambda}(d)^4 \, d}{4 d_k} \int_0^{d_k} \int_\R \Big( 2 \langle S_0, \partial_{\lambda, \lambda} \tilde{W}^* \rangle_\C |\tilde{\chi}_k|^2 \\
& + 4 \langle S_0, \tilde{\chi}_k \rangle_\C \, \langle \tilde{\chi}_k, \partial_{\lambda, \lambda} \tilde{W}^* \rangle_\C + |\tilde{\chi}_k|^4 \Big) + \boO \Big( (d - d_k)^\frac{5}{2} \Big).
\end{align*}
At this stage, recall that $\langle S_0, \tilde{\chi}_k \rangle_\C = 0$, and also that we can write the function $\partial_{\lambda, \lambda} \tilde{W}^*$ as
$$
\partial_{\lambda, \lambda} \tilde{W}^*(x, y) = v(x) \, \cos \Big( \frac{2 \pi k y}{d_k} \Big).
$$
In view of~\eqref{def:chi-k} and~\eqref{eq:equiv-lambda}, we therefore obtain
\begin{align*}
I_4 & = \frac{\Lambda^4 \, (d - d_k)^2}{4 d_k^2} \int_0^{d_k} \int_\R \Big( 2 \langle S_0, \partial_{\lambda, \lambda} \tilde{W}^* \rangle_\C |\tilde{\chi}_k|^2 + |\tilde{\chi}_k|^4 \Big) + \boO \Big( (d - d_k)^\frac{5}{2} \Big) \\
& = \frac{\Lambda^4 \, (d - d_k)^2}{32 d_k} \int_\R \Big( 4 S_0 \, v + 3 |\chi_0|^2 \Big) \, \big| \chi_0 \big|^2 + \boO \Big( (d - d_k)^\frac{5}{2} \Big).
\end{align*}
We argue similarly for the integral $I_3$ for which we deduce from~\eqref{eq:DL-Psi0},~\eqref{eq:DL-eta0} and~\eqref{eq:DL-W} that
$$
I_3 = \frac{\Lambda^4 \, (d - d_k)^2}{8 d_k^2} \int_0^{d_k} \int_\R \Big( \big| \nabla \partial_{\lambda, \lambda} \tilde{W}^* \big|^2 - \big( 1 - S_0^2 \big) \big| \partial_{\lambda, \lambda} \tilde{W}^* \big|^2 + 2 \big\langle S_0, \partial_{\lambda, \lambda} \tilde{W}^* \big\rangle_\C^2 \Big) + \boO \Big( (d - d_k)^\frac{5}{2} \Big),
$$
so that, by~\eqref{eqdllw} and~\eqref{eq:dll-W^*},
$$
I_3 = \frac{\Lambda^4 \, (d - d_k)^2}{16 d_k} \int_\R S_0 \, v \, |\chi_0|^2 + \boO \Big( (d - d_k)^\frac{5}{2} \Big).
$$

We next turn to the integral $I_2$. Integrating by parts and invoking the fact that $\chi_0$ is an eigenvector of the operator $L_0^-$ for the eigenvalue $-1/2$, we first obtain
\begin{align*}
I_2 = \bm{\lambda}(d) \int_0^d \int_\R \bigg( \Big( \frac{\pi^2 k^2}{d^2} - \frac{1}{2} \Big) \big\langle \chi_k, W^{\bm{\lambda}(d)} \big\rangle_\C - \big( S_0^2 & - |\Psi_0^{\bm{\lambda}(d)}|^2 \big) \big\langle \chi_k, W^{\bm{\lambda}(d)} \big\rangle_\C \\
& + 2 \big\langle \Psi_0^{\bm{\lambda}(d)}, \chi_k \big\rangle_\C \big\langle \Psi_0^{\bm{\lambda}(d)}, W^{\bm{\lambda}(d)} \big\rangle_\C \bigg)
\end{align*}
Using~\eqref{def:d-k}, and arguing as for the integral $I_4$, we then check that
$$
I_2 = \boO \Big( (d - d_k)^\frac{5}{2} \Big).
$$

Collecting the previous estimates of the integrals $I_k$, and of the energy $E(\Psi_0^{\bm{\lambda}(d)})$, we deduce from~\eqref{eq:exp-E-Psikd} that
\begin{align*}
E(\Psi_{k, d}) = E(S_0) & - \frac{\Lambda^2 (d - d_k)^2}{4 d_k} \int_\R \big| \chi_0 \big|^2 \\
+ & \frac{\Lambda^4 (d - d_k)^2}{32 d_k} \int_\R \Big( 4 \big\langle S_0, \partial_{\lambda \lambda} \Psi_0^* \big\rangle_\C + 6 S_0 \, v + 3 |\chi_0|^2 \Big) \big| \chi_0 \big|^2 \bigg) + \boO \Big( (d - d_k)^\frac{5}{2} \Big).
\end{align*}
Recall at this stage that
$$
\int_\R \chi_0^2 = 2 \sqrt{2}, \quad \text{ and } \quad \int_\R \big\langle S_0, \partial_{\lambda \lambda} \Psi_0^* \big\rangle_\C \, \big| \chi_0 \big|^2 = - \frac{1}{8} \int_\R \chi_0^4,
$$
while
$$
\Lambda^2 = \frac{12 \sqrt{2}}{\omega} = \frac{16 \sqrt{2}}{\int_\R \big( 11 |\chi_0|^4 + 4 S_0 \, v \, |\chi_0|^2 \big)}.
$$
This provides the final expansion
$$
E(\Psi_{k, d}) = E(S_0) - \frac{(d - d_k)^2}{d_k} \, \boE + \boO \Big( (d - d_k)^\frac{5}{2} \Big),
$$
with
$$
\boE := \frac{\Lambda^2}{\sqrt{2}} \bigg( 1 - \frac{\int_\R \big( 5 |\chi_0|^4 + 12 S_0 \, v \, |\chi_0|^2 \big)}{2 \int_\R \big( 11 |\chi_0|^4 + 4 S_0 \, v \, |\chi_0|^2 \big) } \bigg).
$$
In order to check that this number is positive, we go back to~\eqref{eq:estim-int-v} so as to write
$$
\frac{\int_\R \big( 5 |\chi_0|^4 + 12 S_0 \, v \, |\chi_0|^2 \big)}{2 \int_\R \big( 11 |\chi_0|^4 + 4 S_0 \, v \, |\chi_0|^2 \big) } \leq \frac{5 \int_\R |\chi_0|^4}{14 \int_\R |\chi_0|^4} = \frac{5}{14},
$$
and we conclude that
$$
\boE \geq \frac{9 \Lambda^2}{14 \sqrt{2}} > 0.
$$
This completes the proof of~\eqref{eq:DL-E}. \qed

\begin{merci}
P.~Gravejat is supported by CY Initiative of Excellence (Grant ``Investissements d'Avenir'' ANR-16-IDEX-0008). The mathematical motivation for this work emerged from the CNRS project 80|Prime, Tradisq1d, Transport and dissipation in one dimensional quantum systems.
\end{merci}

\bibliographystyle{plain}
\bibliography{Bibliogr}

\end{document}